\documentclass[twoside,11pt]{article}
\usepackage{authblk}
\usepackage[hmarginratio=1:1]{geometry}
\usepackage{amssymb,amsmath,amsthm,bbm}
\usepackage{bm,mathabx,mathtools}
\usepackage{color}
\usepackage{graphicx}
\usepackage{enumerate}
\graphicspath{{Figs}}
\newcommand{\mres}{%
  \,\raisebox{-.127ex}{\reflectbox{\rotatebox[origin=br]{-90}{$\lnot$}}}\,%
}

\usepackage{setspace}
\newif\ifspace
\spacefalse  

\newcommand{\nwc}{\newcommand}
\nwc{\N}{\mathbb{N}}
\nwc{\R}{\mathbb{R}}
\nwc{\E}{\mathbb{E}}
\nwc{\D}{\partial}
\nwc{\calC}{\mathcal{C}}
\DeclareMathOperator{\co}{co}

\DeclareMathOperator{\dist}{dist}

\nwc{\one}{{\mathbbm{1}}}
\nwc{\eps}{\varepsilon}
\nwc{\inv}{^{-1}}
\nwc{\etal}{{\it et al.\ }}
\nwc{\id}{{\rm id}}

\newcommand{\vp}{\varphi}

\nwc{\intr}[1]{{\kern0pt#1^\mathrm{o}}}%

\nwc{\mdot}{\!\cdot\!}

\nwc{\ip}[2]{(#1)\cdot(#2) }
\nwc{\red}[1]{\textcolor{red}{[#1]}}
\nwc{\blue}[1]{\textcolor{blue}{#1}}
\nwc{\point}{\medskip $\bullet$\ }
\nwc{\tripl}{v_\star}
\nwc{\tstar}{t_\star}
\nwc{\Triang}{\triangle}
\nwc{\sst}{\sigma_\star}

\nwc{\calA}{{\mathcal A}}
\nwc{\calB}{{\mathcal B}}
\nwc{\calD}{{\mathcal D}}
\nwc{\calG}{{\mathcal G}}
\nwc{\calL}{{\mathcal L}}
\nwc{\calM}{{\mathcal M}}
\nwc{\calN}{{\mathcal N}}
\nwc{\calQ}{{\mathcal Q}}
\nwc{\calR}{{\mathcal R}}
\nwc{\calS}{{\mathcal S}}
\nwc{\calX}{{\mathcal X}}
\nwc{\calZ}{{\mathcal Z}}
\nwc{\leb}{{\mathcal L}^d}

\newcommand{\pts}{\psi_t^{*}}
\newcommand{\ptss}{\psi_t^{**}}

\newcommand{\fss}{f^{**}}
\newcommand{\MA}{Monge-Amp\`ere}
\newcommand{\calH}{{\mathcal H}}
\nwc{\mam}{\kappa}
\nwc{\wkto}{\xrightharpoonup{\star}}
\nwc{\dto}{\xrightharpoonup{\mathcal{D}}}
\nwc{\tX}{{\tilde X}}

\newcommand{\rhoac}{\rho_t^{\rm ac}}
\newcommand{\rhosg}{\rho_t^{\rm sg}}
\newcommand{\mamac}{\mam_t^{\rm ac}}

\newcommand{\xinv}{S_t^{\rm in}}
\newcommand{\xsg}{S_t^{\rm sg}}
\newcommand{\xndf}{S_t^{\rm nd}}

\numberwithin{equation}{section}
\newtheorem{theorem}{Theorem}[section]
\newtheorem{lemma}[theorem]{Lemma}
\newtheorem{prop}[theorem]{Proposition}
\newtheorem{cor}[theorem]{Corollary}
\theoremstyle{remark}
\newtheorem{remark}[theorem]{Remark}
\newtheorem{definition}[theorem]{Definition}
\usepackage[pagebackref]{hyperref}

\begin{document}
\title{Analysis of the adhesion model and the reconstruction problem in cosmology} 
\author[1]{Jian-Guo Liu\footnote{Email address: Jian-Guo.Liu@duke.edu} }
\author[2]{Robert L. Pego\footnote{Email address: rpego@cmu.edu} }
\affil[1]{Departments of Mathematics and Physics, 
Duke University,\linebreak Durham, NC 27708, USA}
\affil[2]{Department of Mathematical Sciences,
Carnegie Mellon University,\linebreak Pittsburgh, PA 15213, USA}
\date{December 15, 2025}

\maketitle
\begin{abstract}
In cosmology, a basic explanation of the observed concentration of mass 
in singular structures is provided by the Zeldovich approximation, 
which takes the form of free-streaming flow
for perturbations of a uniform Einstein-de Sitter universe in co-moving coordinates.
The adhesion model suppresses multi-streaming by introducing viscosity. 
We study mass flow in this model by analysis of Lagrangian advection 
in the zero-viscosity limit. Under mild conditions, we show that a unique limiting
Lagrangian semi-flow exists. Limiting particle paths stick together after collision
and are characterized uniquely by a differential inclusion.
The absolutely continuous part of the mass measure agrees with that of
a \MA\ measure arising by convexification of the free-streaming velocity potential.
But the singular parts of these measures can differ when flows along singular
structures merge, as shown by analysis of a 2D Riemann problem.
The use of \MA\ measures and optimal transport theory for the reconstruction of
inverse Lagrangian maps in cosmology was introduced in work of Brenier \&
Frisch~\etal (Month. Not. Roy. Ast. Soc. 346, 2003).
In a neighborhood of merging singular structures in our examples, however, 
we show that reconstruction yielding a monotone Lagrangian map cannot be exact a.e.,
even off of the singularities themselves.
\end{abstract}



\vfil
\pagebreak
\tableofcontents

\ifspace
\onehalfspacing
\fi

\section{Introduction}

\subsection{Background} 
The Zeldovich approximation and the adhesion model are of basic importance in cosmology 
for understanding the early development of large-scale structures in the universe. 
They serve to roughly explain how mass in the cosmos may
have reached its present distribution, which is highly heterogeneous even on scales
that are huge compared to typical distances between neighboring galaxies.  
They do this by describing the growth of perturbations of a uniformly expanding universe, 
evolving approximately in accord with the theory of general relativity.

The simplicity of the equations that comprise these approximate models of mass flow 
disguises some subtlety in how they should be interpreted physically. 
For they are cast in terms of variables which make it appear that gravity is neglected, although it is not.
In the {Zeldovich approximation}, matter streams freely
along linear paths, moving with constant velocity
with respect to coordinates in a Euclidean space. 
As long as the flow remains smooth,
the velocity $v(x,t)$ satisfies the multi-dimensional inviscid Burgers system
\begin{equation}\label{e:iBurgers}
 \D_t v + v\cdot\nabla v = 0 \,, \qquad x\in\R^d,\  t>0.
\end{equation}
As argued by Shandarin and Zeldovich~\cite{shandarin1989large},
solutions of this system generally develop ``pancake'' singularities
which induce mass concentrations that resemble what one infers from observations.
One should recognize, however, that 
the time variable here is not physical time, and the velocity here
is a scaled perturbation of the velocity in a uniformly expanding
Einstein-de Sitter universe, described in a Newtonian approximation by an Euler-Poisson system.
For the convenience of readers we include a sketch of the derivation 
in Appendix~\ref{a:derivation}. For further details we refer the reader 
to the book of Peebles~\cite{peebles1980} and works of 
Brenier~\etal\,\cite[Appendix A]{brenier2003}
and Gurbatov~\etal\,\cite{gurbatov2012large}. 
As pointed out in these references, the initial velocity perturbation 
$v_0$ obtained as $t\to0^+$ is naturally the gradient of a potential.

The Lagrangian flow map induced by the initial velocity perturbation $v_0$ 
typically loses injectivity
in the Zeldovich approximation, generating ``multi-streaming''
regions in space. The {adhesion model} aims to avoid this
by introducing an artificial viscosity, requiring that the velocity $v$ 
arises as a limit of fields $v^\eps$ that satisfy
the multi-dimensional Burgers system
\begin{equation}\label{e:Burgers}
 \D_t v^\eps + v^\eps\cdot\nabla v^\eps = \tfrac12\eps \Delta v^\eps \,.
\end{equation} 
See \cite{gurbatov2012large} for an extensive discussion of the adhesion model.
We take the viscosity coefficient to be $\frac12\eps$ instead of $\eps$
to simplify formulas later.
With the natural gradient initial velocity $v_0 = \nabla \vp$, 
the solution to \eqref{e:Burgers} will be the 
gradient $v^\eps=\nabla u^\eps$ of a solution of the potential Burgers equation 
\begin{equation}\label{e:vHJ}
\D_t u^\eps + \tfrac12|\nabla u^\eps|^2 = \tfrac12\eps \Delta u^\eps \,, 
\end{equation}
with initial data $u^\eps = \vp$ at $t=0$.
Mathematically one expects that in the limit as $\eps\to0^+$, the 
limiting potential $u$ will be the viscosity solution of
the equation obtained with $\eps=0$. 
This limit is given by the well-known Hopf-Lax formula
\begin{equation}\label{e:HLOmega1}
u(x,t) = \inf_y \frac{|x-y|^2}{2t} + \vp(y).
\end{equation}

The physical mass in cosmology is transported by the adhesion-model velocity $v$
only in a Lagrangian approximation that also ignores multi-streaming.  
In this approximation, however, and as long as it remains smooth,
mass density $\rho$ is governed by the continuity equation
\begin{equation}\label{e:etat}
\D_t \rho + \nabla\cdot(\rho v) = 0, 
\end{equation}
starting from a uniform density $\rho_0$ at $t=0$.

It has long been recognized in the physical literature 
that in more than one space dimension, the physical momentum is not conserved
under the dynamics of the adhesion model \cite{Shandarin1997,brenier2003,gurbatov2012large}.
In one space dimension, however, momentum is conserved, and 
the adhesion model reduces to so-called ``sticky particle''
dynamics of the kind that has by now been studied in great detail 
mathematically --- see \cite{ERS96,BrenierGrenier1998,NatileSavare2009} and references in the 
expository article \cite{Hynd2019}.

In one-dimensional sticky particle dynamics,  delta-mass concentrations develop as particle paths collide.
The mass distribution at times $t>0$ becomes a measure that can be
determined by a variational formula, 
cf.~\cite{ERS96,BrenierGrenier1998,tadmor2011variational}.
For initially uniform mass distribution with density $\rho_0=1$,
the mass distribution measure $\rho_t$ for any $t>0$ is the second distributional derivative of 
a convex function $w(\cdot,t)$ defined by 
\begin{equation}\label{d:w1}
    w(x,t) := \sup_y \left( x\mdot y - \tfrac12 |y|^2 -t \vp(y) \right).
\end{equation}
This corresponds to a formula given by Brenier \& Grenier \cite[Eq.~(25)]{BrenierGrenier1998} 
for cases when total mass is finite.  

From the Hopf-Lax formula \eqref{e:HLOmega1} we see that 
\begin{equation}\label{e:uw_dual}
 w(x,t) = \tfrac12|x|^2 -  t u(x,t)\,.
\end{equation}
If we use this same equation to define $w$ in more than one space dimension,
then it is well-known that {\em as long as the fields remain smooth}, 
the density is related to the convex function $w$ by the \MA\ equation
\begin{equation}\label{e:MA1}
    \rho(x,t) = \det \nabla^2 w(x,t)\,,
\end{equation}
where $\nabla^2w$ denotes the Hessian matrix.  
The reason is that due to \eqref{e:iBurgers},
$v=\nabla u$ is constant along straight flow lines 
\begin{equation}
x=X(y,t)=y+t\nabla\vp(y) = \nabla\left(\tfrac12|y|^2+t\vp(y)\right),
\end{equation}
so the inverse Lagrangian map is given by 
\begin{equation}
Y(x,t) = \nabla w(x,t) = x - t\nabla u(x,t) \,,
\end{equation}
and the density $\rho$ is the Jacobian $\det\nabla Y=\det\nabla^2 w$, due to the change-of-variable formula
\begin{equation}\label{e:ch_vars}
\int_B \rho(x,t)\,dx = \int_B \det \nabla Y(x,t)\,dx = \int_{Y(B,t)}\,dy
\end{equation}
valid for any region $B$ (Borel measurable).

Determining the Lagrangian flow map and its inverse is one of the main tasks
involved in the {\em reconstruction problem} in cosmology: 
Given the present distribution of cold dark matter as inferred from observations, 
one seeks to reconstruct the position and velocity of the matter in the early universe 
which arrives at any given position $x$ at the present epoch $t$.
As the present distribution contains mass concentrated into singular structures such as
(point-like) clusters, (curve-like) filaments and sheets, the inverse Lagrangian map
is bound to be a multi-valued or set-valued map.

Frisch~\etal\,\cite{frisch2002reconstruction} and  Brenier~\etal\,\cite{brenier2003} 
have addressed the reconstruction problem by making use of 
{\em optimal transport theory} (OT) as a framework for determining
the convex potential $w$ and the inverse Lagrangian map $Y$.
For a present mass distribution with singular concentrations, 
the potential $w$ determined by OT is convex but not smooth.
In this case, the authors of \cite{brenier2003} effectively assume the inverse 
of the Lagrangian flow map $X_t:=X(\cdot,t)$
is the {\em subgradient} of $w_t:=w(\cdot,t)$, 
defined in convex analysis by the supporting plane condition:
\begin{equation}\label{d:Dwt}
\D w_t(x) = \{ y\in\R^d: w_t(z)\ge w_t(x)+y\cdot(z-x) \ \text{for all $z$}\}.
\end{equation}
This assumption, that $X_t\inv=\D w_t$,
is equivalent to a presumption made by the authors of \cite{vergassola1994burgers}
to the effect that $X_t$ 
is the same as the transport map $T_t$ defined through convexification 
of the free streaming potential $\psi_t$, via
\begin{equation}\label{e:Lagrange1}
T_t(y)=\nabla\ptss(y),
\quad 
\text{where}\quad \psi_t(y)=\frac12|y|^2+t\vp(y).
\end{equation}
Here $\ptss$ is the convexification of $\psi_t$,
its second Legendre transform.  Note that equation \eqref{d:w1}
states that the Legendre transform of $\psi_t$ is $w_t = \pts$.

The subgradient allows to generalize \eqref{e:ch_vars} by use of the measure
$\mam_t$ defined by 
\begin{equation}\label{d:MAint}
    \int_B d\mam_t(x) = \int_{\D w_t(B)}\,dy = |\D w_t(B)|,
\end{equation}
where $|S|$ denotes the Lebesgue measure of the measurable set $S$.
This measure $\mam_t$ is known as the {\em \MA\ measure} of $w_t$ (see~\cite{Figalli-MongeAmpere}),
and by analogy with \eqref{e:MA1} and \eqref{e:ch_vars}, one says that $w_t$ is an ``Alexandrov solution''
of the \MA\ equation
\[
\mam_t = \det \nabla^2 w\,.
\]
Implicitly, the authors of \cite{vergassola1994burgers} and \cite{brenier2003} 
take the present mass distribution $\rho_t$ to agree with the \MA\ measure $\mam_t$.
This is known to be correct for smooth mass distributions, and also
for one-dimensional sticky particle dynamics with uniform initial density.

Other authors, including Weinberg \& Gunn \cite{weinberg1990largescale} for example,
have regarded the mass distribution $\rho_t$ in the adhesion model as the zero-viscosity 
limit of density fields advected by the velocity field satisfying
the multi-dimensional Burgers equation \eqref{e:Burgers}. 
Our notation will conform with this choice.
In particular, for small $\eps>0$, 
one writes out the solution $v^\eps=\nabla u^\eps$ of \eqref{e:Burgers}
by use of the Cole-Hopf transformation, then solves the continuity equation
\begin{equation}\label{e:rhoeps}
    \D_t\rho^\eps + \nabla\cdot(\rho^\eps v^\eps) = 0, 
\end{equation}
with initially uniform density $\rho_0$.
In \cite{weinberg1990largescale}
this is done by solving numerically for the Lagrangian flow map determined by
\begin{equation}\label{e:Xeps_flow}
    \D_t X^\eps(y,t) = v^\eps(X^\eps(y,t),t), \quad X^\eps(y,0)=y.
\end{equation}

\subsection{Present results}
The principal aim of the present paper is to describe, with mathematical rigor,
how the limiting mass measure $\rho_t$ in the adhesion model is determined,
and characterize its structure.
This we do through study of the Lagrangian flow map and its inverse, 
the convex potential $w_t$ and its \MA\ measure $\mam_t$, and the transport map $T_t$
which is the gradient of the convexified potential $\ptss$.
We carry out our investigation with an eye towards understanding 
how these quantities relate to the reconstruction problem.

We defer precise statements of results to the various sections to follow.
For all of our main results, we assume the initial velocity potential $\vp$
is Lipschitz, meaning initial velocities are bounded.
For most results we have found it convenient to also assume $\vp$ is {\em semi-concave},
meaning $y\mapsto \vp(y)-\tfrac12\lambda|y|^2$ is concave for some $\lambda\ge0.$
Semi-concavity is a rather mild condition (true if second derivatives are bounded, say)
but has the effect of preventing the appearance of totally void regions
in the mass distribution. Thus we leave the study of pure voids aside for future work.

Under these assumptions, we show that the Lagrangian flows $X^\eps$
in the adhesion model converge as $\eps\to0$ to a well-defined 
limit $X$ that is locally Lipschitz, and $y\mapsto X(y,t)$ is surjective for each $t$.
Particle paths $t\mapsto X(y,t)$ may collide, but when they do then they 
{\em stick together afterwards}.  Thus 
the name ``adhesion model'' is justified mathematically.

Moreover, these particle paths are characterized, even after they impinge upon discontinuities
in the velocity field $\nabla u$, in terms of a 
{\em differential inclusion} of the form
\begin{equation}\label{z:Du}
\D_t z(t) \in \D u_t(z(t)) \quad\text{for a.e. $t>0$.}
\end{equation}
Here the differential $\D u_t(x)$ is a set that can be related
to the subgradient of the convex function $w_t=w(\cdot,t)$ 
from \eqref{e:uw_dual} through the formula
\begin{equation}\label{e:Dut1}
\D u_t(x) = \frac{x - \D w_t(x)}t\,.
\end{equation}
Each particle path $t\mapsto z(t)=X(y,t)$ turns out to be the {\em unique} Lipschitz solution of the
differential inclusion \eqref{z:Du} with initial data $z(0)=y$.

The smooth mass distributions $\rho^\eps_t=\rho^\eps(\cdot,t)$ converge in a weak-star sense
to the limit measure $\rho_t$ determined as the {\em pushforward} 
of uniform Lebesgue measure $\leb$ under the limiting Lagrangian map $X_t = X(\cdot,t)$.
This pushforward is denoted by $(X_t)_\sharp\leb$.
The part of $\rho_t$ that is singular with respect to Lebesgue measure is obtained
by pushforward from the set where the matrix $\nabla X_t(y)$ (defined for a.e.~$y$)
is singular.

We obtain additional information concerning the Lagrangian maps $X_t$ and mass measures 
$\rho_t$ from study of the convex functions $w_t$, their associated convexified 
transport maps $T_t$, the \MA\ measures $\mam_t$, and smoothed versions of these objects.

In general, we are able to prove that the measures $\rho_t$ and $\mam_t$
can differ {\em only} in their parts that are singular with respect to Lebesgue
measure. The absolutely continuous parts of these measures {\em must agree},
and the \MA\ equation \eqref{e:MA1} correctly determines their density a.e., 
provided  $\nabla^2 w$ is taken to be the Hessian in the sense of Alexandrov's theorem
for convex functions (see~\cite[p.~242]{EvansGariepy}).
Moreover, for any $t>0$, almost every $x\in\R^d$ is the image of 
a single $y$, depending continuously upon $x$, for which $x=X_t(y)$
and the past history of the particle path arriving at $x$ is one of free streaming, with
\begin{equation}\label{e:straight1}
X_s(y)=T_s(y)=y+s\nabla\vp(y)\,,\quad\text{for $0\le s\le t$.}
\end{equation}
This shows that indeed the adhesion model not only avoids the multi-streaming
phenomenon in kinetic theory (different particles at the same point having different
velocities), but multiple particle paths can currently coincide at the same point $x$
only on a set of zero Lebesgue measure (including inside singular concentrations of mass, e.g.).

Another way to view the \MA\ measure $\mam_t$ has to do with the well-known fact that it is
impossible to determine the past history of particle paths presently inside shocks.
For fixed $t>0$, by defining a modified initial potential $\breve\vp$ so that
\[
\ptss(y)=\tfrac12|y|^2+t\breve\vp(y),
\]
the Hopf-Lax formula provides  a ``collision-free'' velocity potential 
$\breve u_s(x)$ which is $C^1$ for $0\le s<t$ and satisfies $\breve u_t=u_t$.
The modified Lagrangian particle paths $\breve X_s(y)$ are {\em all} straight lines for 
$0\le s< t$ and can collide only for $s\ge t$. Moreover,
as we show in Section~\ref{ss:shock-free},  it turns out  that
\[
\breve X_t(y) = T_t(y) \quad\text{for all $y\in\R^d$,}
\]
and the straight-line paths described above in \eqref{e:straight1} are matched. 
But now the mass measure $\breve\rho_t$ pushed forward by the 
modified Lagrangian flow need not agree with $\rho_t$, 
but instead matches the \MA\ measure, satisfying
\begin{equation}
    \breve \rho_t = (\breve X_t)_\sharp\leb = \mam_t\,.
\end{equation}

Perhaps our most significant finding is that
the advected mass measures $\rho_t=\lim_{\eps\to0}\rho^\eps_t$
arising in the zero-viscosity limit
can indeed {\em differ} from the \MA\ measures $\mam_t$
when singular mass concentrations are present.
This is demonstrated in Section~\ref{s:example} through explicit consideration 
of cases with $d=2$ space dimensions 
when the initial velocity field $\nabla\vp$ 
is constant in each of three sectors of the plane and $\vp$ is concave.
This is a kind of Riemann problem for which 
the velocity discontinuity along each of three rays
generates a shock in the velocity field and a concentration of mass along a filament.  
The singular parts of the mass measure $\rho_t$ and the 
\MA\ measure $\mam_t$ differ if and only if a simple geometric condition
holds: Namely, the triangle with vertices at the three velocity vectors
should not contain its circumcenter. In this case, incoming mass flows along two
of the filaments merge into a flow {\em outgoing} along the third.

For the examples in which $\rho_t\ne\mam_t$, 
we show that a reconstruction strategy as proposed 
in \cite{frisch2002reconstruction} and \cite{brenier2003}
{\em need not produce correct results}, even outside the singular mass concentrations,
for determining the original position $y$ of mass at current position $x=X_t(y)$. 
For from the current distribution of mass $\rho_t$, use of 
optimal transport theory should provide a convex potential $\tilde \psi_t$
with the property that the transport map 
$\tilde T_t=\nabla\tilde\psi_t$ pushes forward
Lebesgue measure to the \MA\ measure 
\[
\tilde\mam_t =(\tilde T_t)_\sharp\leb =  \rho_t \,,
\]
at least in a local sense.
In our examples, we can show that if $\tilde T_t$
is any a.e.-defined map that has this pushforward property 
and also correctly provides unique pre-images
\[
y=X_t\inv(x) = \tilde T_t\inv(x)
\quad\text{for a.e. $x$,  }
\]
then $\tilde T_t$ {\em cannot} be a monotone map.
Hence it cannot be the gradient of a convex function. 

\subsection{Related literature}

Of the many papers in the cosmological literature on the reconstruction problem,
we mention only a few to demonstrate the role and continuing relevance of least action principles, 
the Zeldovich approximation and the adhesion model.  
In 1989, Peebles~\cite{peebles1989} used a least action principle 
for an $N$-body Newtonian approximation for perturbations of an Einstein-de Sitter cosmology
to reconstruct a discrete set of galaxy orbits.
Narayanan \& Croft \cite{NarayananCroft1999} tested the performance of 
a half-dozen variant methods for reconstruction of the  density field 
generated by $N$-body simulations of Vlasov-Poisson equations.
This included methods based on the Zeldovich approximation and a method of 
Croft \& Gazta\~naga \cite{CroftGatanaga1997} based on  finding 
a monotone map to reconstruct the Lagrangian flow map.
Frisch~\etal\,\cite{frisch2002reconstruction} and  Brenier~\etal\,\cite{brenier2003} made explicit
the connection to optimal transport theory and \MA\ equations, and tested 
optimal transport reconstructions against $N$-body simulations of a $\Lambda$CDM cosmology.
More recently, L\'evy~\etal\,\cite{Levy2021fast} performed extensive tests with 
optimal transport reconstructions generated by fast modern numerical methods.
The \MA\ equation also figures in several studies of reconstruction in which 
it serves to approximate or replace the Poisson equation for the gravitational potential,
see Brenier~\cite{Brenier2011modified} and L\'evy~\etal\,\cite{LevyBrenier2024}.

As we have indicated, the adhesion model  in one space dimension reduces to sticky particle flow, 
governed by pressureless Euler equations for conservation of mass and momentum.  
Mathematical studies of multi-dimensional sticky particle flows that conserve mass and momentum
suggest that the concept may not be a well-formulated one.
E.g., Bressan \& Nguyen \cite{BressanNguyen2014} have shown that
multi-dimensional pressureless Euler flow is ill-posed for measure-valued density fields.
Concerning existence of solutions, there is recent work by Cavalletti~\etal\,\cite{CavallettiEA2019}.
Existence and uniqueness of solutions for a restricted class of initial data is obtained by
Bianchini \& Daneri~\cite{BianchiniDaneri2023}.
But even the classic Riemann problem for these equations is 
problematic---Using convex integration, 
Huang~\etal\,\cite{HuangShiWang2025-nonuniqueRiemann} have shown
that in two space dimensions, infinitely many
solutions exist that satisfy a natural entropy/energy inequality.

The analysis of the adhesion model itself is simpler than that of the pressureless Euler system,
due to its decoupling of mass advection from the evolution of the velocity field. 
The mathematical literature on mass transport and Lagrangian flows for given 
discontinuous velocity fields is too large to review here---we will only discuss works 
that are most relevant for the present analysis of measures advected in the adhesion model.
As argued by Poupaud \& Rascle~\cite{PoupaudRascle1997}, we find it
natural to study mass measures by pushforward under Lagrangian flows generated 
from discontinuous velocity fields and satisfying differential inclusions 
in accord with the theory of Filippov~\cite{Filippov1960,Filippov1988book}.
Such flows were shown by Filippov to be unique for velocity fields satisfying 
appropriate one-sided Lipschitz estimates. 
Poupaud \& Rascle used such estimates to establish approximation results for 
Lagrangian flows and associated pushforward measures. 

One-sided Lipschitz estimates naturally arise from semi-concavity and 
have been used to study singularity propagation for viscosity solutions 
of Hamilton-Jacobi equations~\cite{CannarsaSinestrari04}.
We will exploit the fact that semi-concavity 
is propagated by solutions of the potential Burgers equation \eqref{e:vHJ} 
to obtain one-sided Lipschitz estimates that imply compactness of the smooth flows $X^\eps$.
To prove convergence as $\eps\to0$, we provide an uncomplicated argument that makes 
use of a basic principle related to large deviations, 
convergence results for subgradients in convex analysis \cite{Rockafellar,Hiriart2001},
and a convexity argument in the theory of differential inclusions~\cite{Aubin1984}.

Also directly relevant to our study are results of Bianchini \& Gloyer~\cite{BianchiniGloyer2011} for 
Lagrangian flows satisfying differential inclusions with semi-monotone velocity maps.
These authors show that these flows satisfy a differential {\em equation} for
some single-valued velocity field, and they characterize finite pushforward measures as unique 
non-negative solutions of the associated continuity equation by using results of
Ambrosio \& Crippa \cite{AmbrosioCrippa2008}.

\subsection{Discussion}
The results of the present paper explain how the
mass distribution in the adhesion model is determined in the limit $\eps\to0$, 
showing that it arises by pushforward under a limiting Lagrangian flow
whose particle paths are ``sticky'' and are determined by solving
a differential inclusion. 
Moreover,  the entire past history of the flow backward from almost every 
point in space (outside a singular concentration set)
is one of free streaming at constant velocity, 
as in the Zeldovitch approximation.

A rough explanation for how it can happen that $\rho_t\ne\mam_t$ 
in the examples of Section~\ref{s:example} is the following.
The dynamics generates one-dimensional filaments that concentrate mass
along three rays that intersect at a point.
For some choices of the velocities, the mass along all three
filaments flows {\em toward} the intersection point, and concentrates
there into a delta mass.
For other choices, however, mass along one of the filaments flows 
{\em away} from the intersection point. 

In this latter situation,  the limiting mass measure $\rho_t$ from the adhesion model 
carries mass along the incoming filaments immediately onto 
the outgoing one, and develops no delta-mass concentration  at the intersection point. 
Regardless of the configuration of velocities, however, the \MA\ measure
$\mam_t$ in \eqref{d:MAint} always concentrates positive mass into a delta mass
at the intersection point. The reason is that by explicit calculation,
the slopes in the subgradient $\D w_t(x)$ at the intersection point 
comprise a triangle of nonzero area. 
The measures $\mam_t$ and $\rho_t$ thus differ at the intersection point 
and on the filament carrying mass away from it.

In this situation,
the limiting Lagrangian flow $X_t$ retains its ``sticky'' nature,
which it always must, while the convexified transport maps $T_t$ lose
``stickiness,'' allowing paths $t\mapsto T_t(y)$ to coincide for a time
at the point of concentration, then later separate along the extruded filament.
Reconstruction based on optimal transport fails around the point of concentration 
because the monotonicity property in a neighborhood {\em outside} the
region that is mapped to the concentration set determines the map {\em inside},
for our examples.

An issue that we have not resolved is whether the singular set 
where mass is concentrated  (the ``cosmic web'' within the adhesion model) 
needs to be the same for $\rho_t$ as that for $\mam_t$ in general. 
These sets do agree in our examples, but we do not know whether this is 
necessarily always the case.

Neither measure $\rho_t$ nor $\mam_t$ provided by the adhesion model
exactly captures the physical mass  distribution generated from the 
Newtonian approximation  of a perturbed Einstein-de Sitter universe, of course. 
Whether the inability to correctly reconstruct initial position monotonically
proves to be an important drawback to using such a reconstruction procedure
for the adhesion model remains to be determined. After all, 
the adhesion model has other well-known limitations such as 
failure to conserve momentum. 

Perhaps, though, our study may provoke some interest in phenomena related to 
the collision of mass sheets and filaments that extrude singular structures,
such as we have identified in our examples.

\subsection{Plan of the paper}
In Section~\ref{s:lagrangian}, we establish stability estimates
for the Lagrangian flows $X^\eps$, prove convergence as $\eps\to0$
to the solution of the differential inclusion~\ref{z:Du}, 
We study the zero-viscosity limit and Lebesgue decomposition 
of pushforward mass measures in Section~\ref{s:mass}.
In Section~\ref{s:smoothMA} we analyze a smooth approximation 
to the \MA\ measures $\mam_t$ and the transport maps $T_t$.
Then in Section~\ref{s:MAmeasures} we develop several properties
of $\mam_t$ and $T_t$ and use them to characterize
backward particle paths $X_t(y)$ from almost every $x$, the 
agreement of the absolutely continuous parts of $\rho_t$ and $\mam_t$,
and the fact that these parts are determined a.e.\ by the \MA\ equation~\eqref{e:MA1}.

Finally in Section~\ref{s:example} we study a special case
of initial data in $d=2$ space dimensions with constant velocity in three sectors,
obtaining criteria that characterize when $\rho_t\ne\mam_t$.
Section~\ref{s:example} can be read largely independently of 
the earlier sections~\ref{s:mass}--\ref{s:MAmeasures}. 
It depends upon the definition of $\mam_t$ in terms of the subgradient of $w_t$,
and the characterization of the Lagrangian flow $X$ in Theorem~\ref{t:UNIQ} 
via the unique solution of a differential inclusion as in \eqref{z:Du} for which 
the velocity potential $u_t$ is concave.

\section{Lagrangian flow in the adhesion model}\label{s:lagrangian}

In the adhesion model, the velocity field is determined as the zero-viscosity
limit of a gradient field $v^\eps = \nabla u^\eps$, where the potential 
$u^\eps$ satisfies the potential Burgers equation 
\begin{equation}\label{e:potB}
    \D_t u^\eps + \frac12|\nabla u^\eps|^2 = \frac{\eps}{2}\Delta u^\eps\,,
    \qquad u^\eps(x,0)=\vp(x).
\end{equation}
According to the Cole-Hopf transformation, the function $f=e^{-u^\eps/\eps}$
satisfies the heat equation 
\begin{equation}\label{e:heat}
\D_t f = \frac\eps2\Delta f.
\end{equation}
Presuming $\vp$ is subquadratic at infinity (i.e., $\vp(y)=o(|y|^2)$), 
the function $u^\eps_t=u^\eps(\cdot,t)$ is given by
\begin{equation}\label{e:uepsCH}
u^\eps_t(x) = -\eps \log \frac{1}{(2\pi\eps t)^{d/2}}\int_{\R^d} 
\exp\left(-\frac{|x-y|^2}{2\eps t} - \frac{\vp(y)}\eps\right)\,dy \,,
\end{equation}
for all $x\in\R^d$ and $t\in(0,\infty)$. 
Equivalently we can write \eqref{e:uepsCH} in the form
\begin{equation}\label{e:uepsCH2}
   u^\eps_t(x) = -\eps \log \int_{\R^d} 
   G(y,\eps t) e^{-\vp(x-y)/\eps}\,dy\,,
\end{equation}
where $G$ is the heat kernel for \eqref{e:heat} with $\eps=1$: 
\begin{equation}\label{e:heatkernel}
G(y,\tau) = \frac{1}{(2\pi\tau )^{d/2}} \exp\left(\frac{-|y|^2}{2\tau}\right).
\end{equation}
Actually, if $\vp(y) = -\frac12\lambda|y|^2 + o(|y|^2)$ with $\lambda>0$
then \eqref{e:uepsCH} provides
a smooth solution for $0<t<1/\lambda$, a fact that we will make some use of.

\subsection{Symmetries}\label{ss:symmetry}
The Hopf-Lax formula~\ref{e:HLOmega1} and the potential Burgers equation 
\eqref{e:potB} admit a couple of symmetries that simplify
our analysis. 

First we have the Galilean transformations,
rotating spatial coordinates or changing variables 
to a frame moving with constant velocity $\tripl$. E.g., defining
\[
\hat x = x-t\tripl\,, \quad \hat\vp(y) = \vp(y)-\tripl\cdot y\,,
\]
the corresponding functions $\hat u_t(\hat x)$ and $\hat w_t(\hat x)$
coming from the Hopf-Lax formula \eqref{e:HLOmega1} and \eqref{e:uw_dual}
are easily verified to be given by 
\begin{align}\label{e:galilean}
 \hat u_t(\hat x) = u_t(x) - \tripl\cdot x +\tfrac12|\tripl|^2t\,,
 \qquad \hat w_t(\hat x) = w_t(x)\,,
\end{align}
so that $\hat v_t(\hat x)=v_t(x)-\tripl$.
The same formula for $\hat u$, decorated by $\eps$, provides a function satisfying 
the potential Burgers equation \eqref{e:potB}. 

A second symmetry is more particularly valid just for the 
Hopf-Lax formula and the potential Burgers equation: As one can readily check,
given $\lambda>0$ and some function 
$\hat u(x,t)$ given by the Hopf-Lax formula \eqref{e:HLOmega1} 
for all $x\in\R^d$ and $t\in(0,1/\lambda)$ with initial data $\hat \vp$,
the function  given by  
\begin{equation}\label{e:symm2}
u(x,t) = \hat u\left(\frac{x}{1+\lambda t},\frac{t}{1+\lambda t}\right) +
\frac{\lambda}{1+\lambda t}\frac{|x|^2}2 ,
\end{equation}
satisfies \eqref{e:HLOmega1} again for $x\in\R^d$ and $t\in(0,\infty)$, 
with $\vp(z) = \hat\vp(z)+\frac\lambda2|z|^2$.

%

Further, given 
$\hat u^\eps(x,t)$ satisfying the potential Burgers equation \eqref{e:potB} 
for $x\in\R^d$ and $t\in(0,1/\lambda)$ and initial data $\hat\vp$, 
the function defined by 
\begin{equation}\label{e:symm2b}
u^\eps(x,t) = \hat u^\eps\left(\frac{x}{1+\lambda t},\frac{t}{1+\lambda t}\right) +
\frac{\lambda}{1+\lambda t}\frac{|x|^2}2 - \frac{\eps d}2  \log(1+\lambda t)
\end{equation}
satisfies \eqref{e:potB} again for $x\in\R^d$ and $t\in(0,\infty)$
with initial data $\vp$.

\subsection{Semi-concavity and stability estimates for Lagrangian flows}\label{ss:stableLagrange}

The Lagrangian flows generated by the adhesion model enjoy 
some good stability properties under a simple and mild assumption 
on the velocity potential $\vp$. 
We say $\vp$  is {\em $\lambda$-concave} if $\lambda\in\R$ and the function 
$y\mapsto \vp(y)-\frac\lambda2|y|^2$ is concave.
The function $\vp$ is {\em semi-concave} if it is $\lambda$-concave for some
$\lambda\in\R$.   

For example, any function $\vp$ which is $C^2$ with globally bounded 
first and second derivatives is semi-concave.
Semi-concavity is an important concept in the analysis of
Hamilton-Jacobi equations (as evidenced by the book \cite{CannarsaSinestrari04})
but we will need only some basic facts that especially concern how
semi-concavity propagates forward under the dynamics of 
the potential Burgers equation or the Hopf-Lax formula.

\begin{lemma}\label{l:lambda}
    Assume that $\vp:\R^d\to\R$ is subquadratic at infinity and 
    is $\lambda$-concave for some $\lambda\ge0$.
    Then with $\lambda_t = \lambda/(1+\lambda t)$, for any $t>0$ we have:
    \begin{itemize}
    \item[(i)] The function $u^\eps_t(x)$ given by the Cole-Hopf formula \eqref{e:uepsCH} 
    is $\lambda_t$-concave.
      \item[(ii)] The function $u_t(x)$ given by the Hopf-Lax formula \eqref{e:HLOmega1}
        is $\lambda_t$-concave.
    \end{itemize}
\end{lemma}

\begin{proof}
    Treating (i) first, consider the case $\lambda=0$. 
    Supposing $\vp$ is simply concave,
    the function $x\mapsto e^{-\vp(x-y)/\eps}$ is log-convex. As sums and  positive multiples 
    and  limits  of log-convex functions are log-convex, 
    it follows from \eqref{e:uepsCH2} that $-u^\eps_t$ is convex,
    i.e., $u^\eps_t$ is concave. 

   Now, suppose $\lambda>0$ and $\vp$ is $\lambda$-concave.
   Setting $\hat\vp(z)=\vp(z)-\frac\lambda2|z|^2$, we see $\hat\vp$ is concave.
   Then $\hat u_t^\eps(x)$ can be defined as in \eqref{e:uepsCH} for all $x\in\R^d$
   and $t\in(0,1/\lambda)$, and from the first part of the proof 
   it follows $\hat u_t^\eps$ is concave. Using the symmetry \eqref{e:symm2b} 
   from subsection~\ref{ss:symmetry} to define $u^\eps(x,t)$, it follows $u_t^\eps$ is 
   defined for all $t>0$ and is $\lambda_t$-concave.
   This establishes part (i).

    For part (ii) in case $\lambda=0$, we note that any concave $\vp$  is the infimum of 
    some family of affine functions $\{v_\alpha\mdot z+h_\alpha\}_\alpha$.
    Then in the Hopf-Lax formula we can interchange the inf over $z$
    and the inf over $\alpha$, calculate to find a min at $z=x-tv_\alpha$,
    and get
   \begin{align*}
        u_t(x) &= 
       \inf_\alpha \inf_z \left( \frac{|x-z|^2}{2t} + v_\alpha\mdot z+h_\alpha\right)
      = \inf_\alpha \left(v_\alpha\mdot x -\frac{|tv_\alpha|^2}{2t} + h_\alpha\right).
   \end{align*} 
   This shows that $u_t$ is concave.
   In case $\lambda>0$, the proof that $u_t$ is $\lambda_t$-concave is similar using \eqref{e:symm2}. 
   Or, one can apply the proof of Proposition~\ref{p:fprop}(ii) to the 
   $(1+\lambda t)$-concave function $f=\psi_t$
   in \eqref{e:Lagrange1} to show $w_t=\pts$ is strictly convex, with 
   \begin{equation}\label{e:wt-betat}
w_t(x)= \frac1{1+\lambda t} \frac{|x|^2}2 + g_t(x) \,,
\end{equation}
   where $g_t$ is convex. Then from \eqref{e:uw_dual} we infer $u_t$ is $\lambda_t$-concave.
\end{proof}

Any semi-concave function is locally Lipschitz on its domain, since the same
is true for convex functions. The gradient, defined almost everywhere,
satisfies a one-sided Lipschitz condition:

\begin{lemma}\label{lem:oneside}
    Suppose $f:\R^d\to\R$ is $\lambda$-concave for some $\lambda\ge0$.
    Then
    \begin{itemize}
        \item[(i)]
    $(\nabla f(y)-\nabla f(z))\mdot(y-z) \le \lambda|y-z|^2$
    if $f$ is differentiable at $y,z\in\R^d$. 
        \item[(ii)] No eigenvalue of the Hessian
        $\nabla^2 f(z)$ (when it exists) is greater than $\lambda$.
    \end{itemize}
\end{lemma}
\begin{proof}
    The function given by $\hat f(y)=f(y)-\frac\lambda2|y|^2$ is concave
    and thus its gradient has the well-known monotonicity property (easily
    proved using supporting planes)
    \[
    (\nabla \hat f(y)-\nabla \hat f(z))\mdot(y-z) \le 0.
    \]
    This implies (i) by a simple substitution.
    Part (ii) follows from (i).
\end{proof}

This leads directly to the following propagating stability estimates for
the Lagrangian flow determined by the smooth velocity field 
$v^\eps=\nabla u^\eps$ with $\eps>0$:

\begin{prop}\label{p:Xepslip}
    Let $\vp$ be $\lambda$-concave with $\lambda\ge0$ and subquadratic at infinity,
    and let  $u^\eps$ be the solution of \eqref{e:potB} from \eqref{e:uepsCH}.
    Let $X^\eps_t(y)=X^\eps(y,t)$ be the Lagrangian flow map satisfying
    \eqref{e:Xeps_flow}.  Then for all $y,z\in\R^d$ and whenever $0<s<t$,
\begin{equation}\label{e:lambda-contraction}
\frac{|X^\eps_t(y)-X^\eps_t(z)|}{1+\lambda t} \le 
\frac{|X^\eps_s(y)-X^\eps_s(z)|}{1+\lambda s} \le 
|y-z| \,.
\end{equation}
\end{prop}
\begin{proof} 
Let $X_1=X^\eps_t(y)$, $X_2=X^\eps_t(z)$. 
By applying Lemma~\ref{lem:oneside} with $f=u^\eps_t$ 
along with Lemma~\ref{l:lambda}(ii) we deduce that 
\begin{align*}
\frac12
\frac{d}{dt}
|X_1-X_2|^2 &=(\nabla u^\eps_t(X_1)-\nabla u^\eps_t(X_2))\mdot (X_1-X_2)
\le
\frac{\lambda}{1+\lambda t} |X_1-X_2|^2.
\end{align*}
Provided $y\ne z$, the right-hand side can never vanish and we infer that
\[
\frac{d}{dt} \log \frac{|X_1-X_2|}{1+\lambda t} \le 0.
\qedhere
\]
\end{proof}

These stability estimates will allow us to justify the name ``adhesion model.''
Particle paths that become coincident 
in the limit as $\eps\to0$ at some time $s>0$ must remain coincident at all later times.

The conclusions of Proposition~\ref{p:Xepslip} ensure that the functions $(y,t)\mapsto X^\eps_t(y)$
are uniformly Lipschitz in $y$, locally for $t\ge0$.  To obtain a corresponding result in $t$,
it will be convenient to suppose the velocity potential $\vp$ is itself Lipschitz,
meaning initial velocities are bounded.
Given a constant $K\ge0$, a function $f$ on $\R^d$ is called $K$-Lipschitz if  
$|f(y)-f(z)|\le K|y-z|$ for all $y,z$. 

\begin{prop}\label{lem:Klip}
    Assume $\vp$ is $K$-Lipschitz. Then $u^\eps_t$ is $K$-Lipschitz 
    for each $\eps>0$ and $t>0$, whence $|v^\eps_t(x)|\le K$ for all $x$ and 
    $t\mapsto X^\eps_t(y)$ is $K$-Lipschitz, with 
    \begin{equation}\label{bd:DXt}
    |\D_t X^\eps_t(y)|\le K \quad\text{for all $y\in\R^d$, $t>0$.}
    \end{equation}
\end{prop}
\begin{proof}
   Let $x,z\in \R^d$. Observe that for all $y$,
   \[
   e^{-\vp(x-y)/\eps} =  
   e^{-\vp(z-y)/\eps}   e^{(\vp(z-y)-\vp(x-y))/\eps} 
   \le 
   e^{-\vp(z-y)/\eps}   e^{K|x-z|/\eps}
   \]
   Using this in \eqref{e:uepsCH2} we find
   \[
   -u^\eps_t(x) \le \eps\log \int_{\R^d}
   G(y,\eps t)  e^{-\vp(z-y)/\eps}\,dy\,   e^{K|x-z|/\eps}  = 
   -u^\eps_t(z)+ K|x-z|\,.
   \]
   After interchanging $x$ and $z$ we obtain the claimed result.
\end{proof}

\subsection{Zero-viscosity limit and differential inclusion}\label{ss:Xlim}

The uniform local Lipschitz estimates of the previous subsection 
allow one to extract a local uniform limit of the smoothed Lagrangian flows 
$X^\eps$  along a subsequence of any sequence $\eps_j\to0$. 
Our goal in this subsection is to prove that actually these limits are unique, 
and the proof allows us to characterize every limiting Lagrangian path $t\mapsto X_t(y)$ 
as the unique Lipschitz solution of an initial value problem for a
{differential inclusion}. This initial value problem takes the form
\begin{equation} \label{e:DXt}
    \D_t x_t \in \D u_t(x_t) \, \quad \text{for a.e. $t>0$,}
\qquad    x_0 = y.
\end{equation}
The differential $\D u_t$ is a set-valued map that will be
well defined at every $x\in\R^d$ under the conditions of Lemma~\ref{l:lambda}, 
which ensure that $u_t$ is semi-concave. 

Recall that for any $\lambda$-concave function $f$,
the function $f(x)=\frac12\lambda|x|^2-\check f(x)$ where $\check f$ is convex.
In terms of the subgradient $\D\check f(x)$, we define the differential
$\D f(x)=\lambda x - \D\check f(x)$.  Equivalently,
\begin{equation}\label{d:Dflambda}
    \D f(x) = \{q\in\R^d: f(z)\le f(x) + q\mdot(z-x)+\tfrac12\lambda|z-x|^2 \ \text{for all $z$}\}\,.
\end{equation}
This is the set of slopes at $x$ of paraboloids that lie above the graph of $f$, 
are tangent to it at $x$, and have Hessian $\lambda I$.

Our main theorem regarding the convergence of smoothed Lagrangian flows 
is the following.

\begin{theorem}\label{t:UNIQ}
Assume $\vp$ is $K$-Lipschitz, and $\lambda$-concave with $\lambda\ge0$. 
Then:
\begin{itemize}
    \item[(a)]  The limit 
\ \ $X(y,t) = \lim_{\eps\to0} X^\eps(y,t) \quad\text{exists for each $y\in\R^d$, $t\ge0$}$,
with uniform convergence on every compact subset of $\R^d\times[0,\infty)$.
\item[(b)] For each $t\ge0$, $X_t=X(\cdot,t)$ maps $\R^d$ surjectively onto $\R^d$. 
\item[(c)]   The function $X$ is locally Lipschitz on $\R^d\times[0,\infty)$, and more precisely:
   \begin{itemize}
       \item[(i)]  For each $y\in\R^d$ 
       the Lagrangian path $t\mapsto X_t(y)$ is $K$-Lipschitz.
       \item[(ii)] Whenever $0\le s<t$, for all $y,z\in\R^d$ we have the 
       stability estimate
\begin{equation}\label{e:lambda-contraction0}
\frac{|X_t(y)-X_t(z)|}{1+\lambda t} \le 
\frac{|X_s(y)-X_s(z)|}{1+\lambda s} \le 
|y-z| \,.
\end{equation}
If $X_s(y)=X_s(z)$ for some $s\ge0$, then
$X_t(y)=X_t(z)$ for all $t\ge s$.
   \end{itemize}
\item[(d)] For each $y\in \R^d$ the map $t\mapsto X_t(y)$
is the unique Lipschitz solution to the initial value problem \eqref{e:DXt}.
\end{itemize}
\end{theorem}
The characterization of the limit in terms of the differential inclusion in 
\eqref{e:DXt} is particularly useful for several purposes in this paper, including: 
for proving that the full limit in (a) exists; 
for comparing Lagrangian particle paths to convexified transport paths in Section~\ref{ss:consequences};
and for analyzing the examples in Section~\ref{s:example}.

Key to proving convergence for $X^\eps$ is that by Lemma~\ref{l:lambda},
the differentials $\D u_t$ satisfy a one-sided Lipschitz estimate
that extends Lemma~\ref{lem:oneside}(ii):
\begin{lemma}
    \label{lem:Dflambda}
    Assume $f\colon\R^d\to\R$ is $K$-Lipschitz, and $\lambda$-concave with $\lambda\ge0$.
    Then whenever $q\in \D f(x)$ and $\hat q\in \D f(\hat x)$,
    \[
    \text{(i)}\quad |q|\le K, \qquad\text{and \quad (ii)}
    \quad (q-\hat q)\mdot(x-\hat x)\le \lambda|x-\hat x|^2.
    \]
\end{lemma}
\begin{proof}
    Using the definition \eqref{d:Dflambda},
    for any $q\in \D f(x)$ we have 
    \begin{equation}\label{e:fzx-lambda}
    f(\hat x)-f(x)\le q\mdot(\hat x-x)+\tfrac12\lambda|\hat x-x|^2
    \end{equation}
    for all $\hat x$. Taking $\hat x=x- \alpha q$ for $\alpha>0$, 
    since $|f(\hat x)-f(x)|\le K|\hat x-x|$, upon rearranging terms we find 
    \begin{align*}
        q\mdot(x-\hat x)=\alpha|q|^2 \le K\alpha|q| + \tfrac12\lambda \alpha^2|q|^2.
    \end{align*}
    Cancelling $\alpha|q|$ and taking $\alpha\to0$ we find $|q|\le K$, proving (i).

    Supposing also that $\hat q\in\D f(\hat x)$, similar to \eqref{e:fzx-lambda} we find
    \[
    f(x)-f(\hat x) \le \hat q\mdot(x-\hat x)+\tfrac12\lambda|\hat x-x|^2.
    \]
    Adding this to \eqref{e:fzx-lambda} we obtain (ii).
\end{proof}

\begin{remark}
Though it will not matter in our application,
the definition \eqref{d:Dflambda} is independent of $\lambda$.
The reason is that whenever 
$f(z)=\tfrac12\tilde \lambda|z|^2 - \tilde f(z)$ with $\tilde f$ convex and $\tilde \lambda>\lambda$,
we have 
$\tilde f(z) = \check f(z) + \tfrac12(\tilde \lambda-\lambda)|z|^2$ 
and it is not difficult to show  that 
$\D\tilde f(x) = \D\check f(x) +(\tilde\lambda-\lambda)x,$ 
see~\cite[Prop.~A.1(ii)]{LPS19}.
\end{remark}

The use of one-sided Lipschitz estimates leads to 
a simple proof of uniqueness for Lipschitz solutions of \eqref{e:DXt}.
\begin{lemma}[Uniqueness]\label{l:uniq}
    Assume $\vp$ is $K$-Lipschitz and $\lambda$-concave with $\lambda\ge0$. 
    Then there is at most one Lipschitz solution to the initial value problem \eqref{e:DXt}.
\end{lemma}
\begin{proof}
  Necessarily $u_t$ is $\lambda_t$-concave for each $t\ge0$ by Lemma~\ref{l:lambda}.
    Supposing $x$ and $\hat x$ are both Lipschitz solutions of \eqref{e:DXt},
    then $t\mapsto |x_t-\hat x_t|^2$ is locally Lipschitz, and for a.e.~$t>0$,
    due to Lemma~\ref{lem:Dflambda}
    we have 
    \[
    \tfrac12\D_t |x_t-\hat x_t|^2 = (\D_t x_t - \D_t \hat x_t)\cdot(x_t-\hat x_t) \le \lambda_t |x_t-\hat x_t|^2\,.
    \]
  Since $x_0=\hat x_0$, upon integration and use of Gronwall's lemma we infer 
  that $x_t=\hat x_t$ for all $t\ge0$.  
\end{proof}

In order to establish Theorem~\ref{t:UNIQ}, we need to study the convergence of
$u_t^\eps$ to $u_t$ and $\nabla u_t^\eps$ to $\D u_t$ as $\eps\to0$.
Results concerning the convergence of the velocity potentials $u^\eps$ 
to the function $u$ given by the Hopf-Lax formula~\eqref{e:HLOmega1} are 
well-known in the theory of viscosity solutions for Hamilton-Jacobi equations.
But here we will make use of arguments based on 
convex analysis and an elementary case of
the Laplace principle (related to large deviations, see~\cite{DupuisEllis1997}),
which are rather uncomplicated and facilitate later comparison with convexified transport maps.

It will be convenient 
(also for later use in Section~\ref{s:MAmeasures}) to define the function 
\begin{equation}\label{e:wepsueps}
    w^\eps_t(x) := \tfrac12|x|^2 - tu^\eps_t(x) \,.
\end{equation}
This can be written in the form
\begin{equation}\label{e:weps2}
   w^\eps_t(x) = \eps t \log \int_{\R^d}  e^{(x\cdot y-\psi_t(y))/\eps t}\,dy  - \eps t \log(2\pi\eps t)^{d/2}\,,
\end{equation}
where $\psi_t(y) = \tfrac12|y|^2 + t\vp(y)$
as in \eqref{e:Lagrange1}.
The function $w^\eps_t$ is smooth, and is strictly convex because sums and limits of positive log-convex functions are log-convex.
Equation~\eqref{e:weps2} has the form of a ``soft'' Legendre transform. 
By comparison, the potential $w_t=w(\cdot,t)$ in \eqref{d:w1} 
is given by a standard Legendre transform as
\begin{equation}
    w_t(x) = \pts(x) = \sup_y \left( x\mdot y -\psi_t(y)\right) \,. 
\end{equation}

Via the Laplace principle, we obtain the following convergence result. 
\begin{lemma} \label{p:lim_wueps}
Assume $\vp\colon\R^d\to\R$ is $K$-Lipschitz. Then for each $x\in\R^d$ and $t>0$,
    \[
\lim_{\eps\to0} w_t^\eps(x) = w_t(x), \qquad 
\lim_{\eps\to0} u_t^\eps(x) = u_t(x). 
\]
Also, $u_t$ is $K$-Lipschitz, and the convergence is uniform in $x$ on each compact set in $\R^d$. 
\end{lemma}
\begin{proof}  Fixing $x\in\R^d$, write $f(y)=x\cdot y-\psi_t(y)$ and $1/p=\eps t$. 
Then the claimed convergence of $w^\eps_t(x)$ follows directly from the Laplace principle,
which here is equivalent to the statement that as $p\to\infty$, 
the log of the $L^p$ norm of $e^f$ converges to the log of its $L^\infty$ norm, i.e.,
\[
\lim_{p\to\infty} \frac1p \log \int_{\R^d} e^{pf(y)}\,dy  = \sup_y f(y) \,.
\]

    Moreover, the limit is uniform in compact sets as a consequence
    of the fact that if a sequence of convex functions converges pointwise,
    then it converges locally uniformly inside the (relative) interior of its 
    domain~\cite[Theorem~B.3.1.4]{Hiriart2001}. 
\end{proof}

Next we study the convergence of gradients to subgradients.

\begin{lemma}\label{p:wepslim}
Assume $\vp$ is Lipschitz. Let $t>0$.
Then at each point $x$ where $\nabla w_t(x)$ exists,  we have 
\begin{equation}\label{e:nablalim}
\lim_{\eps\to0} \nabla w^\eps_t(x) = \nabla w_t(x),
\qquad 
\lim_{\eps\to0} \nabla u^\eps_t(x) = \nabla u_t(x).
\end{equation}
Also, for $x$ arbitrary, given sequences $\eps_k\to 0$ and $x^k\to x$ 
as $k\to\infty$ we have 
\begin{equation}\label{e:subgrad_lim}
\dist\bigl(\nabla u^{\eps_k}_t(x^k), \D u_t(x)\bigr) \to 0 \quad\text{ as $k\to\infty$}.
\end{equation}
\end{lemma}
\begin{proof}
All three limit claims follow from the claim that for $x$ arbitrary,
\begin{equation}
    \label{e:wseqlim}
    \dist\bigl( \nabla w^{\eps_k}_t(x^k),\D w_t(x)\bigr) \to 0 
    \quad\text{as $k\to\infty$}.
\end{equation}
    But this is a direct consequence of the approximation property
    for subgradients established in Theorem~D.6.2.7 in \cite{Hiriart2001}.
\end{proof}

We now proceed to prove Theorem~\ref{t:UNIQ}. 

\begin{proof}[Proof of Theorem~\ref{t:UNIQ}]
1. {\em(Subsequential limits)} By the results of subsection~\ref{ss:stableLagrange}, 
the functions $X^\eps$ are Lipschitz on $\R^d\times[0,\tau]$ for each $\tau>0$, 
uniformly over $\eps>0$.  
By a standard subsequence selection argument using the Arzel\`a-Ascoli theorem,
we can find a subsequence $(\eps_k)$ of any given sequence $(\eps_j)$ 
converging to $0$ such that $X^{\eps_k}$
converges uniformly on every compact subset of $\R^d\times[0,\infty)$.
Any such limit $X$ is locally Lipschitz, and naturally satisfies 
(i) the $K$-Lipschitz condition with respect to $t$, due to Proposition~\ref{lem:Klip},
as well as (ii) the stability estimate \eqref{e:lambda-contraction0},
due to Proposition~\ref{p:Xepslip}.

2. {\em(Surjectivity)} Next we show that for each such limit, 
$X_t$ is surjective for each $t\ge0$. Let $x\in\R^d$.
For each $k$, since $X^{\eps_k}$ is a Lipschitz ODE flow, solving backward 
we can find $y_k$ such that $X^{\eps_k}_t(y_k)=x$.  By the velocity bound
in Proposition~\ref{lem:Klip} we have $|y_k-x|\le Kt$, so the sequence $(y_k)$ 
is bounded and a subsequence converges to some $y\in\R^d$. By 
the stability estimate it follows $x=X^{\eps_k}_t(y_k)$ must converge to 
$X_t(y)$ along the subsequence. Thus $x=X_t(y)$.

3. 
{\em(Differential inclusion)}  Now we show that 
for any such subsequential limit $X= \lim_{k\to\infty}X^{\eps_k}$,
the function $t\mapsto X_t(y)$ satisfies the differential inclusion in \eqref{e:DXt}.
Let $y\in\R^d$ and $\tau>0$ be arbitrary, and write $x_t = X_t(y)$.
Passing to a subsequence (denoted the same), we may suppose that for any 
$\tau>0$, the (bounded) derivatives $p^k_t = \D_t X^{\eps_k}_t(y)$ 
converge weakly in $L^1([0,\tau],\R^d)$
to a function $p$ (i.e., $t\mapsto p_t$).
Writing $x^k_t=X^{\eps_k}_t(y)$ and taking $k\to\infty$ in the formula
\[
\int_a^b p^k_t \,dt = x^k_b-x^k_a 
\]
for $0\le a<b\le \tau$ arbitrary, then $p_t = \D_t x_t$ follows.  
By Mazur's theorem \cite[p.~61]{Brezis2011},  
$p$ is a {\em strong limit} in $L^1([0,\tau],\R^d)$ 
of a sequence of convex combinations of elements in 
the sequence of derivatives $(p^k)_{k\ge n}$, for each $n$.
We can then extract a sequence $(q^n)_{n\ge 1}$ of such convex combinations
such that 
\[\int_0^\tau |q^n_t - p_t|\,dt \le \frac1n \,.\]
Passing to a subsequence (denoted the same), we can assume $q^n_t$ converges
as $n\to\infty$ to $p_t$ pointwise, for all $t$ in a set $I_\tau$ of full measure in $[0,\tau]$.

Let $\delta>0$ and let ${B(x,\delta)}$ 
denote the open ball with center $x$ and radius $\delta$. 
By Lemma~\ref{p:wepslim}, for each $t>0$ there exists $N_t$ such 
that for all $k\ge N_t$, $p^k_t=\nabla u^{\eps_k}_t(x^k_t)$ lies in the closed convex set 
$\D u_t(x_t)+\overline{B(0,\delta)}$. Hence for all $n\ge N_t$,
the convex combination $q^n_t$ lies in the same convex set,
which therefore must also contain $p_t$ for each $t\in I_\tau$.  
Since $\delta$ is arbitrary, it follows $p_t=\D_t x_t\in \D u_t(x_t)$ for a.e.~$t$.

4. In view of Lemma~\ref{l:uniq}, 
each subsequential limit $X_t(y)=\lim_{k\to\infty}X^{\eps_k}_t(y)$ must be the same.
Hence the full limit $X=\lim_{\eps\to0} X^\eps$ exists, with local uniform convergence.
This concludes the proof of Theorem~\ref{t:UNIQ}.
\end{proof}

The use of Mazur's theorem as above is 
a classical technique in the theory of differential inclusions, cf.~\cite[p.~60]{Aubin1984}.
It relies on the fact that the weak closure of any convex set in a Banach
space agrees with the strong closure.

\begin{remark}[Lagrangian semiflow]\label{r:semiflow}
Let $\{X_t\}_{t\ge0}$ be given by Theorem~\ref{t:UNIQ}, and let $0\le s\le t$.
By virtue of the facts that $X_s$ is surjective and by
the stability estimate \eqref{e:lambda-contraction0}, we can define a map
$X_{t,s}\colon \R^d\to\R^d$ by
\begin{equation}
    X_{t,s}(z) = X_t(y) \quad\text{whenever $z=X_s(y)$}.
\end{equation}
The map $X_{t,s}$ is surjective and Lipschitz with Lipschitz constant $(1+\lambda t)/(1+\lambda s)$.
The family of maps $\{X_{t,s}: 0\le s\le t\}$ 
determine a {\em Lagrangian semiflow} on $\R^d$, satisfying 
\[
X_{t,s} = X_{t,r}\circ X_{r,s} \quad\text{for $0\le s\le r\le t$}, \qquad X_{t,t}={\rm id}.
\]
\end{remark}

\begin{remark}\label{rem:outersemi}
The (set-valued) inverse $X_t\inv$ is outer semi-continuous, meaning that for all $x\in\R^d$,
\begin{equation}\label{e:outersemi}
\forall \eps>0 \  \exists\delta>0  \quad  X_t\inv(B(x,\delta)) \subset X_t\inv(x)+ B(0,\eps).
\end{equation}
The reason $X_t$ has this property is that it holds for any  vector function on $\R^d$
which is a continuous, bounded perturbation of the identity, by the lemma below. 
For the same reason, 
the property also holds with $X_t\inv$ replaced by $X_{t,s}\inv$ with $0\le s\le t$.
Subdifferentials of convex functions $f:\R^d\to\R$ always have this property; see \cite[Thm.~D.6.2.4]{Hiriart2001}.
\begin{lemma} Let $F:\R^d\to\R^d$ be continuous and assume $|F(y)-y|$ is uniformly bounded.
  Then $F\inv$ is outer semi-continuous, i.e., for all $x\in\R^d$, 
  \[
\forall \eps>0 \  \exists\delta>0 \quad  F\inv(B(x,\delta)) \subset F\inv(x)+ B(0,\eps).
  \]
\end{lemma}
\begin{proof}
   If the conclusion fails, then there exists $\eps>0$ and a sequence $y_k\in\R^d$ such that $x_k=F(y_k)\to x$
   as $k\to\infty$ but $\dist(y_k,F\inv(x))\ge\eps$.  As $|F(y)-y|$ is bounded, necessarily 
  $|y_k|=|y_k-F(y_k)+x_k|$ is uniformly bounded, so passing to a subsequence (denoted the same), $y_k$ converges
  to some $y\in\R^d$. Then $\dist(y,F\inv(x))\ge\eps$, but by continuity, $x_k=F(y_k)\to F(y)=x$.
  This contradiction establishes the result.
\end{proof}
\end{remark}

\section{Mass flow for the adhesion model}\label{s:mass}

When $\eps>0$, the mass density given by 
\begin{equation}\label{e:rhoepsdet}
\rho^\eps_t(x) = \left( \det \nabla X^\eps_t(y) \right)\inv \,,
\quad x=X^\eps_t(y),
\end{equation}
satisfies the continuity equation
\eqref{e:rhoeps} with uniform initial data $\rho^\eps_0(x)= 1$.
In the limit $\eps\to0$ we expect mass to concentrate on singular sets
where no formula analogous to \eqref{e:rhoepsdet} applies.
Studying this limit directly from the continuity equation \eqref{e:rhoeps}
is also problematic as the limiting velocity field is difficult
to define on singular sets.

Instead we will study the limit of the mass distribution
with density $\rho^\eps_t$  by 
using its characterization as the {\em pushforward} under the Lagrangian flow map $X^\eps_t$
of the  initial uniform mass distribution measure (Lebesgue measure $\leb$).
As is common, we also overload notation by writing $\rho^\eps_t$ to denote the 
mass distribution measure $f\leb$ with  density $f(x)=\rho^\eps_t(x)$ at time $t$.
Then for any Borel measurable function $g\colon\R^d\to\R$ of compact support, the
integral of $g$ with respect to $\rho^\eps_t$ is given by the change of variables formula
\begin{equation}\label{e:change_rhoeps}
\int_{\R^d} g(x)\,\rho^\eps_t(x)\,dx = \int_{\R^d} g(X^\eps(y,t))\,dy \,.
\end{equation}
In measure-theoretic notation, this pushforward is written
$\rho^\eps_t = (X^\eps_t)_\sharp \leb$.

\subsection{Zero-viscosity limit for mass flow}

Based upon our Theorem \ref{t:UNIQ} for convergence of flow maps, we obtain
the following convergence theorem for mass distributions in the adhesion model.
In the sequel we use a notion of convergence for (possibly infinite) Radon measures on $\R^d$
equivalent to convergence in the sense of distributions.  

\begin{definition}
Given a Radon measure $\mu$ and a family $(\mu^\eps)_{\eps>0}$ of such,
we say $\mu^\eps$ converges {\it locally weak-$\star$} to $\mu$
if for each continuous function $g\colon\R^d\to\R$ having compact support,
\[
\int_{\R^d} g(x)\,d\mu^\eps(x) 
\to 
\int_{\R^d} g(x)\,d\mu(x) 
\quad\text{as $\eps\to0$.} 
\]
\end{definition}

\begin{theorem}\label{t:rhodef}
Assume $\vp$ is $K$-Lipschitz, and $\lambda$-concave with $\lambda\ge0$. 
Let $X=\lim_{\eps\to0} X^\eps$ as given by Theorem~\ref{t:UNIQ}. 
For each $t\ge0$ define the Borel measure 
\[
\rho_t := (X_t)_\sharp\leb.
\]
Then 
\begin{itemize}
    \item[(i)] 
   $\rho^\eps_t$ converges locally weak-$\star$ to $\rho_t$ as $\eps\to0$,  for each $t\ge0$, and
   \item[(ii)] 
    $\rho_t$ converges locally weak-$\star$ to $\leb$ as $t\to0$.
\end{itemize}
\end{theorem}

\begin{proof} Let $g\colon\R^d\to\R$ be continuous with compact support.
Due to a standard change of variables formula for pushforward measures 
\cite[Prop.~1.7]{AmbrosioBrueSemola2021}
and the dominated convergence theorem,  as $\eps\to0$ we have
   \[
  \int_{\R^d} g(x)\rho^\eps_t(x)\,dx = \int_{\R^d} g(X^\eps_t(y))\,dy \to  
  \int_{\R^d} g(X_t(y))\,dy = \int_{\R^d} g(x)\,d\rho_t(x). 
   \]
   The convergence as $t\to0$ holds since $|X_t(y)-y|\le Kt$ and $g$ is uniformly
   continuous on its support.
\end{proof}

\begin{cor} Make the assumptions of the previous theorem, and let $t>0$.
Then 
        \begin{itemize}
        \item[(i)] $\rho_t^\eps(x)\ge (1+\lambda t)^{-d}$ for all $\eps>0$ and all $x\in\R^d$.
        \item[(ii)] $\rho_t(B)\ge (1+\lambda t)^{-d}|B|$ for all Borel sets $B\subset\R^d$.
    \end{itemize}
\end{cor}
\begin{proof}
By the Lipschitz bound in Proposition~\ref{p:Xepslip},
each eigenvalue of the matrix $\nabla X_t^\eps$ is no greater than 
$1+\lambda t$. Then the lower bound on $\rho_t^\eps$ follows from \eqref{e:rhoepsdet}.
The convergence result in Theorem~\ref{t:rhodef}(i) then
implies the lower bound for $\rho_t(B)$ through a standard approximation argument
using the regularity of Radon measures---see~\cite[p.~212]{Folland1999}.
\end{proof}

\subsection{Lebesgue decomposition of the mass distribution}

The structure of mass concentrations in the limiting mass distribution $\rho_t$ 
is related to differentiability properties of the Lagrangian flow map $X_t$
in a way that we describe here. 
Throughout this subsection
we assume $\vp$ is $K$-Lipschitz and $\lambda$-concave with $\lambda\ge0$.

For any $t>0$, the measure $\rho_t$ has a {Lebesgue decomposition} that 
we write 
\begin{equation}\label{e:lebrho}
\rho_t = \rhoac+\rhosg, \qquad \rhoac\ll\leb, \qquad \rhosg\perp\leb.
\end{equation}
The measure $\rhoac$ is absolutely continuous with respect to Lebesgue measure
$\leb$, and the measures $\rhosg $ and $\leb$ are mutually singular.

Since $X_t$ is Lipschitz by Theorem~\ref{t:UNIQ}, it is differentiable a.e.
Define the (Lagrangian) sets
\begin{align}
 \xinv &= \{y\in\R^d: \nabla X_t(y) \text{ exists and is invertible} \}, \\
 \xsg &=  \{y\in\R^d: \nabla X_t(y) \text{ exists and is singular} \}, \\
 \xndf &= \{y\in\R^d: \nabla X_t(y) \text{ does not exist} \}.
\end{align}
Below, $\mu\mres\calS$ denotes the restriction of a measure $\mu$ to a set $\calS$,
so $(\mu\mres\calS)(B)=\mu(\calS\cap B)$ for all $B$.

\begin{theorem}\label{t:rho_decomp}
    The mass measure $\rho_t=(X_t)_\sharp\leb$ has the Lebesgue decomposition 
    $\rho_t = \rhoac+\rhosg$ with
\begin{equation}\label{e:rho_decomp}
\rhoac = \rho_t\mres\calR_t \,, \qquad \rhosg = \rho_t\mres\calS_t\,,
\end{equation}
where $\calS_t=X_t(\xsg)$ has Lebesgue measure $|\calS_t|=0$, and 
$\calR_t =  \calS_t^c = \R^d\setminus\calS_t$.
\end{theorem}

\begin{remark}\label{r:Rt}
For the characterization of $\rhoac$ to come in Section~\ref{ss:consequences},
it is convenient to note that the decomposition of the Theorem
holds with $\calR_t$ taken as any subset $\calR_t\subset\calS_t^c$ with the property that $|\calR_t^c|=0$.
\end{remark}

The ``sticky'' property of the maps $X_t$ that was used in Remark~\ref{r:semiflow} 
to define the Lagrangian semiflow 
$\{X_{t,s}:0\le s\le t\}$ allows us to say that,
following the flow, the singular part can only increase, via mass concentrations accumulating
from the absolutely continuous part.
\begin{cor}\label{cor:masssemiflow}
Let $X_{t,s}$, $0\le s\le t$, be the Lagrangian semiflow maps defined in Remark~\ref{r:semiflow}.
Then  $\rho_t =  (X_{t,s})_\sharp \rho_s $. 
Furthermore, $(X_{t,s})_\sharp \rho_s^{\rm sg}\perp \leb$,
and 
in terms of the  measure $\rho_{t,s}=(X_{t,s})_\sharp\rho_s^{\rm ac}$
and its Lebesgue decomposition  $\rho_{t,s}^{\rm ac} + \rho_{t,s}^{\rm sg}$,
\[
\rhoac = 
\rho_{t,s}^{\rm ac} 
\,,\qquad
\rhosg = 
\rho_{t,s}^{\rm sg} 
+ (X_{t,s})_\sharp \rho_s^{\rm sg}\,.
\]
\end{cor} 
\begin{proof} By the semiflow property, $X_t = X_{t,s}\circ X_s$.
It is straightforward to infer that $X_t\inv = X_s\inv\circ X_{t,s}\inv$
as an identity of set mappings. Now if $B\subset\R^d$ is Borel,
\[
\rho_t(B) = |X_t\inv(B)|=|X_s\inv( X_{t,s}\inv(B))| = 
(X_{t,s})_\sharp\rho_s(B).
\]
Furthermore, the Lipschitz image $X_{t,s}(\calS_s)$ of $\calS_t$ has measure zero,
so 
\[
(X_{t,s})_\sharp\rho_s^{\rm sg}(B) = \rho_t(X_{t,s}\inv (B)\cap \calS_s) = 0,
\]
if $B$ is disjoint from it.
This proves  $(X_{t,s})_\sharp \rho_s^{\rm sg}\perp \leb$, and the remaining claims
follow by considering the Lebesgue decomposition of the measure 
$(X_{t,s})_\sharp \rho_s^{\rm ac}$, which may have a singular part.
\end{proof}

    
We start the proof of Theorem~\ref{t:rho_decomp} by first developing two lemmas.
(Note the first lemma holds with $X_t$ replaced by any Lipschitz and surjective map.)

\begin{lemma} \label{lem:calSt}
The sets $\xndf$, $X_t(\xndf)$, and $X_t(\xsg)$  have zero Lebesgue measure.
Moreover, 
\[
|\xinv\cap  \{y: X_t(y)\in X_t(\xsg\cup\xndf)\}|=0.
\]
\end{lemma}
\begin{proof}
1. Evidently the set $\xndf$ has Lebesgue measure $|\xndf|=0$, so the same is true of
its Lipschitz image, i.e., $|X_t(\xndf)|=0.$
Also, since $\det\nabla X_t(y)=0$ for a.e.~$y\in \xsg$ and $X_t$ is Lipschitz, we find that
$|X_t(\xsg)|=0$, due to \cite[Lem.~2.73]{AmbrosioFuscoPallara}. 
(Note that we do not expect $|\xsg|=0$ in general, however.)

2.  According to the area formula for Lipschitz 
functions from~\cite[Thm.~2.71]{AmbrosioFuscoPallara}
or \cite[Thm.~3.2.3]{FedererGMT}, we have that
for any Lebesgue measurable set $E\subset\R^d$,
\begin{equation}\label{e:areaE}
    \int_E |\det\nabla X_t(y)|\,dy = \int_{\R^d} N_E(x)\,dx\,,
\end{equation}
where the multiplicity function $N_E$ is Lebesgue measurable and
given in terms of counting measure ($0$-dimensional Hausdorff measure) as 
\begin{equation}\label{d:NE}
N_E(x) :=  \#\bigl(X_t\inv(x)\cap E\bigr) =  \#\{y\in E: x=X_t(y)\}  . 
\end{equation}
Taking $E= \{y: X_t(y)\in X_t(\xsg\cup\xndf)\}$, 
since the determinant is non-vanishing on $\xinv$, it follows
$|\xinv\cap E| = 0$.
\end{proof}

\begin{lemma} \label{lem:detXt}
The determinant $\det\nabla X_t(y)\ge0$ for a.e.~$y\in\R^d$.
\end{lemma}
\begin{proof}
For any $\eps>0$ and $t\ge0$, $\det\nabla X_t^\eps(y)>0$ for all $y$ since the flows induced
by \eqref{e:Xeps_flow} are smooth. The gradients $\nabla X_t^\eps$ are uniformly bounded and converge to 
$\nabla X_t$ weak-$\star$ in $L^\infty$ due to the local uniform convergence in Theorem~\ref{t:UNIQ}
and fact that smooth functions of compact support are dense in $L^1$.
Then by the weak-$\star$ continuity property of determinants 
stated in \cite[Thm.~2.16]{AmbrosioFuscoPallara}
(see also \cite[Def.~2.9]{AmbrosioFuscoPallara}),
for any integrable $g:\R^d\to[0,\infty)$ we have 
\[
0\le  \int_{\R^d} g(y)\det\nabla X_t^\eps(y)\,dy \to 
\int_{\R^d} g(y)\det\nabla X_t(y)\,dy \quad\text{as $\eps\to0$}.
\]
The claimed result follows.
\end{proof}

\begin{proof}[Proof of Theorem~\ref{t:rho_decomp}] 
We claim $\rho_t\mres\calS_t^c\ll\leb$.
Let $B\subset\calS_t^c$ with Lebesgue measure $|B|=0$. 
For $E=X_t\inv(B)$, noting that $N_E(x)=0$ whenever $x\notin B$,
the area formula  \eqref{e:areaE} yields
\begin{align*}
    \int_{X_t\inv(B)} \det\nabla X_t(y)\,dy = 
    \int_{\R^d} N_E(x)\,dx =
    \int_B N_E(x)\,dx = 0,
\end{align*}
since even $\int_B \infty\,dx = 0$ as a Lebesgue integral.
Then, because $X_t\inv(\calS_t^c)\subset \xinv\cup\xndf$ and  
the determinant is positive a.e.~on $\xinv$, it follows $|X_t\inv(B)|=0$.
Hence $\rho_t\mres\calS_t^c\ll \leb$, while $|\calS_t|=0$ by Lemma~\ref{lem:calSt}.
The result follows.
\end{proof}

\subsection{Unique backward images a.e.}
The next result shows that the adhesion model indeed possesses
one of the properties that motivated it,
namely that of avoiding overlapping Lagrangian images a.e.

\begin{prop}\label{p:multiplicity}
Let $t>0$. For a.e.~$x\in\R^d$, the pre-image $X_t\inv(x)$ is a singleton set
$\{y\}\subset \xinv$.
Thus the multiplicity $N_{\R^d}(x)=\# X_t\inv(x)=1$ for a.e.~$x\in\R^d$.
\end{prop}
\begin{proof}
Let $h:\R^d\to [0,\infty)$ be continuous with compact support. 
For any $\eps>0$ the standard change of variable formula yields
\begin{equation}
    \int_{\R^d} h(x)\,dx = \int_{\R^d} h(X_t^\eps(y)) \det \nabla X_t^\eps(y)\,dy\,.
\end{equation}
Using the weak-$\star$ continuity of determinants cited in the proof of 
Lemma~\ref{lem:detXt}, we find that  as $\eps\to0$,
\[
\int_{\R^d} h(X_t(y)) \det\nabla X_t^\eps(y)\,dy
\to
\int_{\R^d} h(X_t(y)) \det\nabla X_t(y)\,dy \,.
\]
Since $|h(X_t^\eps(y))-h(X_t(y))|\to0$ uniformly as $\eps\to0$
due to Theorem~\ref{t:UNIQ}, and $\det\nabla X_t^\eps(y)$ is uniformly bounded,
it follows that 
\begin{equation}
    \int_{\R^d} h(x)\,dx 
    = \int_{\R^d} h(X_t(y)) \det \nabla X_t(y)\,dy \,.
\end{equation}
Taking $h$ along an increasing sequence converging to the characteristic function of 
any bounded open set $\Omega$, it follows 
\[
\int_\Omega dx = \int_{X_t\inv(\Omega)} \det\nabla X_t(y)\,dy.
\]
Taking $E=X_t\inv(\Omega)$, we find the multiplicity function in \eqref{d:NE}
vanishes for $x\notin\Omega$ and is simply given as
\begin{equation}\label{e:NOmega}
N_E(x) = \one_{\Omega}(x)\, \# X_t\inv(x) .
\end{equation}
Applying the area formula \eqref{e:areaE} we obtain 
\[
\int_\Omega dx = \int_\Omega \# X_t\inv(x)\,dx \,.
\]
Since $X_t$ is surjective, so $\# X_t\inv(x)\ge1$ for all $x$,
and $\Omega$ is arbitrary, the claimed result follows, 
after taking into account Lemma~\ref{lem:calSt}.
\end{proof}

\begin{remark}
For each point in the set $\{x\in \R^d:\# X_t\inv(x)=1\}$, the outer semi-continuity 
property of $X_t\inv$ described in Remark~\ref{rem:outersemi} reduces to the statement
that for the unique $y$ satisfying $X_t(y)=x$,
\[
\forall \eps>0\ \exists \delta>0 \quad X_t\inv( B(x,\delta)) \subset B(y,\eps).
\]
Thus, any right inverse $Y_t:\R^d\to\R^d$ for $X_t$ must be continuous at each such point $x$,
so continuous a.e.  A similar point has been made by Lions \& Seeger~\cite[Prop. 2.2]{LionsSeeger2024}
for a general family of backward Lagrangian flows satisfying a differential inclusion 
with one-sided Lipschitz condition.
\end{remark}

\begin{cor}\label{cor:rhoac}
 The result $\rhoac = \rho_t\mres\calR_t$ in Theorem~\ref{t:rho_decomp} holds 
 whenever 
 \[
 \calR_t \subset 
 \{x\in\R^d: \# X_t\inv(x)=1\}  \cap \calS_t^c
 \quad\text{and}\quad
 |\calR_t^c|=0.
 \]
\end{cor}
Note, if $\calR_t$ is taken as large as possible in this result, so equality holds,
then
\begin{equation}
    \calR_t^c = \R^d\setminus\calR_t = 
    \{x\in X_t(\xinv):\# X_t\inv(x)>1\} \cup  \calS_t \, ,
\end{equation}
and indeed $|\calR_t^c|=0$. Moreover, $X_t: X_t\inv(\calR_t)\to\calR_t$ is a bijection,
and for any Borel $B\subset \calR_t$,
\[
\rhoac(B) = |X_t\inv(B)| = (\leb\mres X_t\inv(\calR_t)) (X_t\inv(B)).
\]
Thus 
\begin{equation}
\rhoac = (X_t)_\sharp(\leb\mres X_t\inv(\calR_t)) \,.
\end{equation}

\subsection{Continuity equation}

Let $\calQ=\R^d\times(0,\infty)$ denote the space-time domain, and 
let $\calB$ be the Borel $\sigma$-algebra of subsets of $\calQ$.
We wish to show that pushforward under the Lagrangian flow maps
determines a solution of a continuity equation
\begin{equation}\label{e:continuity_rho}
   \D_t \varrho + \nabla\cdot (\upsilon\varrho)  = 0, 
\end{equation}
with a single-valued, discontinuous velocity field $\upsilon:\calQ\to\R^d$ 
that also provides the Lagrangian flow.  We will discuss this 
by defining a velocity field in a manner similar to 
Bianchini \& Gloyer, who rely on some results in descriptive set theory;
see \cite[Sec.~5]{BianchiniGloyer2011}.

Recall that $X$ is locally Lipschitz on $\calQ$ and $X_t$ is surjective for each $t\ge0$. 
Also, by the ``sticky'' property enjoyed by $X$ from Theorem~\ref{t:UNIQ}(c.ii),
if $X_t(y)=X_t(z)$ then $X_s(y)=X_s(z)$ for all $s\ge t$. Thus, we may define
a bounded vector field $\upsilon:\calQ\to\R^d$ by 
\begin{equation}\label{d:vf}
\upsilon(x,t) = \D^+_t X(y,t) \quad \text{whenever\ \ } x=X(y,t) ,
\end{equation}
where $\D^+_t X$ denotes the componentwise upper right Dini derivative. 
Clearly
\begin{equation}\label{e:upsi-K}
|\upsilon(x,t)|\le K \quad  \text{for all $(x,t)\in\calQ$}.
\end{equation}
Note that for each $y\in\R^d$ we then have that 
\begin{equation}\label{e:Dtvae}
    \D_t X(y,t) = \upsilon(X(y,t),t) \quad\text{for Lebesgue-a.e. $t>0$}.
\end{equation}

It is not clear whether $\upsilon$ is Borel measurable. 
Define $\calX:\calQ\to\calQ$ by
\[
\calX(y,t) = (X(y,t),t) \,. 
\]
The function $\D^+_t X =\upsilon\circ\calX$ is Borel, and 
the pre-image under $\upsilon$ of a Borel set $E\subset \R^d$ is the forward image 
under $\calX$ of the Borel set $(\D^+_t X)\inv(E)$. Such a set may not be Borel, but may 
be characterized by results in descriptive set theory that we will mention below.

What is clear and uncomplicated is the following.
The pushforward of the Borel $\sigma$-algebra $\calB$, defined by
\begin{equation}\label{d:pushsigma}
\calX_\sharp\calB := \{ E\subset \calQ: \calX\inv(E)\in\calB\}\,,
\end{equation}
is a $\sigma$-algebra that extends $\calB$ itself,
with respect to which $\upsilon$ is measurable, 
since $(\upsilon\circ\calX)\inv=\calX\inv\circ\upsilon\inv$.
Moreover, the mass measure
in $\calQ$ is naturally defined on this $\sigma$-algebra by the pushforward formula
$\varrho= \calX_\sharp\calL^{d+1}$, i.e., $\varrho(E)=\calL^{d+1}(\calX\inv(E))$.
When $E$ is Borel, we note
\begin{equation}\label{e:rhorho_t}
\varrho(E) = \int_E \,d\varrho = \int_0^\infty\int_{\R^d} \one_E(X(y,t),t)\,dy\,dt
=\int_0^\infty \int_{\R^d} \one_E(x,t)\,d\rho_t(x)\,dt.
\end{equation}
Then for any bounded function $g:\calQ\to\R$ with compact support that is $\calX_\sharp\calB$-measurable,
\begin{equation}\label{d:varrho}
\int_\calQ g(x,t)\,d\varrho(x,t) = \int_\calQ g(X(y,t),t)\,dy\,dt.
\end{equation}

We now claim that $\varrho$ is a solution of the continuity equation~\eqref{e:continuity_rho}.

\begin{prop}
The pushforward mass measure $\varrho=\calX_\sharp\calL^{d+1}$ satisfies
the continuity equation~\eqref{e:continuity_rho} 
in the sense of distributions on the space-time domain $\calQ$. 
\end{prop}

\begin{proof}
Let
$f:\calQ\to\R$ be any smooth function with compact support. Then by \eqref{e:Dtvae}, 
\begin{align}
    \label{e:continuity_rhom}
    0 & = \int_\calQ \frac{d}{dt} f(X(y,t),t)\,dy\,dt 
    \nonumber \\&=
    \int_\calQ \Bigl(\D_t f(X(y,t),t) + \nabla f(X(y,t),t)\cdot\upsilon(X(y,t),t) \Bigr)\,dy\,dt
    \nonumber \\&=
    \int_\calQ \Bigl(\D_t f(x,t)+ \nabla f(x,t)\cdot \upsilon(x,t)\Bigr) \,d\varrho(x,t) .
\end{align}
Thus equation~\eqref{e:continuity_rho} holds in the sense of distributions on $\calQ$.
\end{proof}

\begin{remark}
Results in descriptive set theory apply as in \cite{BianchiniGloyer2011} 
to provide further information about the sets in $\calX_\sharp\calB$.  
Each set $E$ in this $\sigma$-algebra is the image under $\calX$ of a Borel set.
Since $\calX$ is Borel, $E$ is an analytic set, therefore universally measurable
\cite[Thm.~21.10]{Kechris1995}.
This means that for any $\sigma$-finite Borel measure $\mu$ on $\calQ$,
$E$ is the union of a Borel set and a $\mu$-negligible set, and  
thus $\calX_\sharp\calB$  lies in the domain of the completion of every such measure $\mu$.
In particular, the measure $\varrho$ is determined by the 
fact that \eqref{e:rhorho_t} holds for Borel sets $E$.
\end{remark}

\section{Smoothed \MA\ measures and transport maps}\label{s:smoothMA}

In a number of numerical studies of reconstruction through least action,
the mass distribution is smoothed before 
obtaining a backward transport map through solution of a \MA\ equation of the form in \eqref{e:MA1}.
In this section we study a family of \MA\ measures with smooth densities,
and associated transport maps, that are naturally related to the adhesion model. 
Our results provide a natural smooth approximation to the \MA\ measures $\mam_t$
and the associated transport maps $T_t$,
and provide insight into the dynamics and structural properties of these objects.

\subsection{Smoothed \MA\ measures}

We define \MA\ measures  $\mam^\eps_t$  with smooth densities
(written with the same notation) by
\begin{equation}\label{d:mamepst}
\mam^\eps_t(x)  = \det \nabla^2 w^\eps_t (x) \,,
\qquad
    w^\eps_t(x) = \tfrac12|x|^2 - tu^\eps_t(x)\,.
\end{equation}
Recall from Lemma~\ref{p:lim_wueps} that $w^\eps_t$ smoothly 
approximates the potential $w_t$ 
whose associated \MA\ measure $\mam_t$ is described in \eqref{d:MAint}.
We will show that in the zero-viscosity limit, 
the measures $\mam^\eps_t$ consistently approximate the measures $\mam_t$,
and describe a continuity equation that they satisfy.

\begin{prop} \label{p:MAepslim} 
Assume $\vp\colon\R^d\to\R$ is Lipschitz, let $t>0$, and let 
$\mam_t$ denote the \MA\ measure  for $w_t=\pts$.
Then as $\eps\to0$, the \MA\ measures $\mam^\eps_t$ converge locally weak-$\star$ to $\mam_t$.
I.e., for every continuous $g\colon\R^d\to\R$
having compact support,
\[
\int_{\R^d} g(x)\,\mam^\eps_t(x)\,dx \to 
\int_{\R^d} g(x)\,d\mam_t(x)
\qquad\text{as $\eps\to0$}. 
\]
\end{prop}
\begin{proof}
Recall from Lemma~\ref{p:lim_wueps} that $w^\eps_t$ converges to $w_t$
uniformly on compact sets. By applying a standard stability theorem for \MA\ measures  
(e.g., see  \cite[Prop.~2.6]{Figalli-MongeAmpere}),   
we immediately obtain the stated weak convergence of \MA\ measures.
\end{proof}

This convergence result yields the following lower bound on \MA\ mass density when
the initial velocity potential $\vp$ is $\lambda$-concave.
\begin{cor}\label{cor:lowerbd-mam}
    In addition to the assumptions of Proposition~\ref{p:MAepslim}, 
    assume $\vp$ is $\lambda$-concave with $\lambda\ge0$. Let $t>0$.
    Then 
    \begin{itemize}
        \item[(i)] $\mam_t^\eps(x)\ge (1+\lambda t)^{-d}$ for all $\eps>0$ and all $x\in\R^d$.
        \item[(ii)] $\mam_t(B)\ge (1+\lambda t)^{-d}|B|$ for all Borel sets $B\subset\R^d$.
    \end{itemize}
\end{cor}
\begin{proof}
    By Lemmas~\ref{l:lambda} and~\ref{lem:oneside},
    each eigenvalue of the Hessian $\nabla^2u^\eps_t$ is no greater than $\lambda_t=\lambda/(1+\lambda t)$.
    Then by the definition of $w^\eps_t$ in \eqref{d:mamepst} it follows
    each eigenvalue of $\nabla^2 w^\eps_t(x)$ is bounded below by  
    \begin{equation}\label{d:betat}
    \beta_t:= 1-t \lambda_t  = (1+\lambda t)\inv.  
    \end{equation}
    Part (i) follows.
    Then Proposition~\ref{p:MAepslim} yields part (ii) by a standard approximation argument 
    using the regularity of Radon measures, cf.~\cite[p.~212]{Folland1999}.
\end{proof}

\subsection{Smoothed transport maps}
The smooth transport maps $T_t^\eps$ defined as 
inverse to $\nabla w^\eps_t$ satisfy Lipschitz estimates that 
can be compared to those satisfied by $X_t^\eps$ as shown in Proposition~\ref{p:Xepslip}.
Recall from Lemma~\ref{l:lambda} that $u_t^\eps$ is $\lambda_t$-concave,
so similar to \eqref{e:wt-betat} we have
\begin{equation}\label{e:wtepsg}
w_t^\eps(x)= \frac{1}{1+\lambda t} \frac{|x|^2}2 + g_t^\eps(x)\,,
\end{equation}
where $g_t^\eps$ is convex, and also now smooth.

\begin{prop} \label{p:Tepslim}
Assume $\vp$ is Lipschitz, and $\lambda$-concave with $\lambda\ge0$. 
Let $t>0$.
    For any $\eps>0$, the gradient map $\nabla w^\eps_t:\R^d\to\R^d$
    is bijective, with 
    \begin{equation}\label{e:weps_lower}
    | \nabla w^\eps_t(x_1)- \nabla w^\eps_t(x_2)| \ge \frac{|x_1-x_2|}{1+\lambda t}.
    \end{equation}
   for all $x_1,x_2\in\R^d$. Also, for all $y_1,y_2\in\R^d$,
   the inverse $T^\eps_t = (\nabla w^\eps_t)\inv$ satisfies
    \begin{equation} \label{e:Teps_stable}
{|T^\eps_t(y_1)-T^\eps_t(y_2)|}\le (1+\lambda t)|y_1-y_2| .
    \end{equation}
\end{prop}
\begin{proof}
  Due to the representation \eqref{e:wtepsg}, the surjectivity of $\nabla w^\eps_t$
    follows from a simple minimization argument, see \cite[Prop.~A.1]{LPS19}.
    The lower bound \eqref{e:weps_lower} follows from \eqref{e:wtepsg} and 
    the easily proved monotonicity formula
    \[
    (\nabla g_t^\eps(x_1)-\nabla g_t^\eps(x_2))\cdot(x_1-x_2)\ge 0,
    \]
    whence \eqref{e:Teps_stable} follows.
\end{proof}

Complementing the result of Proposition~\ref{p:MAepslim} regarding the convergence of the 
smoothed \MA\ measures $\mam_t^\eps$,  next we show 
$T_t^\eps=(\nabla w_t^\eps)\inv$ converges to $T_t=\nabla w_t^*$.
The key to this  is to extend the convergence $w^\eps_t\to w_t$ from Lemma~\ref{p:lim_wueps}
to convergence of Legendre transforms $w_t^{\eps*}\to w_t^*=\ptss$.
The idea of the proof is to exploit the strong convexity of $w_t$ in \eqref{e:wt-betat}.
This avoids use of the theory of Mosco convergence for Legendre transforms, 
as developed in \cite{Mosco1971,Beer1990}.

\begin{theorem}\label{t:limTeps}
Assume $\vp$ is Lipschitz, and $\lambda$-concave with $\lambda\ge0$. 
Let $t>0$. Then $w_t^*$ is $C^1$, $w_t^{\eps*}$ is smooth for all $\eps>0$, 
and  for all $y\in\R^d$,
\begin{itemize}
    \item[(i)] $w_t^{\eps*}(y) \to w_t^*(y)$ as $\eps\to0$, and
    \item[(ii)] $T_t^\eps(y)=\nabla w_t^{\eps*}(y) \to T_t(y)=\nabla w_t^*(y)$ as $\eps\to0$.
\end{itemize}
In both (i) and (ii), the convergence is uniform on each compact set in $\R^d$.
\end{theorem}

With this proof, we start to make use of some more technical concepts and results from convex analysis. 
Particularly basic is a fact from Rockafellar \cite[Thm.~23.5]{Rockafellar},
implying that for any convex function $f:\R^d\to\R$ and $y,x\in\R^d$, 
$y\in\D f(x)$ if and only if $x\in \D f^*(y)$, with $f^*$ the Legendre transform of $f$.
Also, the Young identity  $f(x)+f^*(y)=x\cdot y$ holds if and only if  $y\in\D f(x)$.

\begin{proof}
That $w_t^*=\ptss$ is $C^1$ follows  from  Proposition~\ref{p:fprop}(ii) by taking $f=\psi_t$. 
To prove (i), fix $y_0\in\R^d$  and let $x_0=\nabla w_t^*(y_0)=T_t(y_0)$. 
Then $y_0\in\D w_t(x_0)$ and 
\begin{equation}\label{d:h0}
w_t^*(y_0)  = \sup_x\bigl( y_0\mdot x - w_t(x)\bigr)  = y_0\mdot x_0 - w_t(x_0) \,,
\end{equation} 
by the Young identity.  Recall from \eqref{e:wt-betat} that $w_t(x)= \tfrac12\beta_t|x|^2 + g_t(x)$, 
where $\beta_t:=(1+\lambda t)\inv$ and $g_t$ is convex.  Then 
\begin{equation}\label{e:dwtgt}
   \D w_t(x) = \beta_t x + \D g_t(x)
\end{equation}
(see \cite[Prop.~A.1(ii)]{LPS19} for a quick proof), and it follows that 
\begin{equation}\label{e:wt_strong}
   w_t(x) \ge w_t(x_0)+y_0\mdot(x-x_0)+\tfrac12\beta_t|x-x_0|^2 
\quad\text{for all $x\in\R^d$.}
\end{equation}
(Alternatively, \eqref{e:wt-betat} implies \eqref{e:wt_strong} by \cite[Prop.~B.1.1.2 and Thm.~D.6.1.2]{Hiriart2001}.) 
Note
\begin{equation}\label{d:heps}
    w_t^{\eps*}(y_0) = \inf\{h\in\R: h> y_0\mdot x - w_t^\eps(x) \ \ \text{for all $x\in\R^d$}\}.
\end{equation}
Let $\delta>0$ and let $r>0$ satisfy $3\delta=\frac12\beta_t r^2$.
Invoking the local uniform convergence result in Lemma~\ref{p:lim_wueps}, 
there exists $\eps_0>0$ such that  for all $\eps\in(0,\eps_0)$,
\[
|x-x_0|\le r \quad\text{implies}\quad |w_t^\eps(x)-w_t(x)| <\delta .
\]
Let $\eps\in(0,\eps_0)$. Then taking $x=x_0$ in \eqref{d:heps}, by \eqref{d:h0} we have 
\[
w_t^{\eps*}(y_0) \ge y_0\mdot x_0 - w_t^\eps(x_0) = w_t^*(y_0) + w_t(x_0)-w_t^\eps(x_0) > w_t^*(y_0) -\delta.
\]
When $|x-x_0|=r$ on the other hand, by 	 
\eqref{e:wt_strong} we find
\[
y_0\mdot x - w_t^\eps(x) <y_0\mdot x - w_t(x) + \delta  
\le w_t^*(y_0) - \tfrac12\beta_t r^2 + \delta = w_t^*(y_0)-2\delta.
\]
By its concavity, the quantity $y_0\mdot x-w_t^\eps(x)\le w_t^*(y_0)-2\delta$ also whenever $|x-x_0|\ge r$. It follows
\[
w_t^{\eps*}(y_0) = \sup_{|x-x_0|\le r} \bigl(y_0\mdot x - w_t^\eps(x)\bigr) < w_t^*(y_0)+\delta.
\]
Thus $|w_t^{\eps*}(y_0)-w_t^*(y_0)|<\delta$ whenever $\eps\in(0,\eps_0)$, and this proves (i).

Next we prove (ii). Note first that $\D w_t^{\eps*}(y)$ is the singleton set $\{T_t^\eps(y)\}$ due
to the bijectivity of $\nabla w_t^\eps$ from Proposition~\ref{p:Tepslim}. 
This singleton property implies $w_t^{\eps *}$ is $C^1$ by \cite[Cor.~25.5.1]{Rockafellar}.
Since $\nabla w_t^{\eps*}=(\nabla w_t^\eps)\inv=T_t^\eps$ and this is smooth by the
inverse function theorem, $w_t^{\eps*}$ is smooth.

Then because  $w_t^{\eps*}$ and $w_t^*=\ptss$ are differentiable and convex, 
the locally uniform convergence of gradients $\nabla w_t^{\eps*}(y)\to \nabla w_t^*(y)$
follows immediately from  Corollary~D.6.2.8 of~\cite{Hiriart2001}.
\end{proof}

\subsection{Continuity equation}
The smooth \MA\ density $\mam^\eps$ for $w^\eps$ satisfies 
\begin{equation}\label{e:mamCE}
    \D_t\mam^\eps + \nabla\cdot(\mam^\eps V^\eps) = 0, 
\end{equation}
with velocity field taking the form
\begin{equation}\label{d:Veps}
    V^\eps = \nabla u^\eps - \tfrac12{\eps} (\nabla^2 w^\eps)\inv \nabla\Delta w^\eps  \,.
\end{equation}
The reason for this is the following.
By  differentiating the relation
\begin{equation}\label{e:ywX}
y = \nabla w^\eps_t( T^\eps_t(y))   
\end{equation}
in $t$ and  using $w^\eps_0(y)=\frac12|y|^2$ we find
\begin{equation}
\D_t T^\eps_t(y) = V^\eps_t(T^\eps_t(y)), 
\qquad T^\eps_0(y)=y \,,
\end{equation}
with
\begin{equation}\label{e:Veps2}
V^\eps_t = -(\nabla^2 w^\eps_t)\inv \nabla\D_t w^\eps_t\,.
\end{equation}
Differentiating \eqref{e:ywX} in $y$ and using \eqref{d:mamepst} we find 
\[
\mam^\eps_t(x) = \det \nabla T^\eps_t(y)\inv,
\]
whence we get the continuity equation~\eqref{e:mamCE} with $V^\eps$ given by \eqref{e:Veps2}.
But now, from \eqref{d:mamepst} we have that $\nabla w^\eps = x - t\nabla u^\eps$, 
    that $\nabla^2 w^\eps = I - t\nabla^2 u^\eps$,  and that
\begin{align*}
    -\nabla\D_t w^\eps &= \nabla u^\eps + t\nabla\left(\tfrac12\eps\Delta u^\eps - \tfrac12|\nabla u^\eps|^2\right)
    \\
    &= (I-t\nabla^2 u^\eps)\nabla u^\eps - \tfrac12\eps\nabla\Delta w^\eps\,.
\end{align*}
Hence \eqref{d:Veps} follows.

Similar to \eqref{e:ch_vars},  the measures $\mam_t^\eps$ relate to the 
inverse Lagrangian maps by the change-of-variables formula valid for any Borel set $B$:
\begin{equation} \label{e:ch_var2}
    \int_B \mam^\eps_t(x)\,dx = \int_B \det \nabla^2 w^\eps_t(x)\,dx = 
    \int_{\nabla w^\eps_t(B)} dy = \int_{(T^\eps_t)\inv(B)}dy.
\end{equation}
This equation shows 
$\mam^\eps_t$ is the measure-theoretic pushforward of Lebesgue measure $\leb$ under $T^\eps_t$,
written 
\[
\mam^\eps_t = (T^\eps_t)_\sharp \leb.
\]

In one space dimension ($d=1$), one has $\mam^\eps_t = \D_x^2 w^\eps_t$ and equations \eqref{e:mamCE}--\eqref{d:Veps}
combine into the single advection-diffusion equation
\begin{equation}
    \D_t \mam^\eps + \D_x(\mam^\eps\D_x u^\eps) = \tfrac12\eps\D_x^2 \mam^\eps \,.
\end{equation}

\section{\MA\ measures and convexified transport}\label{s:MAmeasures}

In this section we study the maps
$T_t=\nabla\ptss$ generated by the convexified transport potential $\ptss$
and the associated \MA\ measures  $\mam_t=(T_t)_\sharp\leb$,
focusing on results that are relevant for analysis of the adhesion model.
In particular, we establish a domain of dependence result
and a number of properties involving strict convexity and 
the ``touching set'' where $\psi_t$ agrees with its convexification $\ptss$.

\subsection{Convexified transport maps}

We begin our analysis of the transport maps $T_t$ 
by  establishing a number of their basic properties, assuming $\vp$ is semi-concave.
Recall $w_t=\pts$.

\begin{prop}\label{p:Tbasic}
Assume $\vp$ is Lipschitz, and $\lambda$-concave with $\lambda\ge0$. 
Then for each $t>0$, the function $\ptss$ is $C^1$ 
with Lipschitz gradient $T_t=\nabla\ptss$. Also:
\begin{itemize}
    \item[(i)]  $T_t:\R^d\to\R^d$ is surjective, and $T_t\inv=\D w_t$, so $T_t\circ\D w_t$
    is the identity map.
\item[(ii)] For all $y_1,y_2\in\R^d$, 
\ $
\displaystyle \frac{|T_t(y_1)-T_t(y_2)|}{1+\lambda t}\le |y_1-y_2|.
$
\end{itemize}
\end{prop}

\begin{remark}
The stability property in part (ii) does {\em not} propagate in a way
similar to the property that $X_t$ enjoys from Theorem~\ref{t:UNIQ}(ii).
As we show in section \ref{s:example} below, the left-hand side of the inequality here in (ii)
is {\em not} a decreasing function of $t$ in general.
In fact, $|T_t(y_1)-T_t(y_2)|$ can vanish for some time, then later become positive.
\end{remark}

\begin{remark} Under the conditions of Proposition~\ref{p:Tbasic},
    the Hessian $\nabla^2\ptss=\nabla T_t$ exists a.e.~in the classical sense,
    by Rademacher's theorem. 
\end{remark}

\begin{proof}[Proof of Prop.~\ref{p:Tbasic}] 
That $\ptss$ is $C^1$ was proved in Theorem~\ref{t:limTeps} using  Proposition~\ref{p:fprop}(ii).
To prove (i), let $x\in\R^d$. 
Then because $\psi_t$ grows quadratically at $\infty$,
$x\cdot y-\psi_t(y)$ is maximized at some $y\in\D\pts(x)$,
and then necessarily
\[ x\in\D\ptss(y)=\{\nabla\ptss(y)\}.\]
This shows $T_t$ is surjective,  and $T_t\circ\D\pts$ is the identity map on $\R^d$. 
We can conclude $T_t\inv=\D w_t$ since $x\in\D\ptss(y)$ if and only if $y\in\D\pts(x)$, 
by \cite[Thm.~23.5]{Rockafellar}.

For part (ii), for $j=1,2$ let $x_j=T_t(y_j)$, so that $y_j\in\D w_t(x_j)$. 
Recall from \eqref{e:dwtgt} that $\D w_t(x)=\beta_t x+\D g_t(x)$ 
where $\beta_t=(1+\lambda t)\inv$ and $g_t$ is convex. 
Then $z_j:=y_j-\beta_t x_j\in \D g_t(x_j)$. 
The monotonicity of $\D g_t$ then implies $(x_1-x_2)\mdot(z_1-z_2)\ge0$, 
whence it follows 
\[
(x_1-x_2)\mdot(y_1-y_2) \ge \frac{|x_1-x_2|^2}{1+\lambda t}  .
\]
We can then conclude by the Cauchy-Schwarz inequality.
\end{proof}

\subsection{Domain of dependence}\label{ss:domain}

Recall $\psi_t(y)=\frac12|y|^2+t\vp(y)$, and that 
$\mam_t(B)=|\D\pts(B)|$ for every Borel set $B\subset\R^d$.
Intuitively, if the subgradient $\D\psi_t^*(x)$ has nonempty interior,
it should contain all the mass that concentrates at $x$ at time $t$.
If $t$ is small, this mass should come from a set $T_t\inv(x)$ of diameter $O(t)$, limited
by finite propagation speed. Moreover, the values of $\psi_t^*$ and 
its subgradient at $x$ should only depend on the values of $\vp$ in the same region.
In this direction we have the following result. 
Recall $B(x,r)$ denotes the open ball with center $x$ and radius $r$.
(We remark that in this subsection and the next we do {\em not} assume $\vp$ is semi-concave, 
except at the end, for Proposition~\ref{p:Sigt123}.)

\begin{prop}[Finite propagation speed and domain of dependence]\label{p:speed}
Fix $t>0$, $x\in\R^d$.
\begin{itemize}
    \item[(i)] Suppose $\vp:\R^d\to\R$ is $K$-Lipschitz.
Then $\D\pts(x)\subset\overline{B(x,Kt)}$. 
\item[(ii)] Suppose in addition that $\hat\vp:\R^d\to\R$ is $K$-Lipschitz, and 
define $\hat\psi_t(y)=\frac12|y|^2+t\hat \vp(y)$.
If  $\hat\vp=\vp$ in $B(x,Kt)$, then  $\hat\psi_t^*(x)=\pts(x)$ and  
$\D\hat\psi_t^*(x)=\D\pts(x)$.
\end{itemize}
\end{prop}

First we prove a lemma involving the ``touching sets'' defined for $t>0$ by 
\begin{align} \label{d:Ths}
\Theta_t &=  \{y\in\R^d:  \psi_t(y)=\ptss(y)\}\,,
\end{align}
and similarly define $\hat\Theta_t$ for $\hat\psi_t$.

\begin{lemma}\label{lem:localspeed1}
Under the assumptions of Proposition~\ref{p:speed}, 
for all $x\in\R^d$ and $t>0$ we have
\begin{itemize}
    \item[(i)] $\D\pts(x)\cap \Theta_t\subset \overline{B(x,Kt)}$.
\item[(ii)]  $\hat\psi^*_t(x)=\pts(x)$ \ \ and \ \ %
$\D\hat\psi_t^*(x)\cap\hat\Theta_t=\D\pts(x)\cap\Theta_t$.
\end{itemize} 
\end{lemma}

\begin{remark}\label{r:HLoptimizers}
The set $\D\pts(x)\cap\Theta_t$ is the set of optimizers $y$ in the 
Hopf-Lax formula~\ref{e:HLOmega1}. The proof is nearly identical to that
of Lemma~6.6 of \cite{liu2025rigidly}.  
But we will not make any use of this fact in this paper. 
\end{remark}

\begin{proof}
(i) Let $y_0\in \D\pts(x)\cap\Theta_t$, so that 
$\psi_t(y_0)=\ptss(y_0)$.
From the touching and subgradient conditions we infer 
that for all $y\in \R^d$,
\[
\psi_t(y_0)-x\cdot y_0 \le \psi_t(y)-x\cdot y.
\]
Adding $\frac12|x|^2$ to both sides and recalling $\psi_t(y)=\frac12|y|^2+t\vp(y)$ we find 
\[
t\vp(y_0) +\tfrac12|y_0-x|^2
\le t\vp(y) +\tfrac12|y-x|^2
\]
Taking $y=y_0+s(x-y_0)$, for small $s>0$ we have $|y-x|=(1-s)|y_0-x|$
and
\[
\tfrac12 |y_0-x|^2 (1-(1-s)^2) \le t(\vp(y)-\vp(y_0)) \le Kts |y_0-x|,
\]
Taking $s\to0$, we infer $|y_0-x|\le Kt$ so (i) follows.

(ii) Let $x\in\R^d$. From the definition of $w_t=\pts$ and the fact that 
$\ptss=w_t^*$ is the largest convex function majorized by $\psi_t$, we have
\begin{equation} \label{e:youngpts}
\pts(x)  \ge x\cdot y  - \ptss(y)  
\ge x\cdot y  - \psi_t(y) 
\end{equation}
for all $y\in\R^d$.  Because $\psi_t(y)$ grows quadratically as $|y|\to\infty$,
both equalities hold for some $y=y_0\in\Theta_t$. 
Then $y_0\in\D\pts(x)$ also, and
indeed both equalities hold for an arbitrary $y_0\in\D\pts(x)\cap\Theta_t$.
By part (i) we infer $|y_0-x|\le Kt$.

Now, because $\vp=\hat\vp$ on $B(x,Kt)$ we have $\psi_t(y_0)=\hat\psi_t(y_0)$, 
hence similarly to \eqref{e:youngpts} we have that
\begin{align*}
\hat\psi_t^*(x) 
\ge  x\cdot y_0  - \hat\psi_t^{**}(y_0)  
\ge  x\cdot y_0  - \hat\psi_t(y_0)  
=  \pts(x) .
\end{align*}
Interchanging $\vp$ and $\hat\vp$ we deduce
$\psi_t^*(x)=\hat\psi_t^*(x)$. It follows 
$\hat\psi_t^{**}(y_0)=\hat\psi_t(y_0)$ and
$\hat\psi_t^*(x) + \hat\psi_t^{**}(x) = x\cdot y_0$ also, 
hence $y_0\in \D\hat\pts(x)\cap\hat\Theta_t$.
This finishes the proof.
\end{proof}

Proposition~\ref{p:speed} follows directly from the next result,
showing the subgradient $\D\pts(x)$ is the convex hull of its points
that lie in the touching set $\Theta_t$.

\begin{lemma}\label{lem:convexcomb}
   Under the assumptions of Proposition~\ref{p:speed}, for all $x\in\R^d$
   and $t>0$ each point $y\in \D\pts(x)$ is a convex combination of points in 
   $\D\pts(x)\cap\Theta_t$.
\end{lemma}

\begin{proof}
1. Let $x\in\R^d$. By the Young inequality,
$x\mdot y - \pts(x) \le \ptss(y)$
for all $z\in\R^d$, with equality if and only if $y\in \D\pts(x)$
\cite[Thm.~23.5]{Rockafellar}.
The set $\D\pts(x)$ is closed and convex. 

2. Hence,
 the set $C=\{ (y,\ptss(y)): y\in \D\pts(x)\}$ in the graph of $\ptss$ constitutes a {\em face}
of the epigraph of $\ptss$. (See \cite[p.~162]{Rockafellar}.)  This epigraph is the convex hull
of the epigraph of $\psi_t$.  According to \cite[Thm.~18.3]{Rockafellar}, $C$ is the convex hull
of a set $C'\subset C$ such that $C'$ lies in the (epi)graph of $\psi_t$. Moreover, all extreme points
of $C$ lie in $C'$ by \cite[Corollary 18.3.1]{Rockafellar}. 

3. Let $(y_0,\psi_t(y_0))\in C'\subset C$. Then $y_0\in\Theta_t$.
Since every point of $C$ is a convex combination of points in $C'$
by Caratheodory's theorem on convex sets, 
we infer every point $y\in\D\pts(x)$ is a convex 
combination of points in $\Theta_t\cap\D\pts(x)$, hence the result.
\end{proof}

\subsection{Convexified transport maps redux}

In this section,
we extend  our analysis of the transport maps $y\mapsto T_t(y)$ 
by establishing backward propagation properties of 
the touching sets $\Theta_t$ defined in \eqref{d:Ths} 
and ``points of strict convexity'' for $\ptss$.
Some general properties of convexification,
strict convexity, and semi-concavity are developed in Appendix~\ref{a:convex}.
Below, we continue to assume $\vp$ is $K$-Lipschitz,
adding the hypothesis that $\vp$ is semi-concave only for the last result
of this subsection.

Given a convex function $g$ on $\R^d$ we say $g$ is {\em strictly convex at} $y$ 
if there exists $x\in\D g(y)$ such that $z\mapsto g(z)-x\mdot z$ has a strict minimum at $y$.
For each $t>0$ we define the set
\begin{equation}\label{d:Sigmat}
    \Sigma_t = \{y\in\R^d: \ptss \text{ is strictly convex at $y$}\},
\end{equation}
and recall the definition of the touching sets
\begin{equation}\label{d:Thetat}
    \Theta_t = \{y\in\R^d: \ptss(y)=\psi_t(y)\}.
\end{equation}
These sets ``propagate backwards'' in time:

\begin{prop}\label{p:back1} Assume $\vp$ is Lipschitz. Then:
    \begin{itemize}
        \item[(i)]  For each $t>0$, $\Sigma_t\subset\Theta_t$.
        \item[(ii)] $\Theta_s\supset \Theta_t$ whenever $0<s\le t$.
        \item[(iii)] $\Sigma_s\supset \Sigma_t$ whenever $0<s\le t$.
    \end{itemize}
\end{prop}

Here and below, we will use the notation 
\[
f(z)\le_{z=y} g(z)
\]
to indicate that $f(z)\le g(z)$ for all $z$ but equality holds when $z=y$.

\begin{proof}
    (i) This follows from Proposition~\ref{p:fprop}(i) with $f=\psi_t$.

    (ii)
    Fix $0<s<t$ and let $y\in\Theta_t$. Then there exists an affine function
supporting $\ptss$ at $y$, meaning for some vector $v_t\in\R^d$ and some $h_t\in\R$,
\[
v_t\mdot z + h_t \le_{z=y} \ptss(z) 
\le_{z=y} \psi_t(z) = \tfrac12|z|^2+t\vp(z).
\]
Since the function $\psi_0(z)=\frac12|z|^2$ is convex, for some $v_0$ and $h_0$ we have
\[
v_0\mdot z + h_0 \le_{z=y} \psi_0(z).
\]
Taking a convex combination of these inequalities, we find an affine function 
supporting $\psi_s^{**}$ at $y$, with 
\[
v_s\mdot z + h_s
\le_{z=y}   \psi_0(z) + s\vp(z) = \psi_s(z),
\]
where $tv_s = s v_t + (t-s)v_0$ and $th_s = s h_t+(t-s) h_0$.
The left-hand side is affine in $z$ so must also agree at $y$ with the convexification
$(\psi_0+s\vp)^{**} = \psi_s^{**}$.
Hence $y\in\Theta_s$, proving (ii).

(iii) Fix $0<s<t$ and let $y\in\Sigma_t$ be a point of strict convexity for $\ptss$. 
\[
\psi_s(z) = \frac12|z|^2 + s\vp(z) = 
\frac{s}{t}\psi_t(z) + \frac{t-s}{2t}|z|^2
\ge_{z=y}
\frac{s}{t}\ptss(z) + \frac{t-s}{2t}|z|^2\,.
\]
Because the right-hand side is convex, and $y\in\Theta_s$ by parts (i) and (ii), it follows
\begin{equation}\label{e:psss_ge}
\psi_s^{**}(z)\ge_{z=y} 
\frac{s}{t}\ptss(z) + \frac{t-s}{2t}|z|^2\,.
\end{equation}
The right-hand side is clearly strictly convex at $y$, so $y\in\Sigma_s$.
\end{proof}

Next we establish several properties of the transport maps $T_t=\nabla\ptss$
that involve the touching sets and points of strict convexity.

\begin{prop}\label{p:back2}
Assume $\vp$ is $K$-Lipschitz. 
Then for each $t>0$, 
\begin{itemize}
    \item[(i)]  
    $|T_t(y)-y|\le Kt$ for all $y\in\R^d$ where $T_t(y)=\nabla\ptss(y)$ exists.
    \item[(ii)] 
For each $y\in\Theta_t$, 
if $\nabla\vp(y)$ exists then 
\ $ T_t(y)=\nabla\psi_t(y) = y+t\nabla\vp(y). $
\item[(iii)]
For each $y\in \Sigma_t$,  if $\nabla\vp(y)$ exists then $\nabla w_t(x)$ exists 
at $x=T_t(y)$ and we have  
\[
y=\nabla w_t(x) 
\quad\text{and}\quad
\nabla u_t(x)=\nabla\vp(y).
\]
\item[(iv)] 
For each $y\in \Sigma_t$,  if $\nabla\vp(y)$ exists then for $x=T_t(y)$,
the pre-image $T_t\inv(x)=\{y\}$, a singleton.
\end{itemize}
\end{prop}
\begin{remark} In parts (ii) and (iii) of this proposition, the stated properties 
``propagate backwards'' to hold for the same $y$ for all $s\in(0,t]$ by Proposition~\ref{p:back1}.
\end{remark}
\begin{remark}\label{r:diff}
    If we assume $\vp$ is $\lambda$-concave with $\lambda\ge0$, then $\vp$ is differentiable
    at every point $\Theta_t$, by Proposition~\ref{p:fprop}(iii), so the differentiability hypothesis
    in parts (ii) and (iii) holds automatically.
\end{remark}

\begin{proof} 
To prove (i), let $y\in\R^d$ and suppose $x=T_t(y)=\nabla\ptss(y)$ exists.
Then
$y\in\D\pts(x)$, hence $y\in \overline{B(x,Kt)}$ 
by Proposition~\ref{p:speed}.

For part (ii), let $y\in\Theta_t$ and $x\in\D\ptss(y)$. Then 
for all $z$ we have
\[
\psi_t(z)\ge \ptss(z)\ge \psi_t(y) + x\mdot(z-y).
\]
Because $\psi_t$ is differentiable at $y$, necessarily $\ptss$ is also and  
\[
x=\nabla\psi_t(y)=y+t\nabla\vp(y)=T_t(y).
\]
For part (iii), let $y\in \Sigma_t$ and recall $\Sigma_t\subset\Theta_t$.
By part (ii), $x=T_t(y)=\nabla\ptss(y)$ exists and 
$y\in\D\pts(x)$. By the Young identity 
\[
\pts(x) = x\mdot y - \ptss(y) = \sup_z \bigl(x\mdot z -\ptss(z)\bigr),
\]
and the maximum is achieved only at $y$ by strict convexity. 
Then it follows $\D\pts(x)$ is the singleton $\{y\}$.
By \cite[Thm.~25.1]{Rockafellar} we infer that
$w_t=\pts$ is differentiable at the point $x=\nabla\ptss(y)=y+t\nabla\vp(y)$, 
and using \eqref{e:uw_dual} we have
\[
y= \nabla w_t(x) = x-t\nabla u_t(x) = x - t\nabla\vp(y).
\]
This proves (iii). 

For part (iv), assume $x=T_t(y)$ for some $y\in\Sigma_t$, and suppose 
$x=T_t(\hat y)=\nabla\ptss(\hat y)$ for some $\hat y\in\R^d$. 
Then $\hat y\in\D w_t(x)$ so $\hat y=y$ by part (iii).
\end{proof}

As a final result in this section, assuming semi-concavity we establish
an analog of Lemma~\ref{lem:calSt} for sets transported by the maps $T_t$.

\begin{prop}\label{p:Sigt123}
    Assume $\vp$ is $K$-Lipschitz and $\lambda$-concave with $\lambda\ge0$.
    Let $t>0$ and let
    \begin{align*}
    \Sigma_t^{\rm in} &= \{ y\in\R^d: \nabla T_t(y) \text{\ exists and is invertible}\},
    \\
    \Sigma_t^{\rm sg} &= \{ y\in\R^d: \nabla T_t(y) \text{\ exists and is singular}\},
    \\
    \Sigma_t^{\rm nd} &= \{ y\in\R^d: \nabla T_t(y) \text{\ does not exist}\}.
    \end{align*}
    Then 
    \begin{equation}\label{e:Sig0}
    |\Sigma_t^{\rm nd}|=0, \quad
    |T_t(\Sigma_t^{\rm nd})|=0, \quad\text{and}\quad
    |T_t(\Sigma_t^{\rm sg})|=0. 
    \end{equation}
    Also,  $\Sigma_t^{\rm in}\subset \Sigma_t$, and 
   $T_t(\Sigma_t^{\rm in})$ and $T_t(\Sigma_t)$ have full Lebesgue measure in $\R^d$.
\end{prop}
\begin{proof}
    The proof of \eqref{e:Sig0} goes the same as step 1 of the proof of Lemma~\ref{lem:calSt}
    since $T_t$ is Lipschitz. (Note we do not expect $|\Sigma_t^{sg}|=0$ in general.)
    Since $T_t$ is surjective and 
    $\R^d=\Sigma_t^{\rm in}\cup \Sigma_t^{\rm sg}\cup \Sigma_t^{\rm nd}$, it follows
    the complement of $T_t(\Sigma_t^{\rm in})$ has zero Lebesgue measure.
    Finally, because $T_t$ is the gradient of the $C^1$ convex function $\ptss$, 
    whenever the symmetric matrix $\nabla T_t(y)$ exists and is invertible
    then $\ptss$ is strictly convex at $y$, so $\Sigma_t^{\rm in}\subset\Sigma_t$.
\end{proof}

\subsection{Absolutely continuous parts and \MA\ equation} 
\label{ss:consequences}

In this subsection, we show that the inverses of the convexified transport maps $T_t$ agree with
those of the Lagrangian flow maps $X_t$ at almost every (Eulerian) image point.
As a consequence, the past history of almost every Lagrangian path
is one of free streaming (ballistic motion), and 
the absolutely continuous parts of the \MA\ measure $\mam_t$ 
and the mass measure $\rho_t$ must be the same.

Moreover, we show that these parts are determined by the \MA\ equation 
\eqref{e:MA1} in which the Hessian $\nabla^2 w$ is interpreted in the sense of 
Alexandrov. We recall that Alexandrov's theorem (see  \cite[p.~242]{EvansGariepy}) 
states that if $f:\R^d\to\R$ is convex, 
then for a.e. $x\in\R^d$ a symmetric matrix $\nabla^2 f(x)$ exists 
such that 
\begin{equation}\label{e:Alex2nd}
f(x+h) = f(x)+\nabla f(x)\cdot h + \tfrac12 h\cdot \nabla^2 f(x) h + o(|h|^2)
\qquad \text{as}\quad |h|\to0.
\end{equation}

\begin{prop}\label{p:XT1}
    Assume $\vp$ is $K$-Lipschitz and $\lambda$-concave. Let $t>0$. 
    Then for each $y\in\Sigma_t$ we have 
      $X_s(y) = T_s(y) = y+ s\nabla\vp(y)$ {whenever $0\le s\le t$.}
\end{prop}
\begin{proof}
Let $t>0$ and $y\in\Sigma_t$. Then whenever $0<s\le t$,
$y\in\Sigma_s\subset\Theta_s$ by Proposition~\ref{p:back1},
and $\nabla\vp(y)$ exists due to Remark~\ref{r:diff}. 
Also, from Proposition~\ref{p:back2} we infer 
$T_s(y)=y+s\nabla\vp(y)$ and 
$u_s$ is differentiable at $T_s(y)$ with $\nabla u_s(T_s(y))=\nabla\vp(y)$.
    By consequence, for $0\le s\le t$  the function $s\to T_s(y)$ is Lipschitz  and satisfies 
    \[
    \D_s T_s(y) = \nabla\vp(y) = \nabla u_s(T_s(y)) \in \D u_s(T_s(y))\,, \quad T_0(y)=y\,.
    \]
    By the uniqueness argument in Lemma~\ref{l:uniq} for solutions of the differential 
    inclusion, and the uniqueness assertion in Theorem~\ref{t:UNIQ}, we conclude $T_s(y)=X_s(y)$ for all
    $s\in[0,t]$.
\end{proof}

\begin{theorem}\label{t:acMA} 
Let $t>0$. For a.e.~$x\in\R^d$ we have  $X_t\inv(x)=T_t\inv(x)=\{y\}$
and 
    \[     X_s(y) = T_s(y) = y+ s\nabla\vp(y) \quad\text{for $0\le s\le t$,} \]
where $y\in\Sigma_t$.
Moreover,
    \[
    \rhoac = \mamac\ ,
    \]
    and the density of these measures (denoted the same) satisfies the \MA\ equation
   \[
   \rhoac = \det \nabla^2 w_t \qquad\text{ Lebesgue-a.e.\ in $\R^d$,} 
   \] 
   where the Hessian of the strictly convex function $w_t=\pts$ is taken in the sense of Alexandrov.
\end{theorem}
\begin{proof}
    1. Recall $T_t(\Sigma_t)$ has full Lebesgue measure in $\R^d$ by Proposition~\ref{p:Sigt123}.
    By invoking Corollary~\ref{cor:rhoac}, we can say that $\rhoac = \rho_t\mres\calR_t$
    where we can take the set $\calR_t$ to be 
    \[
    \calR_t =  T_t(\Sigma_t) \cap  \{x\in\R^d: \#X_t\inv(x)=1\}
    \cap  X_t(S_t^{\rm sg})^c \, ,
    \]
    since then $|\calR_t^c|=0$. 
    By Proposition~\ref{p:back2}(iv), $T_t\inv$ is single-valued on $\calR_t$, and clearly $X_t\inv$
    is also. Since for each $x\in\calR_t$ we have $x=T_t(y)$ with $y\in\Sigma_t$, 
    the results of Proposition~\ref{p:XT1} yield
    the stated conclusions regarding $X_t\inv(x)$ and $X_s(y)$ for $0\le s\le t$.

    2. Since $T_t\inv$ and $X_t\inv$ agree on $\calR_t$,
    for any Borel set $B\subset \calR_t$ we find 
    \[
    \mam_t(B)=|T_t\inv(B)| = |X_t\inv(B)|=\rho_t(B)=(\rho\mres\calR_t)(B) \,.
    \]
    This shows that $\mam_t\mres\calR_t = \rho\mres\calR_t$,  which equals $\rhoac$
    and is absolutely continuous with respect to Lebesgue measure. Since the complement
    $\calR_t^c$ has Lebesgue measure zero, it follows 
    the absolutely continuous part of $\mam_t$ is 
    $\mamac = \mam_t\mres\calR_t = \rhoac$.
    
3. According to a well-known theorem concerning the Lebesgue decomposition
of a locally finite measure \cite[Thm.~3.22]{Folland1999},
the density of the measure $\mamac$ is given a.e. by the symmetric derivative
$D\mam_t$, defined through
     \[ 
D\mam_t(x) = \lim_{r\to0} \frac{\mam_t(B(x,r))}{|B(x,r)|}, 
\]
at points where the limit exists. 
In his proof of Corollary~4.3 of \cite{McCann97},
McCann shows that for any convex function $\psi$ on $\R^d$, the absolutely
continuous part of the measure $\omega=(\nabla\psi^*)_{\sharp}\leb$ has
density $D\omega = \det\nabla^2\psi$ Lebesgue a.e.,
in terms of the Alexandrov Hessian of $\psi$.
Taking $\psi = w_t = \pts$, since $\mam_t = (\nabla \ptss)_\sharp\leb$
we obtain the claimed \MA\ equation.
\end{proof}

\subsection{Transport by collision-free modification}\label{ss:shock-free}

As promised in the introduction, here we explain how, for {\em fixed} $t>0$, 
the \MA\ measure $\mam_t=(T_t)_\sharp\leb$
can be obtained from a modified Lagrangian flow 
$(\breve X_s)_{0\le s\le t}$ whose particle paths are all straight lines
and remain {\em collision-free} for $0\le s<t$.  
These paths will correspond to a modified velocity potential $\breve u_s$ 
determined by the simple prescription that its initial data $\breve\vp$
is determined by the relation
\begin{equation}\label{d:vpbreve}
 \ptss(y)=    \tfrac12|y|^2+t\breve\vp(y), \qquad y\in\R^d.
\end{equation}
The modified velocity potential $\breve u_s$ is then given by the Hopf-Lax formula
\begin{equation}\label{e:HLmod}
\breve u_s(x) = \inf_y \frac{|x-y|^2}{2s} + \breve\vp(y)\,,
\end{equation}
which provides the viscosity solution of the initial-value problem
\begin{equation}
\D_s \breve u_s + \tfrac12|\nabla \breve u_s|^2 =0\,,\qquad \breve u_0 = \breve \vp.
\end{equation}

The key properties of the modified potentials 
are summarized as follows. Similar to the unmodified case, define
\begin{equation}\label{d:brevepsi}
        \breve\psi_s(y)=\tfrac12|y|^2+s\breve\vp(y), \quad 
        \breve w_s(x) = \breve\psi_s^*(x) = \sup_y x\mdot y-\breve\psi_s(y)\,,
\end{equation}
and note
        \[
        \breve w_s(x) = \tfrac12|x|^2 - s\breve u_s(x)\,, \quad x\in\R^d.
        \]
\begin{prop}\label{p:breve}
    Assume $\vp$ is $K$-Lipschitz and $\lambda$-concave with $\lambda\ge0$.
    Then 
    \begin{enumerate}[(i)]
        \item  $\breve\vp$ is $C^1$, $K$-Lipschitz, and $\lambda$-concave.
        \item For each $s>0$, $\breve u_s$ is $K$-Lipschitz and $\lambda_s$-concave, 
        $\lambda_s = \lambda/(1+\lambda s)$.
        Moreover $\breve w_s$ is strictly convex.
        \item For $0\le s<t$, 
       the function $\breve \psi_s$ is strictly convex, and 
       $\breve w_s$ and $\breve u_s$ are $C^1$.
        \item For all $s\ge t$, $\breve u_s = u_s$.
    \end{enumerate}
\end{prop}
As previously, the modified potentials determine a modified Lagrangian flow 
as the solution to  the differential inclusion initial-value problem
\begin{equation}
    \D_s \breve X_s(y) \in \D \breve u_s(X_s(y))\,, \quad \breve X_0(y)=y.
\end{equation}

    The main results of this section are stated as follows.
\begin{theorem}\label{t:breve}
    Assume $\vp$ is $K$-Lipschitz and $\lambda$-concave with $\lambda\ge0$.
    Fix $t>0$ and let $\breve\vp$, $\breve u_s$, $\breve X_s$ be determined as above. 
    Then:
    \begin{enumerate}[(i)]
        \item For $0\le s<t$, the map $\breve X_s:\R^d\to\R^d$ is bijective, with 
        \[ \breve X_s(y)=y+s\nabla\breve\vp(y) \quad\text{ for all $y$}.
        \]
        \item At time $t$, $\breve X_t(y) = T_t(y)$ for all $y\in\R^d$, and 
        with $\breve \rho_s=(\breve X_s)_\sharp\leb$ we have
        \[
        \breve\rho_t = \mam_t\,.
        \]
    \end{enumerate}
\end{theorem}

We proceed to first prove Proposition~\ref{p:breve}, except that we postpone
the proof that $\breve\vp$ is $\lambda$-concave in part (i).

\begin{proof}[Proof of Prop.~\ref{p:breve}]
(i) By Proposition~\ref{p:Tbasic}, $\breve\vp$ is $C^1$, and
by Proposition~\ref{p:back2}(i) it follows $|\nabla\breve\vp(y)|\le K$ for all $y$, so 
$\breve\vp$ is $K$-Lipschitz. This proves (i) except for the $\lambda$-concavity,
to be proved below.

The properties of $\breve u_s$ in
part (ii) follow by applying Lemmas~\ref{l:lambda} and \ref{p:lim_wueps}
with $\breve\vp$ in place of $\vp$. To see $\breve w_s$ is strictly convex, note that
\[
\breve w_s(x) = \frac1{1+\lambda s} \frac{|x|^2}2 + 
s\left(\frac{\lambda}{1+\lambda s}\frac{|x|^2}2 - \breve u_s(x)\right), 
\]
and the last expression in brackets is convex.

For part (iii), since $\ptss$ is convex it follows 
from the definition in \eqref{d:brevepsi} that 
$\breve\psi_s$ is strictly convex for $0\le s<t$.  
Then the subgradient $\D \breve w_s(x)=\D\psi_s^*(x)$ is a singleton for each
$x$, whence it follows that $\breve w_s$ (and also $\breve u_s$)
is $C^1$ by~\cite[Cor.~25.5.1]{Rockafellar}.

Finally, to prove (iv) observe that $\breve\psi_t =\ptss$. This implies
$\breve w_t=(\ptss)^*=w_t$, which entails $\breve u_t=u_t$. The Hopf-Lax
semigroup property then implies $\breve u_s=u_s$ for all $s>t$.
\end{proof}

\begin{proof}[Proof of Theorem~\ref{t:breve}] (i) For $0\le s<t$, the function
$\breve\psi_s^{**}=\breve\psi_s$ is strictly convex, so for this function,
$\breve\Sigma_s=\R^d$ is the set of points of strict convexity.
Then the formula $\breve X_s(y)=y+s\nabla\breve\vp(y)$ follows by 
Proposition~\ref{p:XT1}. Moreover, $\breve X_s = \breve T_s:=\nabla\breve\psi_s^{**}$
is injective by the strict convexity of $\breve\psi_s$, and $\breve X_s$ is surjective
by Theorem~\ref{t:UNIQ}.

To prove (ii), note that by continuity of $s\mapsto \breve X_s(y)$,
for all $y$ we get 
\[
\breve X_t(y)=y+t\nabla\breve\vp(y) = \nabla\breve\psi_t(y)=\nabla\ptss(y) = T_t(y).
\]
Then the pushforward formula $\breve\rho_t=\mam_t$ follows automatically.
\end{proof}

\begin{remark}
    The measures $\breve\rho_s$ ($0<s<t$) 
    interpolate between Lebesgue measure $\leb$ and the \MA\ measure $\mam_t$
    in a way analogous to displacement interpolants between finite measures
    in optimal transport theory.
\end{remark}

To conclude this section we note that $\lambda$-concavity of $\breve\vp$
follows by showing $\breve\psi_t = \ptss$ is $(1+\lambda t)$-concave.
For this we need a lemma.

\begin{lemma}\label{lem:Gineq}
    Let $y_1,\ldots,y_n\in\R^d$ and let $G(z)=\min_i |z-y_i|^2$. Suppose $y=\sum_i c_iy_i$ with 
    $c_i\ge0$ for all $i$ and $\sum_ic_i=1$. Then $G^{**}(z)\le|z-y|^2$ for all $z$.
\end{lemma}
\begin{proof}
    Let $z$ be arbitrary and define $w= z-y$ and  $z_i=y_i+w$, so $\sum_i c_i z_i=z$.
    Since $G^{**}$ is convex  and  $G^{**}(z_j)\le |z_j-y_i|^2$  for all $i$ 
    (hence for $i=j$), we get 
    \[
    G^{**}(z) \le  \sum_j c_j G^{**}(z_j) \le \sum_j c_j |z_j-y_j|^2 = |w|^2 = |z-y|^2.
    \qedhere
    \]
\end{proof}

\begin{prop}\label{p:1ltconcave}
    Assume $\vp$ is $K$-Lipschitz and $\lambda$-concave with $\lambda\ge0$.
    Then for each $t>0$, $\ptss$ is $(1+\lambda t)$-concave.
\end{prop}
\begin{proof}
It suffices to prove that for all $y\in\R^d$, $\ptss(z)\le_{z=y} P(z)$ for all $z$, where
$P$ is quadratic with Hessian $\nabla^2 P= (1+\lambda t)I$.
    Let $y\in\R^d$, and let $x=\nabla\ptss(y)$. By Lemma~\ref{lem:convexcomb} we may
write $y=\sum_i c_iy_i $ where $y_i\in \D\pts(x)\cap\Theta_t$ and $c_i\ge0$ with $\sum_i c_i=1$.
Note $\nabla\ptss(y_i)=\nabla\ptss(y)$ for all $i$, and $\ptss$ is affine on the convex hull
of the $y_i$, so 
\[
\ptss(y_i) = \ptss(y)+\nabla\ptss(y)(y_i-y)
\]
for all $i$. Then since $\psi_t$ is $(1+\lambda t)$-concave,
\begin{align*}
\psi_t(z)  & \le  \psi(y_i)+\nabla\psi(y_i)\mdot(z-y_i)+\tfrac12(1+\lambda t)|z-y_i|^2.
\\ & = \ptss(y_i)+\nabla\ptss(y_i)\mdot(z-y_i)+\tfrac12(1+\lambda t)|z-y_i|^2.
\\ &= \ptss(y)+\nabla\ptss(y)\mdot(z-y)+\tfrac12(1+\lambda t)|z-y_i|^2.
\end{align*}
With $G(z)=\min_i|z-y_i|^2$, it follows 
\[
\psi_t(z)\le \ptss(y)+\nabla\ptss(y)\mdot(z-y) + \tfrac12(1+\lambda t) G(z) 
\]
for all $z$, whence we infer by passing to the convexification and using Lemma~\ref{lem:Gineq},
\begin{align*}
\ptss(z)&\le \ptss(y)+\nabla\ptss(y)\mdot(z-y)+\tfrac12(1+\lambda t) G^{**}(z) 
\\ & \le \ptss(y)+\nabla\ptss(y)\mdot(z-y)+\tfrac12(1+\lambda t) |z-y|^2 :=P(z).
\qedhere
\end{align*}

The use of Proposition~\ref{p:1ltconcave} completes the proof of part (i)
of Proposition~\ref{p:breve}.
\end{proof}

\section{Three-sector velocity in two dimensions}\label{s:example}

To see that there can be a difference between the \MA\ measures 
$\mam_t=(T_t)_\sharp\leb$  and the  limiting mass measures 
$\rho_t = (X_t)_\sharp\leb$ advected by the adhesion-model velocity, 
we study a special class of examples in $d=2$ space dimensions.
The initial velocity 
will chosen to be constant in three sectors of the plane, so that mass concentrates
immediately into filaments along the boundaries.
We shall develop a simple criterion to determine when the convexified transport flow
$T_t$ is ``sticky'' and when it is {\em not}, and show that in the latter case 
the \MA\ measures  $\mam_t$ {\em differ} from the limiting mass measures 
$\rho_t$ in their singular parts.
We expect similar examples can be constructed in higher space dimensions,
as there does not seem to be anything special about the planar case.

{\bf Initial data.}
Let $v_1, v_2, v_3$ be distinct and non-collinear vectors in $\R^2$. 
We take the initial velocity potential $\vp$ to be piecewise linear and \emph{concave}, given by 
\begin{equation}\label{d:3vp}
\vp(y) = \min_{i=1,2,3} v_i\cdot y \,.
\end{equation}
By consequence of Proposition~\ref{p:Tbasic}, the  convex function $\ptss=w_t^*$ is $C^1$ 
and its gradient  $T_t = \nabla w_t^*$ is 1-Lipschitz.

The initial velocity $v_0=\nabla\vp$ then takes the value $v_i$ in the sector 
$A_i^0$  where
\begin{equation}\label{d:Ai0}
A_i^0= \{y\in\R^2:
\ \ v_i\cdot y<v_j\cdot y \quad\text{\ for all } j\ne i \}\,.
\end{equation}
Sectors $A^0_i$ and $A^0_j$ meet along the ray $R^0_{ij}$ ($=R^0_{ji}$) for which
\begin{equation}\label{d:R0ij}
R^0_{ij} = \{y\in\R^2: 
\ \ v_i\cdot y = v_j\cdot y < v_k\cdot y\quad  \text{ for }k\ne i,j\}\,.
\end{equation}
The vector $n_{ij} = v_i-v_j$ ($=-n_{ji}$) is normal to this ray, and 
points from $A^0_i$ to $A^0_j$ since 
$n_{ij}\cdot y<0$ in $A^0_i$ and  $n_{ij}\cdot y>0$ in $A^0_j$.
The cyclic sum 
\begin{equation}\label{e:cyclicn}
\sum_{\rm cyc} n_{ij}= n_{12}+n_{23}+n_{31}=0.
\end{equation}
We take $v_1, v_2, v_3$ to traverse the vertices of the triangle $\Triang$ 
that they determine in the counterclockwise direction, relabeling if necessary.
With this orientation, whenever $i,j,k$ are increasing in cyclic order, 
we have
\[
(v_j-v_i)\times(v_k-v_j) = Jn_{ji}\cdot n_{kj}>0, \qquad 
J=\begin{pmatrix} 0&-1\\1&0 \end{pmatrix}.
\]
Then by \eqref{d:R0ij}, the vector $\tau_{ij}= Jn_{ji}=-Jn_{ij}$ lies on the ray 
$R^0_{ij}$ for $i,j$ in cyclic order, i.e., $(i,j) = (1,2)$, (2,3) or (3,1). 
Thus 
\begin{equation}
    R^0_{ij}=\{s\tau_{ij}:s>0\}, \quad \tau_{ij}=J(v_j-v_i).
\end{equation}

\subsection{Potentials, subgradient and \MA\ measure}
For given $x,t$, the quantity $\frac12|y-x|^2+tv_i\cdot y$ is minimized at
$y=x-v_i t$ and takes the value $tv_i\cdot x-\frac12|tv_i|^2$.
Hence by the Hopf-Lax formula \eqref{e:HLOmega1} we find 
\begin{align}
t u_t(x) 
& = \min_i \left(tv_i\cdot x-\tfrac12|tv_i|^2\right) \,.
\end{align}
Note that $u_t$ is concave for each $t>0$. By the relation \eqref{e:uw_dual} it follows that
\begin{equation}
    w_t(x) = \pts(x) = \tfrac12|x|^2-tu_t(x) = \max_i \tfrac12|x-tv_i|^2 \,.
\end{equation}
The maximum occurs for a single value of $i$ in three open sectors given by 
\begin{equation} \label{d:Ait}
A_i^t := \{x \in\R^2: \ \ |x-tv_i|^2>|x-tv_j|^2 \ \text{ for all $j\ne i$}\} .
\end{equation}
The set $A_i^t$ is the intersection of half-planes bounded by the
perpendicular bisector (parallel to $R^0_{ij}$)
of the line segment joining the points $tv_i$ and $tv_j$. 
The three bisectors meet at a single point $t\tripl$, where $\tripl\in\R^2$ 
is the \emph{circumcenter} of the triangle with vertices at $v_1$, $v_2$, $v_3$.
That is, $\tripl$
is the center of the circle on which the points $v_1$, $v_2$, $v_3$ lie,
the unique point where
\begin{equation}\label{e:tripl3}
    |v_1-\tripl| =  |v_2-\tripl| =  |v_3-\tripl| .
\end{equation}
As sets, we have the relation 
\begin{equation}\label{e:tAi}
A^t_i = A^0_i+t\tripl = t(A^0_i+\tripl)\,,
\end{equation}
since for $x=y+t\tripl$ the conditions in \eqref{d:Ai0} and \eqref{d:Ait} are equivalent
due to \eqref{e:tripl3}, and $A^0_i=tA^0_i$.
The common boundary of $A_i^t$ and $A_j^t$ is the set $R_{ij}^t\cup\{t\tripl\}$
where 
\begin{equation}\label{d:Rijt}
R_{ij}^t =  R^0_{ij} + t\tripl = t(R^0_{ij}+\tripl) \,. 
\end{equation}
Below, $\co S$ denotes the convex hull of the set $S$.
We write 
\begin{equation}\label{d:co}
[v_i,v_j]=\co\{v_i,v_j\} \,,\quad
\Triang=\co\{v_k:k=1,2,3\} \,, 
\end{equation}
to denote respectively
the line segment joining $v_i$ and $v_j$ and 
the triangle with vertices $v_k$, $k=1,2,3$.
Recall if $S\subset\R^2$ is Borel then $|S|$ denotes its Lebesgue measure.

\begin{prop}\label{prop:Dpts}
For any $x\in\R^2$ and $t>0$, the subgradient $\D w_t(x)$ is 
\begin{align*}
\D w_t(x) = \begin{cases}
\{x-tv_i\}\,, & x\in A^t_i\,,\\
x-t[v_i,v_j] \,,
& x\in R^t_{ij}\,,\\
x-t\Triang\,,
& x = t \tripl \,.
\end{cases}
\end{align*}
Moreover, the Monge-Amp\`ere measure $\mam_t$ associated with $w_t$
is determined by the following, for any Borel set $B$ in the plane:
\begin{align} \label{e:ma_3sector}
\mam_t(B) = \begin{cases}
|B| & 
\text{if }\ B\subset A^t_i\,, \\
t|v_i-v_j|\,\calH_1(B)
& \text{if }\  B\subset R^t_{ij}\,, \\
t^2 |\Triang|
& \text{if }\  B=\{t\tripl\}\,.
\end{cases}
\end{align}
Here $\calH_1$ is one-dimensional Hausdorff measure. 
\end{prop}

\begin{proof}
1. In the region $A^t_i$, $w_t$ is smooth 
with $\nabla w_t(x)=x-tv_i$, hence 
$\D w_t(x)$ is the singleton set containing the point $x-tv_i$.

At a point $x$ on the ray $R^t_{ij}$, the vectors $x-tv_k$, $k\in\{i,j\}$,
are slopes of planes that respectively support the paraboloids 
$z\mapsto\frac12|z-tv_k|^2$ at $x$. 
These planes also support the maximum $w_t(z)$, therefore
the two slopes $x-t v_k$ and their convex hull are contained in $\D w_t(x)$.
On the other hand, if $p\in\R^2$ is not a point in the convex hull, 
it is necessarily separated from it by a line. This means there is a vector $q$
such that $q\cdot (x-tv_k)<q\cdot p$ for $k\in\{i,j\}.$  
For sufficiently small $r>0$ we can conclude that for $k\in\{i,j\}$,
\[
\tfrac12|x+rq-tv_k|^2 - \tfrac12|x-tv_k|^2 = rq\cdot(x-tv_k)+\tfrac12|rq|^2
< rq\cdot p,
\]
and we can conclude that that $p\notin \D w_t(x)$.

By the same argument with $\{i,j\}$ replaced by $\{1,2,3\}$ we infer
\[
\D w_t(t\tripl)=\co\{x-tv_k : k=1,2,3\}.
\]

2. Now the \MA\ measure of $w_t$  is straightforward to compute, as follows.
Observe that for any distinct points $x$, $\hat x$ in the plane,
the subgradients $\D w_t(x)$ and $\D w_t(\hat x)$ are disjoint.
Let $B$ be a Borel set in the plane. 
Since we may decompose $B$ by its intersections with the sets 
$A^t_i$, $R^t_{ij}$ and $\{\tripl t\}$, 
we assume without loss of generality that $B$ is a subset of one of these sets.

First suppose that $B\subset A_i^t$.
Note $A^t_i-tv_i\subset A^0_i$, because when
\[
|x-tv_i|>|x-tv_j|=|(x-tv_i)+t(v_i-v_j)|
\]
for all $j$, it follows $(x-tv_i)\cdot(v_i-v_j)<0$ so that $x-tv_i\in A^0_i$.
Hence $\D w_t(B)=B-tv_i$, whence $\mam_t(B)=|B-tv_i|=|B|$.

Next, suppose $B\subset R^t_{ij}$. Note that $\D w_t(x)=x-t\co\{v_i,v_j\}$
for each $x\in B$. 
Then because the line segments $\co\{v_i,v_j\}$ and the ray $R^t_{ij}$ 
have orthogonal tangents, the Lebesgue measure
\[
|\D w_t(B)| = 
t|v_i-v_j|\,\calH_1(B).
\]

Finally, for the singleton set $B=\{t\tripl \}$ we have 
\begin{equation*}
|\D w_t(B)| = t^2
|\co\{v_1,v_2,v_3\}| = 
\tfrac12 t^2 |(v_1-v_2)\times (v_2-v_3)|. 
\qedhere
\end{equation*}
\end{proof}

\subsection{Transport maps and paths}

Recall the Lagrangian paths $y\mapsto X_t(y)$ are determined as the unique 
Lipschitz solution of the differential inclusion 
\begin{equation}\tag{\ref{e:DXt}}
  \D_t x_t \in \D u_t(x_t) \, \quad \text{for a.e. $t>0$,} \qquad    x_0 = y.
\end{equation}
By Proposition~\ref{prop:Dpts} and  the formula~\eqref{e:Dut1}, 
the differential is given by
\begin{equation}\label{e:Dutex}
    \D u_t(x) = \begin{cases}
        \{v_i\}\,, & x\in A^t_i\,,\\
        [v_i,v_j]\,, & x\in R^t_{ij}\,,\\
        \Triang\,, & x=t\tripl\,.
    \end{cases}
\end{equation}
We will use this characterization to determine the Lagrangian paths, 
but first we look at the simpler transport maps and paths.

\subsubsection{Transport maps and paths}
For any $t>0$, 
the transport map $T_t=\nabla\ptss$ is the inverse of the subdifferential $\D w_t$
as given by Proposition~\ref{prop:Dpts}. Explicitly we find 
\begin{equation}\label{e:Ttex}
T_t(y) = \begin{cases}
    y+tv_i & \text{if } x\in A^t_i \text{ and } y=x-tv_i \,,
    \\
    x & \text{if }  x\in R^t_{ij} \text{ and } y\in x -t[v_i,v_j] \,, \\
    t\tripl & \text{if } y\in t(\tripl-\Triang) \,.
\end{cases}
\end{equation}
See Fig.~\ref{f:riemann} for an illustration of this map for $t=0.3$
in the case when
\[
v_1 = (a,0), \quad v_2 = (b,c) , \quad v_3 = (0,0), \quad 
(a,b,c)= (1.8,2.5,1.5).
\]
\begin{figure}[!t]
\noindent\makebox[\textwidth]{%
\includegraphics[height=2in]{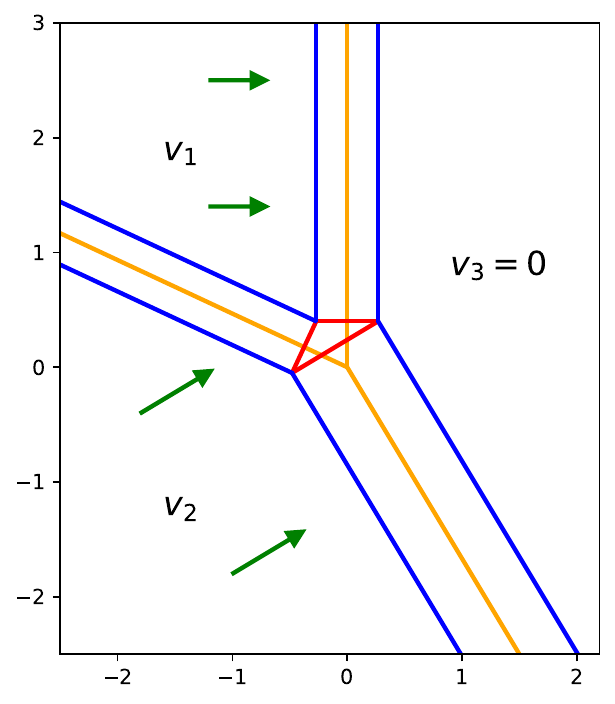}
\ \includegraphics[height=2in]{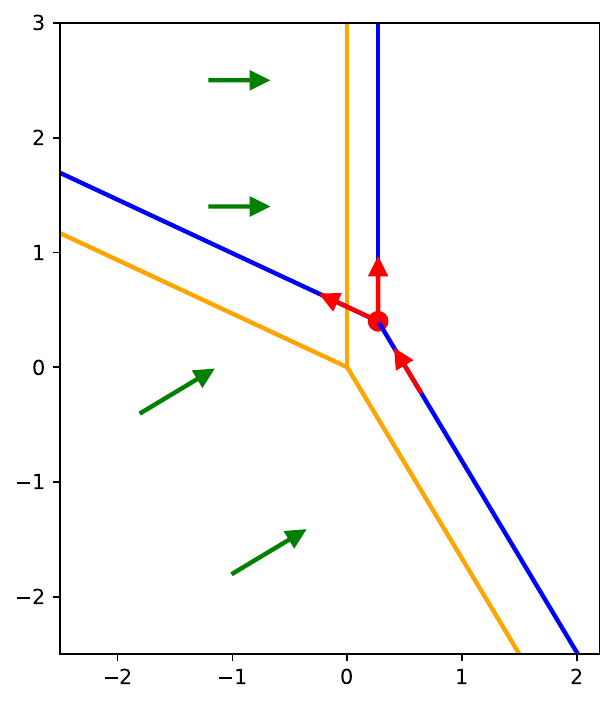}}
\caption{Transport map for 3-sector velocity. 
Left panel: Lagrangian plane. Right panel: Eulerian plane.
Red arrows indicate components of triple-point velocity
$\tripl$ along mass filaments.}  \label{f:riemann} 
\end{figure}

To describe the paths $t\mapsto T_t(y)$, observe
that by \eqref{e:tAi} and \eqref{d:Rijt} we have 
\begin{equation}\label{e:Ttcases}
    T_t(y)\in \begin{cases}
        A^t_i\\ R^t_{ij}\\ \{t\tripl\}
    \end{cases}
    \quad\text{if and only if}\quad
    \frac yt\in 
    \begin{cases}
        \tripl - v_i + A^0_i\\
        \tripl-[v_i,v_j] + R^0_{ij}\\
        \tripl-\Triang
    \end{cases}
\end{equation}
respectively. 
The seven sets to which $y/t$ may belong in \eqref{e:Ttcases}
are disjoint and convex and partition the plane.
As $y/t$ moves along a line toward the origin as $t$ increases,
this means that the set of times $t$ for which  
each of the seven options for $T_t(y)$ occurs
must be an {\em interval} (possibly empty or trivial).

\subsubsection{Velocity field} We claim that the transport paths $t\to T_t(y)$ are 
Lipschitz solutions of a 
differential equation with discontinuous right-hand side. 
\begin{lemma}\label{l:ode}
   For each $y\in\R^d$, the path $t\mapsto T_t(y)$ is continuous and piecewise linear, 
   and at each non-nodal point we have
\begin{equation}
\D_t T_t(y) = V_t(T_t(y))\,,
\end{equation}
where 
\begin{equation}
    V_t(x) = \begin{cases}
        v_i & \text{ for } x\in A^t_i\,,\\
        v_{ij}=\frac12(v_i+v_j) & \text{ for } x\in R^t_{ij}\,,\\
        \tripl & \text{ for } x=t\tripl\,.
    \end{cases}
\end{equation}
\end{lemma}
\begin{proof} 
    For convenience we write $z_t=T_t(y)$ and $\dot z_t = \D_tT_t(y)$.
Consider $t$ to be in the interior of one of the intervals
that partition the times for which each option in \eqref{e:Ttcases} occurs.
For $z_t \in A^t_i$ ($i=1,2$ or $3$) then $z_t= y+tv_i$ so $\dot z_i = v_i$,
and for $z_t\in t(\tripl-\Triang)$ then $z=t\tripl$ so $\dot z_i=\tripl$. 

   It remains to consider the case $z_t\in R^t_{ij}$. 
   Because $z_t-t\tripl\in R^0_{ij}$ and $z_t-y\in t[v_i,v_j]$, recalling $n_{ij}=v_i-v_j\perp R^0_{ij}$
   we have that
   \[
    n_{ij}\cdot (z_t-t\tripl) = 0 \quad\text{and}\quad
    n_{ij}\times(z_t-y) = 
    n_{ij}\times v \quad\text{for any $v\in[v_i,v_j]$}.
   \]
   Noticing that $n_{ij}\cdot \tripl= n_{ij}\cdot v_{ij}$ (because $\tripl$ lies on the 
   perpendicular bisector of $[v_i,v_j]$) we find that 
   \[
   n_{ij}\cdot(\dot z_t - v_{ij})=0 \quad\text{and}\quad 
   n_{ij}\times(\dot z_t - v_{ij})=0.
   \]
   Hence $\dot z_t=v_{ij}$.

Thus $t\mapsto z_t$ is piecewise linear on $(0,\infty)$, and
it is straightforward to check that it is continuous at nodal points.
\end{proof}

Observe next, that if $\tripl$ lies in the interior of the triangle $\Triang$
then eventually $y/t\in \tripl-\Triang$ for large enough $t$ and $z_t=T_t(y)=t\tripl$.
In particular, when $z_t$ is on any ray $R^t_{ij}$, its velocity $v_{ij}$ 
must force it collide with the endpoint of the ray at $t\tripl$. 
Thus the component of the relative velocity vector $v_{ij}-\tripl$ 
in the direction $\tau_{ij}$ pointing along the ray away from $t\tripl$ must be {\em negative}.
That is, the (unnormalized) quantity $\xi_{ij}<0$ where
\begin{equation}\label{d:xiij}
    \xi_{ij} = (v_{ij}-\tripl)\cdot \tau_{ij} \,.
\end{equation}
What happens when the circumcenter $\tripl$ lies {\em outside} the triangle
$\Triang$ is that points along the ray $R^t_{ij}$ move {\em away} from the endpoint
$t\tripl$, and this will lead to interesting consequences.
See Fig.~\ref{f:stationary} for illustrations in the case $\tripl=0$ when $v_i$
lie on a circle centered at the origin.
\begin{figure}[!t]
\noindent\makebox[\textwidth]{%
\includegraphics[height=1.6in]{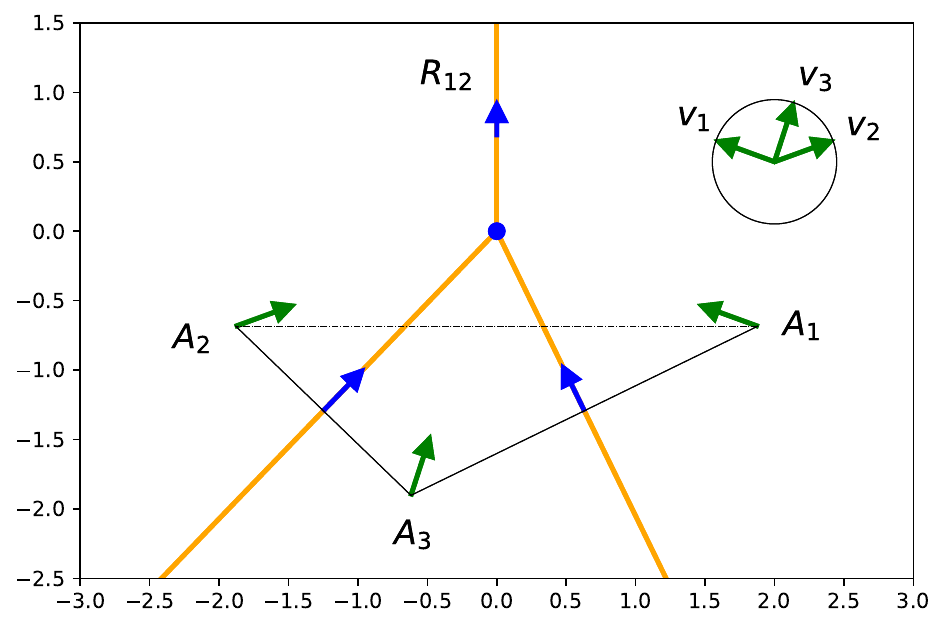}
\ \includegraphics[height=1.6in]{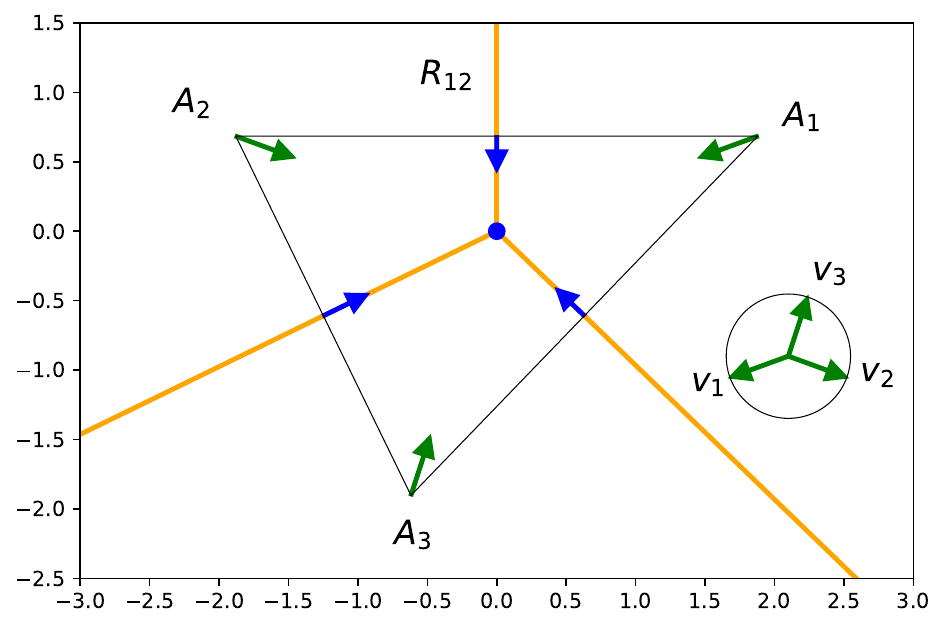}}
\caption{Flow examples.
Insets indicate $v_i$ on a circle with center $\tripl=0$.}  \label{f:stationary} 
\end{figure}

\begin{lemma}\label{l:tripl}
    If $\tripl\in\Triang$, then $\xi_{ij}\le0$ for each pair $(i,j)$ in cyclic order,
    while if $\tripl\notin\Triang$, then $\xi_{ij}>0$ for exactly one such pair,
    the pair for which  
    \[
    |v_{ij}-\tripl|=\dist(\tripl,\Triang).
    \]
\end{lemma}
\begin{proof}
Because $\tau_{ij}=-J(v_i-v_j)$ and both $v_{ij}$ and $\tripl$ 
lie on the perpendicular bisector of the line segment $[v_i,v_j]$, we have
\begin{align*}
\xi_{ij} &= (v_{ij}-\tripl)\cdot \tau_{ij}  = (v_{ij}-\tripl)\times(v_i-v_j) 
\\ &= (v_i-\tripl)\times(v_i-v_j)  =  (\tripl-v_i)\times(v_j-v_i)  \,.
\end{align*}
If $\tripl\in\Triang$ then $\tripl$ lies on or to the {\em left} of each line containing
$[v_i,v_j]$ oriented from $v_i$ to $v_j$ adjacent in cyclic order. 
This yields $\xi_{ij}\le0$.
If $\tripl\notin\Triang$ however, the distance from $\tripl$ to any point $z\in\Triang$
is minimized at $v_{ij}$ for some $(i,j)$. 
For this pair, $\tripl$ lies to the {\em right} of the oriented line, whence $\xi_{ij}>0$.
For the other two lines in this case, $\tripl$ again lies to the {\em left}.
\end{proof}

\subsubsection{Criterion for transport flow to be sticky}
Recall that the Lagrangian flow maps $X_t$ are ``sticky'' in the sense explained in 
Theorem~\ref{t:UNIQ}.
At this point it is easy to determine when the flow maps $T_t$ agree
with the maps $X_t$ and have the ``sticky'' property, and when they do {\em not}.
\begin{prop}\label{p:sticky}
    The following are equivalent:
    \begin{enumerate}[(i)]
      \item The circumcenter $\tripl\in \Triang = \co\{v_1,v_2,v_3\}$.
      \item $V_t(x)\in \D u_t(x)$ for all $x\in\R^d$.
      \item $T_t(y)=X_t(y)$ for all $t\ge0$ and all $y\in\R^d$.
      \item The transport maps $T_t$ are ``sticky,'' meaning that
\[
T_s(y)=T_s(z) \quad\text{implies}\quad T_t(y)=T_t(z)\quad\text{whenever $0\le s\le t$.}
\]
\item For any $t>0$, $V_t$ satisfies a one-sided Lipschitz condition.
    \end{enumerate}
\end{prop}
\begin{proof}
   That (i) is equivalent to (ii) is evident by comparing the formula for 
   $V_t$ in Lemma~\ref{l:ode} with that for $\D u_t$ in \eqref{e:Dutex}.
   Next, (ii) implies (iii), for if (ii) holds,
   then $t\mapsto T_t(y)$ is a 
   Lipschitz solution of the differential inclusion \eqref{e:DXt}, and  
   $T_t(y)=X_t(y)$ follows by the uniqueness lemma~\ref{l:uniq}. 
   Since the family $X_t$
   is ``sticky'' by Theorem~\ref{t:UNIQ}, (iii) implies (iv).
   Also (ii) implies (v) by Lemma~\ref{lem:Dflambda}, for since $\vp$ is concave,
   $u_t$ is concave. 

   It remains to show that each of (iv) and (v) imply (i).
   Suppose (i) fails, so $\tripl\notin \Triang$. First we show (iv) fails,
   i.e., the family $T_t$ ($t\ge0$) is not sticky.
Let $s>0$. The closed triangle $s(\tripl-\Triang)$ does not contain $0$.
Let $y$ be the point of this closed triangle closest to $0$,
and let $\hat y$ be any point in the interior. Then
$T_s(y)=T_s(\hat y)=s\tripl$. But $y\notin t(\tripl-\Triang)$ for any $t>s$,
so for $t-s>0$ small enough we have $t\tripl=T_t(\hat y)\ne T_t(y)$.

Finally, if (i) fails we show (v) fails. Suppose $\tripl\notin\Triang$.
By Lemma~\ref{l:tripl} let $(i,j)$ be in cyclic order such that $\xi_{ij}>0$.
Then for $\hat x=t\tripl$ and $x= t\tripl+ s\tau_{ij}$ with $s>0$ we have $x\in R^t_{ij}$
and 
\[
(V_t(x)-V_t(\hat x))\cdot(x-\hat x) = (v_{ij}-\tripl)\cdot (s\tau_{ij}) =s\xi_{ij}> 0.
\]
But for any $\lambda>0$ this cannot be bounded above by 
$\lambda|x-\hat x|^2 = \lambda s^2|\tau_{ij}|^2$ for all $s>0$.
So $V_t$ does not satisfy a one-sided Lipschitz condition.
\end{proof}

\begin{remark}
Recall that $T_t(y)\in R^t_{ij}$ if and only if $y/t\in S_{ij}$, where the set
\[ S_{ij}:= R^0_{ij}+\tripl -[v_i,v_j] \]
is a half-infinite strip 
bisected by the ray $\{ s\tau_{ij}+\tripl-v_{ij}:s>0\}$.
In case $\tripl\notin\Triang$, the strip $S_{ij}$ for which $\xi_{ij}>0$ 
contains a neighborhood of the origin (since $v_{ij}-\tripl=s\tau_{ij}$ with $s>0$). 
This has the consequence that every path
$t\mapsto T_t(y)$ eventually ends up in the ray $R^t_{ij}$ if $t$ is large enough.
\end{remark}

\subsection{Lagrangian paths and mass flow }

The main result in this subsection concerns when the limiting advected mass measures
$\rho_t = (X_t)_\sharp\leb$ agree with the \MA\ measure $\mam_t = (T_t)_\sharp\leb$,
and when they do {\em not}.

\begin{prop}i\label{p:rhomameq} The following are equivalent:
\begin{enumerate}[(i)]
    \item $\tripl\in \Triang$.
    \item $\rho_t = \mam_t$ for all $t>0$.
    \item $\rho_t = \mam_t$ for some $t>0$.
\end{enumerate}
\end{prop}

Of course, when $\tripl\in\Triang$ then $\rho_t=\mam_t$ for all $t$ since $X_t=T_t$
by Proposition~\ref{p:sticky}. So (i) implies (ii)  and (ii) implies (iii).
Our task remains to show that if (i) fails then (iii) fails.
We will accomplish this by showing later in this section
that if $\tripl\notin\Triang$, then $\rho_t$ contains no delta mass,
while $\mam_t$ always does, by Proposition~\ref{prop:Dpts}.

\subsubsection{Lagrangian paths}
We turn first to describe the Lagrangian paths $t\mapsto X_t(y)$ in case
$\tripl\notin \Triang$.  The key is that these paths cannot ``linger''
at the point $t\tripl$ in this case, but instead are forced to move immediately
onto the ray $R^t_{ij}$ on which the velocity $v_{ij}$ pushes points away from 
the endpoint $t\tripl$.
\begin{lemma}\label{l:tstar}
    Assume $\tripl\notin\Triang$ and choose $(i,j)$ in cyclic order satisfying $\xi_{ij}>0$.
    Let $y\in\R^d$ and let $\tstar=\inf\{t\ge0: T_t(y)=t\tripl\}$.
    \begin{enumerate}[(i)]
        \item If $\tstar=\infty$ (i.e., $T_t(y)\ne t\tripl$ for all $t$)
        then $X_t(y)=T_t(y)$ for all $t\ge0$.
        \item If $\tstar<\infty$ then 
        \[
        X_t(y) = \begin{cases}
            T_t(y) & 0\le t\le \tstar\,,\\
            \tstar\tripl + (t-\tstar)v_{ij} & t>\tstar\,.
        \end{cases}
        \]
    \end{enumerate}
        In all cases, $X_t(y)=t\tripl$ for at most one value of $t$.
        Moreover, if $y/t$ lies in the interior of $\tripl-\Triang$ for some $t>0$,
        then $X_t(y)\ne T_t(y)$ for all $t>\tstar$.
\end{lemma}
\begin{proof}
    If $T_t(y)\ne t\tripl$ for all $t\ge0$,
    then $t\mapsto T_t(y)$ is a Lipschitz solution of the
    differential inclusion in \eqref{e:DXt}, so $T_t(y)=X_t(y)$ for all $t$ by the 
    uniqueness lemma~\ref{l:uniq}.
    On the other hand, if $\tstar<\infty$ then the right-hand side in (ii) satisfies
    the differential inclusion, so agrees with $X_t(y)$ for the same reason.
    If ever $y/t$ is interior to $\tripl-\Triang$, then $T_t(y)=t\tripl$ for all
    $t$ in some open interval with endpoint $\tstar$, with the consequence that 
    $X_t(y)\ne T_t(y)$ for all $t>\tstar$.
\end{proof}

A striking corollary of this result is that
the Lagrangian paths can yield a different solution of the {\em same} differential equation 
with discontinuous right-hand side that is satisfied by the non-sticky transport paths 
$t\mapsto T_t(y)$. Naturally, this is only possible since the velocity field 
fails to satisfy a one-sided Lipschitz condition due to Proposition~\ref{p:sticky}(v).
\begin{prop} \label{p:ode2}
   Let $y\in\R^d$ and let $t>0$ be a point of differentiability of $t\mapsto X_t(y)$.
   Then
\begin{equation}
\D_t X_t(y) = V_t(X_t(y))\,,
\end{equation}
where $V_t(x)$ is defined for all $x$ the same as in Lemma~\ref{l:ode}.
However, if $\tripl\notin\Triang$ then $t$ is never a point of differentiability when $X_t(y)=\tripl$. 
\end{prop}

\subsubsection{Inverse Lagrangian maps}

To understand the mass measures $\rho_t$ when $\tripl\notin\Triang$,
we need to understand the pre-image sets $X_t\inv(x)$.
First we handle the easier cases when $X_t\inv(x)=T_t\inv(x)$.
From \eqref{d:Rijt} and the role of the relative velocity $v_{ij}-\tripl$ 
in \eqref{d:xiij}
we have that 
\begin{align*}
R^t_{ij}=&R^0_{ij}+t\tripl \supset  R^0_{ij}+tv_{ij} \quad\text{if $\xi_{ij}>0$},\\
R^t_{ij}=&R^0_{ij}+t\tripl \subset  R^0_{ij}+tv_{ij} \quad\text{if $\xi_{ij}\le0$}.
\end{align*}
In every case, if $x$ is in the {\em smaller} of the sets 
$R^0_{ij}+t\tripl$ and $R^0_{ij}+tv_{ij}$, then the point $y=x-tv_{ij}\in R^0_{ij}$
and  we have $T_s(y)=y+sv_{ij}\ne s\tripl$ for $0\le s\le t$. Moreover, 
the same is true whenever
\[
y\in T_t\inv(x) = \D w_t(x)=x-t[v_i,v_j],
\]
so in this case we have $X_s(y)=T_s(y)$ for $s\in[0,t]$, 
hence $T_t\inv(x)\subset X_t\inv(x)$.
By the same reasoning, if $x\in A^t_i$ then $X_s(y)=T_s(y)$ for $y=x-tv_i$
and $0\le s\le t$, hence $T_t\inv(x)\subset X_t\inv(x)$.

\begin{lemma}\label{l:easy}
    Assume $\tripl\notin\Triang$. Then for any $x\in\R^d$,
    \[
    X_t\inv(x) = 
    \begin{cases}
       \{x-tv_i\} & \text{if\,\ $x\in A^t_i$\,,} \\
    x-t[v_i,v_j] & 
    \text{if\,\  $x\in (R^0_{ij}+t\tripl) \cap (R^0_{ij}+tv_{ij})$.}
    \end{cases}
    \]
\end{lemma}

\begin{proof}
We claim actually $X_t\inv(x)=T_t\inv(x)$ in each case. 
If not, then there exists $y$ with $x=X_t(y)\ne T_t(y)$.
By Lemma~\ref{l:tstar}, necessarily 
\[
\tstar\tripl  =T_{\tstar}(y)= X_{\tstar}(y)
\quad\text{ for some  $\tstar\in[0,t)$},
\]
and moreover $x=X_t(y)$ must lie on the {\em outflowing} ray $R^t_{ij}$ where $\xi_{ij}>0$.
Thus it cannot be that $x$ lies in $A^t_i$ or on a ray $R^t_{ij}$ where $\xi_{ij}\le 0$.

In the remaining case, if $x\in R^t_{ij}$ with $\xi_{ij}=(v_{ij}-\tripl)\cdot\tau_{ij}>0$, 
it follows 
\[
x-tv_{ij} = \tstar(\tripl-v_{ij}) \notin R^0_{ij} \,,
\]
contradicting the assumption $x\in R^0+tv_{ij}$ in this case. 
\end{proof}

It remains to analyze the cases when $x=t\tripl$ or $x\in R^t_{ij}$ but {\em not}
$R^0_{ij}+tv_{ij}$. Since the endpoints of these rays are 
$t\tripl$ and $tv_{ij}$ respectively, the cases remaining correspond to taking
\begin{equation}\label{e:xcases}
x\in[t\tripl,t v_{ij}], \quad\text{so}\quad x=(t-s)\tripl+sv_{ij}\,, \quad 0\le s\le t. 
\end{equation}
We first consider the endpoints, which will allow us to handle the rest.
\begin{lemma}\label{l:edge} 
Let $t>0$ and take $i,j,k$ in cyclic order with $\xi_{ij}>0$. Then
\begin{enumerate}[(i)]
    \item $X_t\inv(tv_{ij})= x - t[v_i,v_j]$
    \item $X_t\inv(t\tripl) = (x-t[v_j,v_k])\cup (x-t[v_k,v_i])$
\end{enumerate}
\end{lemma}

\begin{proof}
(i) The end of the proof of Lemma~\ref{l:easy} also works when $x=tv_{ij}\in R^t_{ij}$ in case $\xi_{ij}>0$,
so $X_t\inv(x)=x-t[v_i,v_j]$ in that case also.

(ii) For $x=t\tripl$, the segments 
\[
x-t[v_j,v_k]=t[\tripl-v_j,\tripl-v_k], \quad 
x-t[v_k,v_i]=t[\tripl-v_k,\tripl-v_i]
\]
form the sides of the triangle $t(\tripl-\Triang)$ farthest from the origin. 
Thus, for $y$ in either of these segments, we have $T_t(y)=t\tripl$ and $T_s(y)\ne s\tripl$
for all $s<t$. It follows $T_t(y)=X_t(y)$, so both segments are contained in $X_t\inv(x)$.

We claim no other point $y\in X_t\inv(x)$. For any such point, either $y$ lies in $T_t\inv(x)$
or it does not. If it does, then $y$ lies in the triangle $t(\tripl-\Triang)$ but not on the 
farthest sides, so $t>\tstar$ in the notation of Lemma~\ref{l:tstar}.
Then Lemma~\ref{l:tstar} yields $X_t(y)\ne T_t(y)$, 
which is also the case if $y\notin T_t\inv(x)$. 
In both cases, Lemma~\ref{l:tstar} yields that $X_t(y)$ lies on the ray $R^t_{ij}$. 
But this ray does not contain $t\tripl$, contradiction.
\end{proof}

To handle the remaining cases,
the key is to use the semiflow property from Remark~\ref{r:semiflow} to show that  
for $x$ satisfying \eqref{e:xcases},
\begin{equation}\label{e:semiXt}
X_t\inv(x) = X_{t-s}\inv( x - s[v_i,v_j])\,.
\end{equation}
We justify this formula by an argument 
that essentially comes down to using a Galilean transformation.
Notice that from the formula~\eqref{e:Dutex},
the differential $\D u_t$ enjoys the invariance property
\begin{equation}
\D u_t(x) = \D\vp(x-t\tripl).
\end{equation}
Then by invoking the uniqueness lemma for the differential inclusion, we infer that
the semiflow maps are given by 
\begin{equation}
    X_{t,r}(z) = X_{t-r}(z-r\tripl) + r\tripl, \qquad 0\le r\le t, \ \ z\in\R^d.
\end{equation}
Since $X_t = X_{t,r}\circ X_r$, we get 
\begin{equation}
X_t\inv(z) = X_r\inv\circ X_{t,r}\inv(z) =
X_r\inv(X_{t-r}\inv(z-r\tripl) + r\tripl).
\end{equation}
Taking $r=t-s$ and $x=(t-s)\tripl+sv_{ij}$ 
we get 
\begin{equation}
    X_t\inv(x) = X_{t-s}\inv( X_s\inv(sv_{ij})+(t-s)\tripl).
\end{equation}
Invoking Lemma~\ref{l:edge}(i) we have $X_s\inv(sv_{ij})=sv_{ij}-s[v_i,v_j]$,
yielding~\eqref{e:semiXt}. 

\begin{figure}
\begin{center}
\includegraphics[height=2.2in]{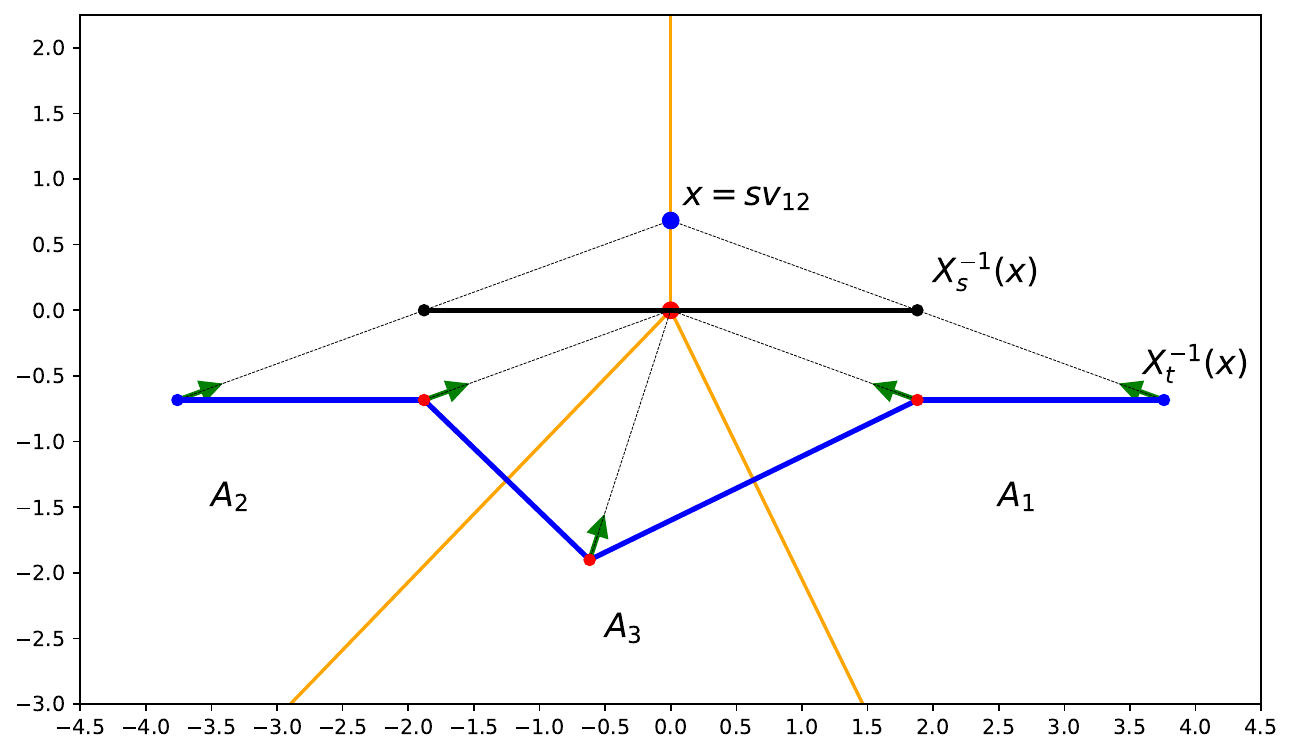}
\caption{Pre-images $X_s\inv(x)$ (black), $X_t\inv(x)$ (blue)
for $x=sv_{12}\in R_{12}$, $t>s$.
Velocities $v_j$ are as in Fig.~\ref{f:stationary}(a).}  \label{f:Xinv} 
\end{center}
\end{figure}
Finally, we obtain the following.
See Fig.~\ref{f:Xinv} for illustration of a case with $\tripl=0$ and $(i,j)=(1,2)$.
\begin{lemma}\label{l:lastcases}
    Let $x=(t-s)\tripl+s v_{ij}$ where $0<s<t$ and $\xi_{ij}>0$. Then 
    \begin{align*}
    X_t\inv(x) &= 
    \quad [x-tv_i, (t-s)(\tripl-v_i)]
    \\ & \quad \cup
    [x-tv_j,(t-s)(\tripl-v_j)]
    \\ & \quad \cup
    [(t-s)(\tripl-v_j), (t-s)(\tripl-v_k)] 
    \\ & \quad \cup
    [(t-s)(\tripl-v_k), (t-s)(\tripl-v_i)].
    \end{align*}
\end{lemma}
\begin{proof}
    Writing $[x,z)=[x,z]\setminus\{z\}$ and similarly for $(x,z]$, 
    since $x-sv_{ij}=(t-s)\tripl$ we can write
    \[
    x-s[v_i,v_j] = [x-sv_i,(t-s)\tripl) \cup \{(t-s)\tripl\} \cup 
    ((t-s)v_j,x-sv_j] \]
    The first and last of these sets lie in $A^{t-s}_i$ and $A^{t-s}_j$
    respectively, and their preimage under $X_{t-s}$ translates them
    by $-(t-s)v_i$ and $-(t-s)v_j$ respectively. For the singleton
    set we invoke the formula in Lemma~\ref{l:edge}(ii).
\end{proof}

\subsubsection{Mass measures that lack deltas}

Recall that the pushforward mass measures $\rho_t=(X_t)_\sharp\leb$
take the values $\rho_t(B)=|X_t\inv(B)|$ for each Borel set $B$.
By direct examination of the results of Lemmas~\ref{l:easy}, \ref{l:edge}
and \ref{l:lastcases}, it is evident that if $\tripl\notin\Triang$,
then the preimage under $X_t$ of any singleton  set $\{x\}$ is a point or a
union of line segments, with Lebesgue measure zero. Thus we find: 

\begin{lemma}\label{l:nodelta}
If $\tripl\notin\Triang$, then for each $t>0$, $\rho_t(B)=0$ for each singleton $B$.
\end{lemma}

With this result, taking into account that the measure $\mam_t(\{t\tripl\})>0$ 
always, even when $\tripl\notin\Triang$, we infer that if $\tripl\notin\Triang$
then $\rho_t\ne\mam_t$ for all $t>0$.
This completes  the proof of Proposition~\ref{p:rhomameq}.

\begin{remark}
    In case $\tripl\notin\Triang$, the measures $\rho_t$ and $\mam_t$
    agree for all sets $B$ that are disjoint from the line segment
    $[t\tripl,tv_{ij}]$ in the (closure of) the outflowing ray
    $R^t_{ij}$ where $\xi_{ij}>0$. 
    The mass  of the delta at $t\tripl$ for $\mam_t$ is 
    redistributed in $\rho_t$ as a piecewise-linear excess line density 
    along this line segment. We omit details.
\end{remark}

\subsection{Impossibility of a.e.-correct monotone reconstruction}\label{ss:impossible}

In this section we use our examples when $\mam_t\ne\rho_t$ to show that 
the strategy for reconstruction implemented in the works of 
Frisch~\etal\,\cite{frisch2002reconstruction} and  Brenier~\etal\,\cite{brenier2003} 
cannot always be correct almost everywhere.

This strategy can be summarized as an approximation of the following ideal:
Given the mass measure $\rho_t$ at the current epoch $t>0$,
find a transport map $\tilde T_t$ which pushes forward Lebesgue measure $\leb$ to $\rho_t$,
just as the Lagrangian flow map $X_t$ does. That is, 
without knowing $X_t$ or the initial velocity, find $\tilde T_t$  so that
\begin{equation}\label{e:agree}
(\tilde T_t)_\sharp\leb = \rho_t = (X_t)_\sharp\leb.
\end{equation}
It is natural to say that $\tilde T_t$ {\em provides a correct reconstruction} at 
a given point $x\in\R^d$ if the pre-image of $x$ under $\tilde T_t$ agrees with that under $X_t$,
i.e., 
\begin{equation}\label{e:correct}
\tilde T_t\inv(x)=X_t\inv(x) \,,
\end{equation}
as sets.
Even ideally, one can expect to achieve correct reconstruction only for points $x$ with
unique pre-image, so for only a.e.~$x$.  As is well explained in \cite{brenier2003} and many other sources, 
when particle paths collide and mass concentrations form in conjunction with shocks in
the velocity field, different past histories of shock development may produce the same
configuration at a given time $t>0$.

The authors of \cite{frisch2002reconstruction} and \cite{brenier2003} aim 
to determine $\tilde T_t$ using  optimal transport principles, equivalent to a least action principle.
In the setup of the present paper, both measures $\leb$ and $\rho_t$ are infinite,
and it is not immediately clear how an optimal transport map should be defined 
so that \eqref{e:agree} holds globally.
But anyway it is reasonable that the measure $\rho_t$ can be determined by measurements
only in some large but finite region of $\R^d$.  So some truncation needs to be done.
E.g., numerical computations reported in \cite{frisch2002reconstruction} and \cite{brenier2003} 
were performed by restricting to a large spatially periodic box. 

No matter how truncation is performed, the use of optimal transport theory produces a map
$\tilde T_t$ that is the gradient of a convex function. Such a map $\tilde T_t$
necessarily is  a {\em monotone map}, meaning
\begin{equation}\label{e:Ttmonotone}
   (\tilde T_t(y)-\tilde T_t(\hat y) )\cdot(y-\hat y)\ge 0  \quad\text{for all $y$ and $\hat y$}.
\end{equation}

This condition does not seem excessively restrictive, but we can show that for the 
``non-sticky'' examples of the previous subsection, it is {\em impossible} for a monotone map 
to provide a correct reconstruction a.e.~in a neighborhood of the triple point at $x=t\tripl$.

\begin{prop}
    Assume $\tripl\notin\Triang=\co\{v_1,v_2,v_3\}$. 
    Let $t>0$, and assume $\tilde T_t$ is defined a.e.~in an open domain 
    that contains the (closed) triangle $t(\tripl-\Triang)$.
    Assume that \eqref{e:agree} holds when restricted to sets in some neighborhood of $t\tripl$, and assume
     $\tilde T_t$ provides a correct reconstruction for a.e.~$x$ in this neighborhood.
    Then $\tilde T_t$ cannot be monotone in any neighborhood of the triangle $t(\tripl-\Triang)$.
\end{prop}

\begin{proof}
  Suppose instead that $\tilde T_t$ is monotone in some convex open set $\Omega_0$
  containing $t(\tripl-\Triang)$. We claim that  necessarily 
   \begin{equation}
   \tilde T_t(y)=t\tripl \quad\text{for almost all $y\in t(\tripl-\Triang)$.} 
   \end{equation}
   This means that for the pushforward measure $\tilde\mam_t=(\tilde T_t)_\sharp\leb$,
   necessarily 
   \begin{equation}
       \tilde\mam_t(\{t\tripl\}) \ge |t(\tripl-\Triang)|> 0 \ne 0 = \rho_t(\{t\tripl\}),
   \end{equation}
   contradicting the assumption that \eqref{e:agree} holds when restricted to sets
   contained in some neighborhood of $t\tripl$.

   Recall that the measure $\rho_t$ concentrates mass in the set $\calS_t$
   consisting of the three mass filaments together with the triple point:
   \[
   \calS_t = \{t\tripl\}\cup R^t_{12} \cup R^t_{23} \cup R^t_{31}.
   \]
   The complement of $\calS_t$ is the union $A^t=A^t_1\cup A^t_2\cup A^t_3$
   of three open sectors with $t\tripl$ as common boundary point.
   For any $x\in A^t_j$, if $\tilde T_t$ correctly reconstructs at $x$ then 
   \[
   \tilde T_t\inv(x) = \{y\} \quad \text{where} \quad y = x-tv_j\,,
   \]
   and the point $y\in A^0_j$.  

   Observe that as $x\to t\tripl$ inside $A_j$ the point
   $y=x-tv_j$ converges to the point $y_j:=t(\tripl-v_j)$, one of the vertices of the triangle $t(\tripl-\Triang)$.
   Since $\tilde T_t$ is defined a.e.~in a neighborhood of $t(\tripl-\Triang)$,
   we may find a sequence of points $x^k_j\to t\tripl$ in $A^t_j$
   such that $\tilde T_t$ is defined at the points $y^k_j = x^k_j - tv_j$, which converge to $y_j$.

   Now, fix $y$ to be a point in the interior of $t(\tripl-\Triang)$ at which $x=\tilde T_t(y)$ is defined.
   By monotonicity, we have 
   \[
    (x-x^k_j)\cdot(y-y^k_j)\ge0  \quad\text{for all $k$}
   \]
   Taking $k\to\infty$ we infer $(x-t\tripl)\cdot(y-y_j)\ge0$, and this holds for each $j=1,2,3$.
   Since any point $z\in t(\tripl-\Triang)$ can be written as a convex combination of the $y_j$ ($j=1,2,3$),
   it follows that $(x-t\tripl)\cdot(y-z)\ge0$ for all such $z$.  But this includes all points $z$
   in a full neighborhood of $y$. Therefore necessarily $x=\tilde T_t(y)=t\tripl$.
\end{proof}

\begin{remark}
    From this argument, it is clear that the obstruction to monotone reconstruction is 
    local in a neighborhood of the intersection of mass filaments of 
    a fairly generic type. This has nothing to do with the possibility that optimal transport 
    maps may be affected by boundary conditions or methods of truncation.
\end{remark}
\ifspace \vfil\pagebreak \fi

\section*{Acknowledgements}
This material is based upon work supported by the National Science Foundation
under grants DMS 2106988 (JGL) and 2106534 (RLP).
JGL is grateful to the Simons Laufer Mathematical Sciences Institute (SLMath)
for support in connection with its Fall 2025 program on Kinetic Theory.

\ifspace
\singlespacing
\fi

\bibliographystyle{siam}
\bibliography{adhesion.bib}

\appendix

\ifspace \onehalfspacing \fi

\section{Properties of convex and semi-concave functions}\label{a:convex}

In this section we note some general properties associated
with convexity and semi-concavity that are used above.
A function $f:\R^d\to (-\infty,\infty]$ is called {\em coercive}
if $f(z)/|z|\to\infty$ as $|z|\to\infty$.

\begin{prop}\label{p:fprop}
    Suppose $f:\R^d\to\R$ is continuous and coercive.
    \begin{itemize}
        \item[(i)]  If $\fss$ is strictly convex at $y$, then $\fss(y)=f(y)$.
        \item[(ii)]  If $f$ is semi-concave, then $f^*$ is strictly convex and $\fss$ is $C^1$.
        \item[(iii)]  If $f$ is semi-concave and $\fss(y)=f(y)$, then
        $f$ is differentiable at $y$ and $\nabla f(y)=\nabla\fss(y)$.
        \item[(iv)]  If $f$ is semi-concave, and $\fss$ is strictly convex at $y$,
        then $f^*$ is differentiable at $x = \nabla f(y)$, with $\nabla f^*(x)=y$.
        Moreover, if $x=\nabla\fss(\hat y)$ then $\hat y=y$.
    \end{itemize}
\end{prop}

\begin{proof}
     (i)  Suppose not, so that $\delta_1= f(y)-\fss(y)>0$.
    Then $x\in\D\fss(y)$ exists such that with 
    \[
    g(z) = \fss(z) - \eta(z), \qquad \eta(z)=\fss(y)+x\cdot(z-y), 
    \]
    $g$ has strict minimum at $z=y$ with $g(y)=0$. 
    Since $\eta(y)+\delta_1=f(y)$,
    for some $r>0$ we have $\eta(z)+\frac12\delta_1<f(z)$ whenever $|z-y|<r$.
    Now since $g$ is convex,
    \[
    0<\delta_2 :=  \min_{|z-y|=r} g(z)= \min_{|z-y|\ge r} g(z).
    \]
    With $\delta=\min(\delta_2,\frac12\delta_1)$, it follows that for all $z\in\R^d$,
    \[
    h(z) = \max(\fss(z),\eta(z)+\delta) \le f(z).
    \]
    But $h$ is convex and this contradicts the fact that $\fss$ is the largest
    convex function below $f$.

(ii)  Supposing $f$ is semi-concave, we claim $f^*$ is strictly convex.
The function $\hat f(y)=f(y)-\frac12\lambda|y|^2$
is concave for some $\lambda>0$. Then $\hat f$ is the infimum of some family 
of affine functions $\{p_\alpha\cdot y+h_\alpha\}_{\alpha}$.  
Taking the Legendre transform we may interchange suprema to find
\begin{align*}
f^*(x) &=  \sup_\alpha \sup_y 
\Bigl( (x - p_\alpha)\cdot y - h_\alpha - \tfrac12\lambda|y|^2 \Bigr) 
= \sup_\alpha \frac{|x-p_\alpha|^2}{2\lambda} - h_\alpha\,.
\end{align*}
Then $x\mapsto f^*(x)-\tfrac1{2\lambda}|x|^2$ is convex, as the sup 
of a family of affine functions.

By \cite[Cor.~25.5.1]{Rockafellar}, it suffices to prove that the subgradient 
$\D\fss(y)$ is a singleton for every $y\in \R^d$.  But the Young identity
for each $x\in \D\fss(y)$ states 
\[
\fss(y) = \sup_z z\cdot y-f^*(z) = x\cdot y-f^*(x) \,.
\]
Because $f^*$ is strictly convex the maximizer $z=x$ is unique. Hence $\fss$ is $C^1$.

(iii) Suppose $\fss(y)=f(y)$. If $f$ is semi-concave,
then for some parabolic function $P$, we have $\fss(z)\le f(z)\le P(z)$ for all $z$,
with equality at $z=y$.
Using the fact from (ii) that $\fss$ is $C^1$,
we deduce $f$ is differentiable at $y$ with $x=\nabla\fss(y)=\nabla f(y)$. 

(iv) By parts (i)-(iii) we have
\[
f^*(x) \ge x\cdot z - \fss(z),
\]
and equality holds {\em only} at $z=y$ by strict convexity. 
Recalling that $z\in \D f^*(x)$ if and only if equality holds, we infer $\D f^*(x)$
is the singleton set $\{y\}$. Therefore by
\cite[Thm.~25.1]{Rockafellar}, $\nabla f^*(x)$ exists and equals $y$.
Moreover, if $x=\nabla f(\hat y)$ then necessarily $\hat y\in\D f^*(x)$
so $\hat y=y$.
\end{proof}

\section{Derivation of the adhesion model}\label{a:derivation}
\nwc{\bmr}{{\bm r}}
\nwc{\bmU}{{\bm U}}
\nwc{\bmu}{{\bm u}}
\nwc{\bmv}{{\bm v}}
\nwc{\bmq}{{\bm q}}
\nwc{\bmx}{{\bm x}}

For the convenience of readers,
here we sketch a derivation of the adhesion model, mostly following the presentation
in Appendix A of Brenier \etal\,\cite{brenier2003} with additional background
and some clarifications.
The system that Brenier \etal\ start with consists of a set of Euler-Poisson equations
for perturbations of an Einstein-de Sitter universe, in which the density of matter
is a function only of a time-like variable, the cosmological constant vanishes 
and the metric is spatially flat.
As discussed by Peebles~\cite{peebles1980}, in regions of space-time having (spatial) distances
small compared to $10^{28}$ cm (about $10^{10}$ light years), 
coordinates can be chosen such that a Newtonian approximation 
is valid for the gravitational potential $\phi_g$.  
The dynamics of the mass density $\varrho(\bmr,t)$
of (pressureless) cold dark matter is then given by an Euler-Poisson system as
\begin{align}
\label{e:EP1}
& \D_t\varrho + \nabla_r\cdot(\varrho\bmU) = 0 \,,
\\ 
\label{e:EP2}
& \D_t \bmU + (\bmU\cdot\nabla_r)\bmU = -\nabla_r\phi_g \,,
\\ \label{e:EP3}
& \Delta_r \phi_g = 4\pi G\varrho\,,
\end{align}
where $\bmU(\bmr,t)$ is the fluid velocity and $G$ is the gravitational constant.

The steps to get the adhesion model will be: (i) describe the Einstein-de Sitter
model of a homogeneous universe;
(ii) in the Newtonian approximation, 
 describe equations for scaled velocity and density perturbations as functions 
of co-moving spatial coordinates and the time $t$; 
(iii) describe the Zeldovich approximation which yields free-streaming, ``pressureless''
flow in these variables; and (iv) introduce a vanishing artificial viscous limit to 
prevent multi-streaming flows after the formation of singularities.

\subsection{A homogeneous universe}
In their 1932 joint paper \cite{einstein1979relation},
Einstein and de Sitter start from the field equations for general relativity in the form 
\begin{equation}
    R_{\alpha\beta}-\frac12 R g_{\alpha\beta} = \kappa T_{\alpha\beta}
\end{equation}
where $\kappa = 8\pi G/c^2$ is Einstein's constant.
For the Einstein-de Sitter universe with density $\bar\varrho(t)$ and metric 
\[
ds^2 = c^2 dt^2 - a(t)^2(dx_1^2+dx_2^2+dx_3^2) \,,
\]
for the $(\alpha,\beta)=(0,0)$ component we find
$T_{00} = \bar\varrho c^2$ and 
\[
R_{00} = -3 \frac{\ddot a}{a} \,,\quad R = -\frac{6}{c^2}\left(\frac{\ddot a}{a}+ \frac{\dot a^2}{a^2}\right),
\]
whence the corresponding component of the field equation  reduces to 
\begin{equation}\label{E00}
H(t)^2 = \frac{8\pi G\bar\varrho}3 \,,
\end{equation}
where $H(t)$ is the 
{\em Hubble parameter}, expressed in terms of the 
{\em expansion scale factor} $a(t)$ as 
\[
H(t) = \frac{\dot a(t)}{a(t)}.
\]
(Equation~\eqref{E00} is essentially the main result of Einstein and de Sitter.)
The tensor entries for $(\alpha,\beta)=(j,j)$ with $j=1,2,3$ are
\[
R_{jj} = \frac{a\ddot a+2\dot a^2}{c^2} \,, \quad T_{jj}=0,
\]
whence the field equation reduces to 
\begin{equation}\label{Eii}
2 a \ddot a + \dot a^2 = 0.
\end{equation}
We solve this equation with initial data at the present epoch $t_0$
chosen as
\[
a(t_0)=1, \quad \dot a(t_0) = H_0 >0, 
\]
which makes the metric locally approximated by the Minkowski metric.
Dividing \eqref{Eii} by $a\dot a$ we find $a \dot a^2  = H_0^2$, 
from which it is easy to deduce that
\[
\frac23(a^{3/2}-1) = H_0(t-t_0).
\]
Requiring that $a(t)\to0$ as $t\to0$ (meaning $t=0$ is consistent
with the big bang), 
we obtain 
\begin{equation}
H_0 = \frac{2}{3t_0}, 
\qquad a(t) = \left(\frac t{t_0}\right)^{2/3}
\end{equation}
Using \eqref{Eii} and differentiating \eqref{E00} yields
\[
\frac{8\pi G}3 \frac{d}{dt} (\bar\varrho a^3) = \dot a( 2 a \ddot a + \dot a^2)=0 \,,
\]
hence, with $\bar\varrho_0=\bar\varrho(t_0)$ and using \eqref{E00},
\begin{equation}\label{eq:rhoa3.1}
\bar\varrho a^3 = \bar\varrho_0 =  \frac{3H_0^2}{8\pi G} = \frac1{6\pi G t_0^2}.
\end{equation}

We note that current best estimates of the Hubble parameter say that $H_0\approx 72$ km/s/Mpc,
which yields $t_0\approx 9.07\cdot 10^9$ years. This is substantially less than
the currently accepted value of about $13.8\cdot 10^9$ years.

\subsection{Newtonian approximation}
The Einstein-de Sitter calculations above are consistent with the 
equations \eqref{e:EP1}--\eqref{e:EP3} that arise in the Newtonian approximation.
To check this consistency, we note that in the homogeneous universe, 
the fields take the following form,
in terms of the Hubble parameter $H(t)$:
\begin{equation}
\varrho(\bmr,t)=\bar\varrho(t), \quad
\bmU(\bmr,t) = H(t)\bmr . \quad
\end{equation}
The continuity equation \eqref{e:EP1} then says
\[
\D_t \bar\varrho = - 3H(t)\bar\varrho = -3\frac{\dot a}{a}\bar\varrho,
\]
and this follows from \eqref{eq:rhoa3.1}. The Poisson equation \eqref{e:EP3}
says that 
\[
\Delta_r\bar\phi_g = 4\pi G\bar\varrho = \frac32 \left(\frac{\dot a}{a}\right)^2
= -3\frac{\ddot a}{a}\,,
\]
which is satisfied by 
\[
\bar\phi_g = -\frac12 \frac{\ddot a}{a}|\bmr|^2.
\]
Finally, a fluid particle satisfying $\dot \bmr = \bmU(\bmr,t) = H(t)\bmr$
satisfies $\bmr = a(t)\bmx$ in terms of the {\em co-moving} coordinate $\bmx$
which remains constant.  
Then the momentum equation \eqref{e:EP2} reduces to say simply that 
\begin{equation}\label{e:barphig}
 \frac{\ddot a}{a}\bmr = -\nabla_r\bar\phi_g 
\,,
\end{equation}
which evidently holds.

\subsection{Dynamics of perturbations in co-moving coordinates}

Following the presentation of Brenier \& Frisch {\it et al.}~\cite{brenier2003},
we write perturbations of the Einstein-de Sitter solution of the Euler-Poisson
equations~\eqref{e:EP1}--\eqref{e:EP3}
as functions of $\tau=a(t)$ and $\bmx$, expressed in the form
\[
\varrho = \bar\varrho \,{\rho}, \quad
\bmU = \frac{\dot a}{a}\bmr + a \dot a \,{\bmv} ,
\quad \phi_g = -\frac{\ddot a}{2a}|\bmr|^2+ 4\pi G\varrho_0 \,{\varphi_g}\,.
\]
The velocity perturbation $a\dot a\,\bmv$ is known as the peculiar velocity
in the cosmology literature. 
After the change of variables, equations \eqref{e:EP1}--\eqref{e:EP3}
become equivalent to the following form of the Euler-Poisson system
which served as the starting point for \cite{brenier2003}:
\begin{align}
\label{e:EP21}
&\D_\tau \rho + \nabla_x\cdot(\rho\bmv) = 0\,,
\\
\label{e:EP22}
&\D_\tau \bmv + \bmv \cdot \nabla_x\bmv = 
{-\frac{3}{2\tau}(\bmv+\nabla_x\varphi_g)}\,,
\\
\label{e:EP23}
&\Delta_x \varphi_g = {\frac{\rho-1}{\tau}}.
\end{align}


\subsection{Zeldovich approximation and adhesion model}

A simple way to understand the Zeldovich approximation is that it provides
an {\em approximate solution for small $\tau$}
to the Euler-Poisson system \eqref{e:EP21}--\eqref{e:EP23}.
This approximation is generated by Lagrangian extrapolation from an initial velocity
that is naturally given by a potential to ensure that the right-hand side of \eqref{e:EP22}
does not blow up as $\tau\to0$.

Thus we take as {\em ansatz} that the Lagrangian flow map for the system is given by 
\begin{equation}\label{dA:X}
X(y,\tau) = y + \tau \nabla\vp(y) \,.
\end{equation}
The Eulerian velocity field is taken as the gradient $\bmv(x,\tau)=\nabla_x u(x,\tau)$ of the solution
of the Hamilton-Jacobi equation
\begin{equation}\label{eA:HJ}
\D_\tau u + \tfrac12|\nabla_x u|^2 = 0, \quad u(x,0)=\vp(x).
\end{equation}
The solution (presumed smooth) satisfies $\nabla_x u(x,\tau)=\nabla_y \vp(y)$ for $x=X(y,\tau)$, 
consistent with the vanishing of the second material derivative: 
$\D_\tau \bmv + \bmv\cdot\nabla \bmv = 0$.

We take as a further {\em ansatz} that $\vp_g(x,\tau) = -u(x,\tau)$, with the result that \eqref{e:EP22}
holds exactly. For the density, we also ensure \eqref{e:EP21} holds exactly by taking $\rho$ so that
\begin{align}
    \rho(X(y,\tau),\tau) &= \det \nabla_y X(y,\tau)\inv  = \det(I+\tau\nabla^2_y\vp(y))\inv  
    \nonumber\\ &= 1-\tau\Delta_y\vp(y) + O(\tau^2).
\end{align}
Then the Poisson equation \eqref{e:EP23} does not hold exactly, but it does hold approximately
for $\tau$ small:
Using that $\nabla_x=\nabla_y+O(\tau)$ we find that since $\nabla_x\vp_g(x)=-\nabla_y\vp(y)$ 
with $x=X(y,\tau)$,
\begin{equation}
    \Delta_x \vp_g(x)  = -\Delta_y \vp(y) + O(\tau) = \frac{\rho-1}\tau  + O(\tau).
\end{equation}

The Zeldovich approximation takes the free-streaming ansatz \eqref{dA:X} to be valid
for longer times, up to and beyond the time that particle paths collide and multi-streaming 
regions form.
Finally, to obtain the adhesion model one adds a viscosity term to the Hamilton-Jacobi equation
in \eqref{eA:HJ} as described in the introduction. As we have shown, this indeed eliminates
multi-streaming regions and results in particle paths sticking together in the limit of vanishing viscosity.
\end{document}